\documentclass[11pt,reqno]{amsart}
\usepackage{color}
\usepackage{amsmath,amssymb,amsthm,amsfonts}
\usepackage{hyperref}
\usepackage{color}
\usepackage{setspace}
\usepackage{placeins}
\makeatletter

\newcommand{\Rmnum}[1]{\expandafter\@slowromancap\romannumeral #1@}
\makeatother

\newtheorem{thm}{Theorem}[section]

\newtheorem{lemma}[thm]{Lemma}
\newtheorem{remark}{Remark}[section]
\newtheorem{theorem}[thm]{Theorem}
\newtheorem{proposition}[thm]{Proposition}

\usepackage[numbers,sort&compress]{natbib}
\usepackage{graphicx}
\usepackage{caption}
\usepackage{subcaption}
\usepackage{tikz}
\usepackage{multirow}

\usepackage[mathlines]{lineno}

\begin{document}

\author{Huaizhi Cao}
\address{School of Mathematics, South China University of Technology,
Guangzhou 510640, China}
\email{mahzcao@163.com}

\author{Jiawei Chu$^{*}$}
\address{School of Mathematics, South China University of Technology,
Guangzhou 510640, China}
\email{majwchu@163.com}

\author{Hai-Yang Jin}
\address{School of Mathematics, South China University of Technology,
Guangzhou 510640, China}
\email{mahyjin@scut.edu.cn}

\thanks{$^*$Corresponding author: Jiawei Chu}

\title[Periodic dynamics in a forager-exploiter system]{Periodic dynamics in a forager-exploiter system under homogeneous and heterogeneous resource environments}
\begin{abstract}
We investigate time-periodic dynamics in a forager-exploiter system with a taxis cascade under both homogeneous and heterogeneous resource environments. The model describes the interactions among foragers, exploiters, and environmental resources, where foragers move toward higher resource densities while exploiters aggregate toward regions with higher forager densities. Our results show that different resource renewal mechanisms shape periodic dynamics in fundamentally different ways. Precisely, for time-periodic resource renewal rates, we establish the existence of positive time-periodic solutions for any positive renewal rate and further prove their global stability under suitable conditions on the parameters. In contrast, for homogeneous environments, only large resource renewal rates can destabilize the constant steady state through the Hopf bifurcation, thereby generating non-constant time-periodic solutions.

Interestingly, for spatially heterogeneous and temporally homogeneous environments, our numerical simulations indicate that spatial heterogeneity exhibits opposing effects on periodic dynamics depending on total resource availability. When resources are sufficiently abundant, spatial heterogeneity tends to suppress the emergence of temporal oscillations, whereas when resources are relatively scarce, it may instead promote oscillatory behaviors, and sufficiently concentrated local resource supplies can trigger local or even global temporal oscillations. These findings reveal a delicate interplay among resource renewal mechanisms, resource availability, and spatial heterogeneity in shaping dynamics behavior.
\end{abstract}



\maketitle

\numberwithin{equation}{section}

\maketitle

\numberwithin{equation}{section}

\section{Introduction}
Social interactions among populations can generate complicated spatiotemporal dynamics, such as aggregation phenomena and oscillatory spatiotemporal patterns. Exploring the dynamics behavior and understanding formation mechanisms of such complex patterns have become active topics in the interdisciplinary field of biology, behavioral science and mathematics \cite{1,5,6,9,8,12}. A typical example of such interaction is the aggregative foraging flocks formed by shearwaters following kittiwakes in Alaska coastal waters \cite{7}. In this ecological system, kittiwakes (as foragers) search for regions with higher resource densities, whereas shearwaters (as exploiters) move toward regions with higher forager densities.

To describe such interaction, a taxis cascade reaction-diffusion model was developed in \cite{9} to depict the coupled dynamics of forager populations, exploiter populations, and resources, which reads as below:
\begin{equation}\label{system}
\begin{cases}
u_t=\Delta u-\chi_1\nabla\cdot(u\nabla w),&x\in\Omega,~t>0,\\
v_t=\Delta v-\chi_2\nabla\cdot(v\nabla u),&x\in\Omega,~t>0,\\
w_t=d\Delta w-\lambda(u+v)w-\mu w+r(x,t),&x\in\Omega,~t>0,\\
\nabla u\cdot\nu=\nabla v\cdot\nu=\nabla w\cdot \nu=0,&x\in\partial\Omega,~t>0,\\
(u,v, w)(x,0)=(u_0, v_0, w_0)(x), &x\in\Omega,
\end{cases}
\end{equation}
where the functions $u=u(x,t),v=v(x,t),w=w(x,t)$ denote the densities of the forager population, exploiter population and resource at space $x$ and time $t$, respectively. Here $\Omega\subset\mathbb{R}^n(n\geq 1)$ is a bounded domain with smooth boundary, and the homogeneous Neumann boundary conditions indicate that no populations across the boundary. The parameters $\chi_1,\chi_2,d$, $\lambda$ and $\mu$ are all positive constants. The taxis terms $-\chi_1\nabla\cdot(u\nabla w)$ and $-\chi_2\nabla\cdot (v\nabla u)$ represent the directed movement of foragers toward areas with higher resource density and exploiters toward areas with higher forager density, respectively. The nonnegative function $r(x,t)\geq 0$ denotes the renewal rate from an external repository.  

When there is no  exploiter population (i.e., $v\equiv0$),  \eqref{system} is reduced to a single-species chemotaxis-consumption model 
\begin{equation}\label{system1}
\begin{cases}
u_t=\Delta u-\chi_1\nabla\cdot(u\nabla w),&x\in\Omega,~t>0,\\
w_t=d\Delta w-\lambda uw-\mu w+r(x,t),&x\in\Omega,~t>0,\\
\nabla u\cdot\nu=\nabla w\cdot \nu=0,&x\in\partial\Omega,~t>0,\\
(u, w)(x,0)=(u_0, w_0)(x), &x\in\Omega.
\end{cases}
\end{equation}
For the system \eqref{system1} with $\mu=r(x,t)\equiv 0$, Tao and Winkler \cite{10} established the existence of global classical solutions in two dimensions and global weak solutions in three-dimensional settings. They further proved that all solutions asymptotically converge to the constant equilibrium. The analysis relies crucially on an entropy-type energy identity of the form
\begin{align*}
&\frac{d}{dt}\left(\frac{1}{2}\int_\Omega \frac{|\nabla w|^2}{w}+\frac{\lambda}{\chi_1}\int_\Omega u\ln u\right)+\frac{\mu}{2}\int_\Omega\frac{|\nabla w|^2}{w}
+d\int_\Omega w\left|D^2\ln w\right|^2
+\frac{\lambda }{\chi_1}\int_\Omega\frac{|\nabla u|^2}{u}+\frac{\lambda}{2}\int_\Omega \frac{u|\nabla w|^2}{w}\\
&=\frac{d}{2}\int_{\partial\Omega}\frac{1}{w}\frac{\partial |\nabla w|^2}{\partial\nu}\,dS.
\end{align*}
This energy structure has been widely adopted for various variants of \eqref{system1}, including prey-taxis systems \cite{JinWang-JDE-2017, Winkler-JDE-2017, WuShiWu-JDE-2016} and chemotaxis–fluid models \cite{Winkler-ARMA-2014}, among others. However, the aforementioned entropy-energy functional is only sufficient to establish global classical solvability in low spatial dimensions ($1\leq n\leq 2$). In higher dimensions ($n\geq 3$), global classical solvability remains unknown for the fully parabolic system \eqref{system1}, in contrast to the parabolic-elliptic counterpart, for which global classical solvability has been obtained by Tao and Winkler \cite{TaoWinkler-JDE-2019-PE}. In their work, the authors established global classical solutions in arbitrary dimensions and proved the global asymptotic stability of the constant steady state under suitable decay conditions imposed on $r(x,t)$.

Compared with the system \eqref{system1}, the system \eqref{system} incorporates coupled cascading taxis mechanisms, which may endow it with richer spatiotemporal dynamics. From a mathematical perspective, however, the energy structure applicable to \eqref{system1} breaks down for \eqref{system} owing to the additional exploiter coupling. This structural difference substantially increases the difficulty in analyzing global boundedness, long-time behaviors and pattern formation. Accordingly, most existing results concerning \eqref{system} are restricted to the global existence and boundedness of solutions, as well as the global stability of constant steady states in a one-dimensional setting or under appropriate smallness assumptions. In particular, when the initial data $u_0,v_0, w_0$ satisfy the regularity conditions
\begin{equation}\label{IC}
   u_0, w_0\in W^{2,\infty}(\Omega), \ v_0\in W^{1,\infty}(\Omega),\ \ u_0,v_0, w_0 \geq 0,\not\equiv 0 \ \text{in}\ \overline{\Omega},
\end{equation}
Tao and Winkler \cite{11} first proved the global boundedness of classical solutions to \eqref{system} with positive constant $r(x,t)\equiv r$ in one dimension ($n=1$), and further proved that the solutions asymptotically converge to the constant steady state under suitable smallness conditions on either total mass of foragers or that of exploiters, or on the constant $r$. Later, \cite[Theorems 1.1–1.2]{13} extended these boundedness results to higher-dimensional setting ($n\geq 2$), where certain smallness conditions were imposed on the initial data $u_0, w_0$, and the resource term $r(x,t)$, or the taxis sensitivity coefficients $\chi_1,\chi_2$. In addition, \cite{16} established the global existence of generalized solutions to \eqref{system} under an explicit structural relation between $w_0$ and $r(x,t)$, and analyzed the long-time dynamics when $r(x,t)$ possesses suitable temporal decay properties. For more related progress concerning \eqref{system}, we refer to \cite{2,3,4,13,14,15,17} and the references therein.

All the aforementioned theoretical results demonstrate that solutions stabilize toward constant equilibria and fail to generate nontrivial spatial or temporal patterns under existing smallness or decay assumptions. In sharp contrast, numerical simulations in \cite{9} reveal that the sufficiently large constant resource supply $r(x,t)\equiv r>0$  can trigger oscillatory spatiotemporal patterns.
Despite such numerical evidence, a rigorous theoretical framework concerning the existence and stability of non-constant positive time-periodic solutions is still absent. Moreover, most available literature focuses on constant resource renewal rates, whereas more realistic temporally or spatially heterogeneous resource environments have received little attention. In practical ecological settings, resource supplies are generally modulated by seasonal variations, climatic cycles and environmental fluctuations, thereby exhibiting periodic temporal variability. Such heterogeneous resource renewal mechanisms may significantly affect the populations movement and aggregation, potentially leading to complicated spatiotemporal dynamics.

Motivated by these questions, we investigate the existence and dynamics of non-constant positive time-periodic solutions to system \eqref{system} under both homogeneous and heterogeneous environments. Our main goals are summarized as follows:
\begin{itemize}
\item Establish the existence of positive periodic solutions under both temporally periodic and homogeneous resource renewal rate;
\item Investigate the uniqueness and stability of such positive periodic solutions;
\item Reveal how different resource environments affect the dynamics and pattern formation mechanisms.
\end{itemize}

Throughout this paper, we let $\Omega\subset\mathbb{R}^n$ ($n\geq 1$) be a bounded domain with a smooth boundary, and make the following basic hypothesis:
\begin{itemize}
\item[(H)] The function $r(x,t)\in C^{\alpha_0,\frac{\alpha_0}{2}}(\overline{\Omega}\times[0,\infty))$ satisfies $r(x,t)\geq,\not\equiv0$ and $r(x,t)=r(x,t+T)$, where the constants $\alpha_0 \in (0,1)$ and $T>0$.
\end{itemize}
For convenience, we denote
$$0<r_*:=\sup_{\overline{\Omega}\times[0,\infty)} r(x,t).$$
Moreover, the initial data $(u_0,v_0,w_0)$ are assumed to satisfy  \eqref{IC}.

\vspace{2mm}
The remaining part of this paper is organized as follows. Section \ref{MR} will declare main results. In Section \ref{timeperiodic}, we establish the existence, uniqueness and the global stability  of non-constant positive periodic solution in time-periodic environment. In Section \ref{HP}, we study Hopf bifurcations, which reveals the existence of time-periodic solutions in homogeneous environment. Finally, in Section \ref{NS}, some numerical simulations are conducted to extend our theoretical results.

\section{Main results}\label{MR}
This paper focuses on the existence and dynamics of non-constant positive time-periodic solution in both homogeneous and heterogeneous environments. 
\subsection{Time-periodic heterogeneous environments} 
Under time-periodic heterogeneous environments, periodic oscillations may be induced directly by the periodic variation of the external resource supply $r(x,t)$. A fundamental question is therefore whether the system \eqref{system} admits positive time-periodic solutions inheriting the temporal periodicity of the environment. To answer this question, we further suppose that
\begin{itemize}
\item[(H0)] the hypothesis $(\operatorname{H})$ holds,  and $r(x,t)$ is spatially and temporally heterogeneous. 
\end{itemize}
\vspace{1.5mm}

We first consider whether there exists a positive $T$-periodic solution to \eqref{system} with time-periodic $r(x,t)$, which satisfies
\begin{equation}\label{periodic}
\begin{cases}
U_t=\Delta U-\chi_1\nabla \cdot( U\nabla W), &x\in\Omega, \ t>0,\\
 V_t=\Delta  V-\chi_2\nabla \cdot(V\nabla  U),&x\in\Omega, \ t>0,\\
 W_t=d\Delta  W-\lambda( U+ V) W-\mu  W+r(x,t),&x\in\Omega, \ t>0,\\
\nabla { U}\cdot\nu=\nabla{ V}\cdot \nu=\nabla{ W}\cdot{\nu}=0,& x\in\partial\Omega,~t>0,\\
( U, V ,  W)(x,t) =( U, V ,  W)(x,t+T), &x\in\Omega,t\geq 0.
\end{cases}    
\end{equation}

The following theorem gives a positive answer.
\begin{theorem}\label{GBSp} 
Let $\operatorname{(H0)}$ hold and $r(x,t)>0$. Then the system \eqref{periodic} admits at least one non-constant positive $T$-periodic solution $( U, V, W) \in \big[C^{2+\alpha,1+\frac{\alpha}{2}}(\overline{\Omega}\times[0,\infty))\big]^3$  satisfying
\[
\|( U, V, W)(\cdot,t)\|_{C^{2+\alpha,1+\frac{\alpha}{2}}(\overline{\Omega}\times[0,\infty))}\le C,
\]
where constants $\alpha\in(0,1)$, and $C>0$ is independent of $t$.
\end{theorem}

More importantly, one is interested in the asymptotic stability of such a $T$-periodic solution, which ensures that the periodic patterns can sustain over time. The following theorem shows a stability result under specific circumstances.
\begin{theorem}\label{GBSps}
Let $( U, V, W)$ be a $T$-periodic solution obtained in Theorem \ref{GBSp}, and $ (u,v,w) $ be a global classical solution of the system \eqref{system} with
\begin{equation}\label{uvb*}
\|u(\cdot,t)\|_{L^\infty}+\|v(\cdot,t)\|_{L^\infty}\leq K_0:=K_0(r_*),
\end{equation}
and whose initial data $u_0,v_0$ fulfill
\begin{equation}\label{IC-A}
\int_\Omega u_0dx=\int_\Omega  U(x,t)dx, \quad \int_\Omega v_0dx=\int_\Omega  V(x,t)dx.
\end{equation}
Then there exists a constant \( r_0\in (0,1) \) such that when \( 0 < r_* \le r_0 \), the \( T \)-periodic solution \( ( U, V, W) \) is unique, and globally asymptotically stable in the following sense:
\[\|u(\cdot,t)- U(\cdot,t)\|_{L^\infty} + \|v(\cdot,t)- V(\cdot,t)\|_{L^\infty} + \|w(\cdot,t)- W(\cdot,t)\|_{L^\infty}\leq C_0e^{-\kappa t}, \forall t\geq2.
\]
Here $K_0(r_*)\geq 1$ is a constant increasing in $r_*$, $C_0>0$ and $\kappa>0$ are constants independent of $t$.
\end{theorem}

\begin{remark}

\em{By slightly modifying the proofs of \cite[Theorems 1.1-1.2]{13} and of \cite[Theorem 1.1]{11}, one can establish the existence and uniqueness of global classical solutions $(u,v,w)$ to \eqref{system} for  $n\geq 1$ and find a constant $K_0(r_*)\geq 1$ increasing in $r_*$ such that \eqref{uvb*} holds.
}
\end{remark}

\begin{remark}
\em{The condition \eqref{IC-A} is natural due to the conservation structure of the system \eqref{system}.}
\end{remark}
In fact, the existence of non-constant positive time-periodic solutions established above is derived using the topological degree theory. But, such technique does not directly apply when the resource renewal rate $r(x,t)$ is constant, as it cannot rule out the existence of positive constant steady states. We thus employ Hopf bifurcation theory to address this scenario.
\subsection{Homogeneous environments}
We now consider the case where the resource supply $r(x,t)$ is spatially and temporally constant. In contrast to heterogeneous environments, periodic oscillations in this setting, if they exist, must arise from intrinsic dynamical instability rather than external forcing. In this circumstance, $r(x,t)\equiv r$ ($r$ is a positive constant), then the system \eqref{system} admits a unique positive constant steady state $(\bar{u}_0,\bar{v}_0,w_c)$ with
$$w_c:=\frac{r}{\lambda(\bar{u}_0+\bar{v}_0)+\mu},$$
where the notation $\bar{f}:=\frac{1}{|\Omega|}\int_\Omega f dx$. To state our results clearly, we make some preparations. Let the sequence $\{\sigma_m\}_{m\geq0}: 0=\sigma_0<\sigma_1\leq \sigma_2\leq \sigma_3\dots $ (counting multiplicities, i.e., repeating each eigenvalue as many times as its multiplicity) denote the sequence of eigenvalues of $-\Delta$ under Neumann boundary condition. Then we define
\begin{align}
\label{rhm}
r_m^H&:=r^H(\sigma_m):= \frac{2[(d+1)\sigma_m+\lambda(\bar{u}_0+\bar{v}_0)+\mu]^2[\lambda(\bar{u}_0+\bar{v}_0)+\mu]}{\lambda \chi_1\bar{u}_0[\sigma_m(\chi_2\bar{v}_0-1-d)-\lambda(\bar{u}_0+\bar{v}_0)-\mu]},\\
\label{q*}
q_*&:=\frac{[2(d+1)+\chi_2\bar{v}_0-d-1][\lambda(\bar{u}_0+\bar{v}_0)+\mu]}{(d+1)(\chi_2\bar{v}_0-d-1)},\\
\label{q0}
q_0&:=\frac{\lambda(\bar{u}_0+\bar{v}_0)+\mu}{\chi_2\bar{v}_0-1-d}.\\ \notag
\end{align}
The following results establish the existence of positive time-periodic solutions to \eqref{system} in homogeneous environments.
\begin{theorem}[Hopf bifurcations]\label{LHP} 
Let $\chi_1,\chi_2,\lambda,\mu,d,\bar{v}_0,\bar{u}_0$ be fixed with $\chi_2\bar{v}_0>d+1$, $r_m^H$ be defined in \eqref{rhm}. Assume that for some $j\in\mathbb{Z}^+$, $\sigma_j$ is a simple eigenvalue of $-\Delta$ under Neumann boundary condition and the eigenvalue $\sigma_j$ satisfies $\sigma_j\geq q_*.$ Then the following results hold:
\begin{itemize}
\item [(1)] $(\bar{u}_0,\bar{v}_0,w_c)$ is linearly stable if  $r<\min\limits_{m_0\in Q_0}\{r_{m_0}^H\}$;
\item [(2)] $(\bar{u}_0,\bar{v}_0,w_c)$ is linearly unstable if $r>\min\limits_{m_0\in Q_0}\{r_{m_0}^H\}$; moreover, the system \eqref{system} undergoes a  Hopf bifurcation near $(\bar{u}_0,\bar{v}_0,w_c)$ at $r=r_j^H$ for any $j\in Q_*$.
\end{itemize}
Here, the sets
\begin{equation}\label{Q0*}
Q_0:=\big\{m_0\in\mathbb{Z}^+|\sigma_{m_0}>q_0\big\},\ \ Q_*:=\big\{j\in\mathbb{Z}^+|\sigma_{j}>q_*\big\}.
\end{equation}
with $q_0$ and $q_*$ being defined in \eqref{q0} and \eqref{q*}, respectively.
\end{theorem}

\begin{remark}
\em{We emphasize that the assumptions of Theorem \ref{LHP} are not vacuous. For example, set $\Omega=(0,l)$ for some constant $l>0$, then $\sigma_j=\big(\frac{j\pi}{l}\big)^2$ is a simple eigenvalue. We further let the spatial averages $\frac{1}{l}\int_0^l u_0(x)dx$, $\frac{1}{l}\int_0^l v_0(x)dx$ be fixed constants, then $q_*$ defined in \eqref{q*} is independent of $l$. For fixed parameters $\chi_1,\chi_2,\lambda,\mu,d,j$, we may select $l$ such that $(\frac{j\pi}{l})^2\geq q_*$, thereby confirming that the assumption $\sigma_j\geq q_*$ holds.}
\end{remark}

\begin{remark}
\em{Theorem \ref{LHP} identifies the exact Hopf bifurcation point. 
This result confirms numerical observations in \cite{9} on spatiotemporal patterns in one and two dimensions, and further addresses the open problem raised therein regarding the behavior of \eqref{system} in higher dimensions (i.e., $n\geq 3$). Moreover, Theorem \ref{LHP} complements the result in \cite{11}, which essentially shows that for fixed $\int_\Omega u_0dx$ and $\int_\Omega v_0dx$,  $(\bar{u}_0,\bar{v}_0,w_c)$ is globally asymptotically stable when $r$ is small.} 
\end{remark}

\begin{remark}
\em{ Regarding the stability of periodic solutions arising from $(\bar{u}_0,\bar{v}_0,w_c)$, we observe the following phenomena: for any $j<j_0$ with some $j_0\in Q_*$, when $r_j^H<r<r_{j_0}^H$, the corresponding periodic solution is globally attractive; however, bistability occurs as $r>r_{j_0}^H$; see  Subsection \ref{NRH} for the numerical evidence.
}
\end{remark}

\begin{remark}
\em{The conditions for admitting time-periodic solutions differ strikingly between heterogeneous and homogeneous environments. In the time-periodic environment, a time-periodic solution exists whenever $r(x,t) > 0$, whereas in the homogeneous case, a time-periodic solution can arise only when $r$ is large. Regarding stability, periodic solutions in the heterogeneity setting are stable under the smallness condition on $r(x,t)$. Thus, analyzing stability in the homogeneous case requires new analytical techniques and remains open.}
\end{remark}

In what follows, we shall abbreviate $\int_\Omega f dx$,  $\|f\|_{L^p(\Omega)}$ and $\int_0^T \|f\|_{L^p(\Omega)}ds$ as $\int_\Omega f$, $\|f\|_{L^p}$ and $\int_0^T \|f\|_{L^p}$, respectively.  The symbols $c_i$, $C_i,~M_i (i=1,2,3\cdots)$ are used to denote generic positive constants which are independent of $t$ and may vary in the context. 
Moreover, we set $Q_T:=\Omega\times(0,T)$ and $\overline{Q}_T:=\overline{\Omega}\times[0,T]$.

\section{Time-periodic environments: proof of Theorems \ref{GBSp}-\ref{GBSps}}\label{timeperiodic}
In this section, we will establish the existence and global stability of $T$-periodic solutions for \eqref{system} with $T$-periodic $r(x, t)$.
We first recall a technical lemma, whose proof follows directly from \cite{Amann, JIN-ZAMP-2017, JIN-PRSE-2020}.
\begin{lemma}\label{LT}
Let $a_1>0$ and $a_2>0$ be constants, and let $f(x,t)\in L_{\mathrm{loc}}^{p}([0,\infty);L^{p}(\Omega))$ be a $T$-periodic function with $1<p<\infty$. Then the following system
\begin{equation*}
\begin{cases}
z_t- a_1\Delta z + a_2 z=f(x,t),&x\in\Omega,t>0,\\
\nabla z \cdot\nu=0,\ \ &x\in\partial\Omega,t>0,\\
z(x,t)=z(x,t+T),\ \ &x\in{\Omega},t\geq0,
\end{cases} 
\end{equation*}
admits a unique strong $T$-periodic solution $z(x,t)\in W^{2,1}_{p}(Q_T)$ satisfying 
\begin{equation}\label{LT1}
\|z\|_{W_{p}^{2,1}(Q_T)}
\leq c_1\|f\|_{L^{p}(Q_T)}.
\end{equation}
Furthermore, if $f\in C^{\alpha,\frac{\alpha}{2}}(\overline{Q}_T)$, then it holds that
\begin{equation*}
\|z\|_{{C^{2+\alpha,1+\frac{\alpha}{2}}}(\overline{Q}_T)}\leq c_2\|f\|_{{C^{\alpha,\frac{\alpha}{2}}}(\overline{Q}_T)}.
\end{equation*}
\end{lemma}
The following lemma provides a key differential inequality which will be used to establish uniform bounds for solutions.
\begin{lemma}[\cite{JIN-PRSE-2020}]\label{BDSS}
Let $T>0, a>0, b\geq0$, and $f:\mathbb{R}^+\to[0,\infty)$ be absolutely continuous. Assume that nonnegative functions $f,h\in L_{\mathrm{loc}}^1([0,\infty))$ are $T$-periodic, and satisfy
\begin{equation*}
    f(t)-f(t_0)+a\int_{t_0}^t f^{1+b}(s)ds\leq\int_{t_0}^t h(s)ds, \ \ \ \text{for any}\ \ \ 0\leq t_0 < t,
\end{equation*}
and
\begin{equation*}
    \int_{0}^T h(s)ds\leq \beta.
\end{equation*}
Then it holds that
\begin{equation*}
    \displaystyle\sup_{t\in(0,T)}f(t)+a\int_0^T f^{1+b}(t)dt\leq \left(\frac{\beta}{aT}\right)^{\frac{1}{1+b}}+2\beta.
\end{equation*}
\end{lemma}

\subsection{Existence}
In this subsection, we shall prove the existence of non-constant positive time-periodic solutions to \eqref{system} in time-periodic environments in any dimension (i.e., $n\geq 1$) by using topological degree theory. 
To this end, we first consider the existence of nontrivial time-periodic solutions to the following approximating problem
\begin{equation}\label{periodic2}
\begin{cases}
 U_t=\Delta  U-\chi_1\nabla \cdot(  U\nabla W)+ \varepsilon  U^{n+1}(1- U), &x\in\Omega, \ t>0,\\
 V_t=\Delta  V-\chi_2\nabla \cdot(V\nabla  U)+\varepsilon  V^{n+1}(1- V),&x\in\Omega, \ t>0,\\
 W_t=d\Delta  W-\lambda( U+ V) W-\mu  W+r(x,t),&x\in\Omega, \ t>0,\\
\nabla U\cdot\nu=\nabla V \cdot\nu=\nabla W\cdot \nu=0, &x\in\partial\Omega, \ t>0, \\
 (U,V,W)(x,t)= (U,V,W)(x,t+T), &x\in\Omega, t\geq 0,
\end{cases}     
\end{equation}
where the constant $\varepsilon\in (0,1)$. 

 In fact, for any given nonnegative $T$-periodic functions $(\widetilde u, \widetilde v) \in [C^{\alpha,\frac{\alpha}{2}}(\overline{Q}_T)]^2=:\mathcal{X}$, we can use the result in \cite[Theorem 1.1]{Amann} to show  that the following linear periodic system
\begin{equation}\label{lw}
\begin{cases}
 W_t=d\Delta  W-\lambda(\widetilde u_++\widetilde v_+) W-\mu  W+r(x,t),&x\in\Omega, \ t>0,\\
\nabla W\cdot \nu=0, &x\in\partial\Omega, \ t>0, \\
 W(x,t)= W(x,t+T), &x\in\Omega,t\geq 0,
\end{cases}      
\end{equation}
admits a unique solution $ W\in C^{2+\alpha,\frac{2+\alpha}{2}}(\overline{Q}_T)$. Here, $\phi_+:=\max\{\phi,0\}$. Moreover, by maximum principle, we have
\begin{equation}\label{wbd}
0\leq  W\leq \frac{r_*}{\mu}.    
\end{equation}
For such given $(\widetilde u , \widetilde v)\in \mathcal{X}$ and obtained $W\in C^{2+\alpha,\frac{2+\alpha}{2}}(\overline{Q}_T)$,  we consider the following linearized system
\begin{equation}\label{lu}
\begin{cases}
 U_t-\Delta  U+\chi_1\nabla \cdot(  U\nabla W)+A_1 U=A_1\sigma\widetilde u_++\sigma \varepsilon \widetilde u_+^{n+1}(1- U), &x\in\Omega, \ t>0,\\
\nabla{U}\cdot \nu=0, &x\in\partial\Omega, \ t>0, \\
 U(x,t)= U(x,t+T), &x\in\Omega, t\geq 0,
\end{cases}        
\end{equation}
where constants $\sigma \in [0,1]$ and $A_1 > \|\chi_1 \Delta  W\|_{L^\infty} + 1$. Then using the result  in \cite[Theorem 1.1]{Amann} along with maximum principle implies that \eqref{lu} admits a unique $T$-periodic solution $U$ satisfying
\begin{equation}\label{UCP}
0\leq U \in C^{2+\alpha,1+\frac{\alpha}{2}}(\overline{Q}_T).
\end{equation}
Similarly, for above given $(\tilde u , \tilde v)\in \mathcal{X}$ and obtained $U\in C^{2+\alpha,\frac{2+\alpha}{2}}(\overline{Q}_T)$, we can choose $A_2 > \|\chi_2 \Delta  U\|_{L^\infty} + 1$ such that the following system 
\begin{equation}\label{lv}
\begin{cases}
 V_t-\Delta  V+\chi_2\nabla \cdot(V\nabla  U)+A_2 V=A_2\sigma\widetilde v_++\sigma\varepsilon \widetilde v_+^{n+1}(1- V),&x\in\Omega, \ t>0,\\
\nabla{V}\cdot \nu=0, &x\in\partial\Omega, \ t>0, \\
 V(x,t)= V(x,t+T), &x\in\Omega, t\geq 0,
\end{cases}        
\end{equation}
admits a unique solution $V$ satisfying
\begin{equation}\label{VCP}
0\leq V \in C^{2+\alpha,1+\frac{\alpha}{2}}(\overline{Q}_T).
\end{equation}
Therefore, we can define a map $\mathcal{F}:  \mathcal{X}\times [0,1] \to \mathcal{X}$ by
\begin{equation}\label{FO}
    \mathcal{F}((\widetilde u, \widetilde v),\sigma)=( U,  V).
\end{equation}
Then to establish the existence of non-constant positive $T$-periodic solutions to \eqref{periodic2}, it suffices to show that there exists non-constant $T$-periodic solution $U,V>0$ such that $\mathcal{F}(( U, V),1)=( U, V).$

First, we establish the properties of the operator $\mathcal{F}$ defined in \eqref{FO}. 
\begin{lemma}\label{C12}
 The operator $\mathcal{F}$ defined in \eqref{FO} is compact and continuous. Moreover, it holds that $\mathcal{F}( (\widetilde{U}, \widetilde{V}), 0 ) = (0, 0)$ for any $(\widetilde{U}, \widetilde{V}) \in \mathcal{X}$.
\end{lemma}
\begin{proof}
The compactness of the operator  $\mathcal{F}$ defined in \eqref{FO} directly follows from Sobolev embedding theorem \( C^{2+\alpha,1+\frac{\alpha}{2}}(\overline{Q}_T) \hookrightarrow C^{\alpha,\frac{\alpha}{2}}(\overline{Q}_T) \). And it is easy to check that $\mathcal{F}$  is continuous. Next, for \(\sigma = 0\), integrating $U$-equation and $V$-equation in \eqref{lu} and \eqref{lv} over \(\Omega \times (0,T)\), respectively, yields
\begin{equation*}
\int_0^T\int_{\Omega} U= \int_0^T\int_{\Omega} V  =0 ,
\end{equation*}
which along with the facts $U,V\geq0$ implies \(\mathcal{F}( (\widetilde{U}, \widetilde{V}), 0 ) = (0, 0)\) for any $(\widetilde{U}, \widetilde{V}) \in \mathcal{X}$. 
\end{proof}

We now establish {\it a priori} bound for the solution of the fixed point equation 
$$\mathcal{F}((U,V),\sigma)=(U,V), \ \sigma\in(0,1],$$
which is equivalent to the following system
\begin{equation}\label{periodic3}
\begin{cases}
 U_t=\Delta  U-\chi_1\nabla \cdot(  U\nabla W)-A_1(1-\sigma) U+ \sigma\varepsilon  U^{n+1}(1- U), &x\in\Omega, \ t>0,\\
 V_t=\Delta  V-\chi_2\nabla \cdot(V\nabla  U)-A_2(1-\sigma) V+\sigma\varepsilon  V^{n+1}(1- V),&x\in\Omega, \ t>0,\\
 W_t=d\Delta  W-\lambda( U+ V) W-\mu  W+r(x,t),&x\in\Omega, \ t>0,\\
\nabla U\cdot \nu=\nabla V\cdot \nu =\nabla W\cdot \nu=0, &x\in\partial\Omega, \ t>0, \\
 (U,V,W)(x,t)=(U,V,W)(x,t+T), &x\in\Omega, t\geq0.
\end{cases}     
\end{equation}
For clarity, we also emphasize that constants $c_i$ and $C_i~(i=1,2,\cdots)$ generated in this subsection are  independent of $\sigma,r_*$ and $\varepsilon$.

\begin{lemma}\label{l2p} Let $\operatorname{(H)}$ hold, and assume that $(U,V,W)$ is a $T$-periodic solution of \eqref{periodic3}. Then there exist constants $C_i>0~(i=1,2,3)$ independent of \(\sigma\), $r_*$ and $\varepsilon$ such that
\begin{align}
  \label{uv2p}
  &\sup_{t \in (0,T)} \int_{\Omega} U(x,t) + \sup_{t \in (0,T)} \int_{\Omega} V(x,t) + \int_0^T\int_{\Omega} U^{n+2} + \int_0^T\int_{\Omega} V^{n+2} \leq C_1,
    \end{align}
    and 
    \begin{align}\label{dwp}
  \int_0^T\|W\|_{W^{2,n+2}}^{n+2} \leq C_2r_*^{n+2},
    \end{align}
    as well as
  \begin{align} 
  \label{nw2p+2}
  \int_0^T \|\nabla W\|_{L^{2n+4}}^{2n+4} \leq C_3r_*^{2n+4}.
  \end{align}
\end{lemma}
\begin{proof}
The combination of \eqref{wbd}, \eqref{UCP} with \eqref{VCP} directly implies 
\begin{equation}\label{UVW}
   0\leq U, \quad  0\leq V, \quad 0(\not\equiv)\leq  W \leq \frac{r_*}{\mu}.
\end{equation}
Next, we show \eqref{uv2p}. Integrating the first equation in \eqref{periodic3} over $\Omega\times[0,T]$, then using the $T$-periodicity of $U$ and Young's inequality, we obtain
\begin{equation*}
    A_1(1-\sigma)\int_0^T\int_\Omega U+\sigma\varepsilon\int_0^T\int_\Omega  U^{n+2}=\sigma\varepsilon\int_0^T\int_\Omega  U^{n+1}\leq\frac{\sigma\varepsilon}{2}\int_0^T\int_\Omega  U^{n+2} + c_1\sigma\varepsilon,
\end{equation*}
which implies
\begin{equation}\label{l2p1}
\int_0^T\int_\Omega  U^{n+2}\leq  2c_1 .
\end{equation}
On the other hand, integrating the first equation in \eqref{periodic3} over $\Omega$, and adding $\int_\Omega U$ to both sides of the results yield
\begin{equation}\label{0629}
\frac{d}{dt}\int_\Omega U+\int_\Omega U+A_1(1-\sigma)\int_\Omega U+\sigma\varepsilon\int_\Omega U^{n+2}
=\sigma\varepsilon\int_\Omega U^{n+2}+\int_\Omega U\le c_2\int_\Omega U^{n+2}+c_2,
\end{equation}
where the last inequality follows from
$\sigma\in(0,1]$, $\varepsilon\in(0,1)$ and  Young's inequality. Therefore, for any $0\leq t_0<t\leq T$, integrating \eqref{0629} over $(t_0,t)$ and using \eqref{l2p1}, we have
\begin{equation}\label{l2p1-1}
\int_\Omega  U(\cdot,t)- \int_\Omega U(\cdot,t_0)+\int_{t_0}^t\int_\Omega  U(\cdot,s)ds\leq c_2\int_{t_0}^t\bigg(\int_\Omega U^{n+2}+1\bigg)ds.
\end{equation}
Noting \eqref{l2p1} and applying Lemma \ref{BDSS} to \eqref{l2p1-1} give
\begin{equation}\label{l2p2}
\sup_{t\in(0,T)}\int_\Omega U(x,t)\le c_3.
\end{equation}
Similarly, we deduce from the second equation of \eqref{periodic3} that
\begin{equation}\label{l2p2-1}
 \sup_{t \in (0,T)} \int_{\Omega}  V(x,t)+\int_0^T\int_\Omega  V^{n+2}\leq  c_4.
\end{equation}
Then the combination of \eqref{l2p1}, \eqref{l2p2}  and  \eqref{l2p2-1} gives \eqref{uv2p}. Finally, applying Lemma \ref{LT} to the third equation of \eqref{periodic3}, we obtain
\begin{equation*}
\begin{split}
\int_0^T\| W\|_{W^{2,n+2}}^{n+2}&\leq c_5 \int_0^T\|\lambda( U+ V)W+r(x,t) \|_{L^{n+2}}^{n+2}\\
&\leq c_5r_*^{n+2}\bigl(2\lambda^{n+2} c_1+\lambda^{n+2} c_4+|\Omega|T\bigr)\leq c_6r_*^{n+2},
\end{split}
\end{equation*}
which gives \eqref{dwp}. Using Gagliardo-Nirenberg inequality and \eqref{wbd}, one obtains
\begin{equation}\label{l2p2-3}
\begin{split}
\|\nabla  W\|_{L^{2n+4}}^{2n+4} 
\leq c_7 \| W\|_{L^\infty}^{n+2} \|D^2 W\|_{L^{n+2}}^{n+2} + c_8 \| W\|_{L^\infty}^{2n+4} 
\leq c_{9}r_*^{n+2} \left( \|D^2  W\|_{L^{n+2}}^{n+2} + r_*^{n+2} \right).  
\end{split}
\end{equation}
Integrating \eqref{l2p2-3} over $(0,T)$ and utilizing \eqref{dwp}, we derive
\begin{equation*}
    \int_0^T \|\nabla  W\|_{L^{2n+4}}^{2n+4} dt \leq c_9 r_*^{n+2} \left( \int_0^T \|D^2  W\|_{L^{n+2}}^{n+2} dt + Tr_*^{n+2} \right) \leq c_{10}r_*^{2n+4}.
\end{equation*}
This directly gives \eqref{nw2p+2}.
\end{proof}

\begin{lemma}\label{nw2p}
Let conditions in Lemma \ref{l2p} hold. Then there exists a constant $C_4>0$ independent of \(\sigma\), $r_*$ and $\varepsilon$ such that
\begin{equation}\label{nw2p-0}
\sup_{t \in (0,T)}\int_{\Omega} |\nabla  W|^{2n+2}\leq C_4(r_*^{2n+6}+1).
\end{equation}
\end{lemma}
\begin{proof}Using the third equation of  \eqref{periodic3} and \eqref{UVW}, we have
\begin{align}\notag
&\frac{1}{2n+2} \frac{d}{d t} \int_{\Omega}|\nabla  W|^{2n+2} + \mu \int_{\Omega}|\nabla  W|^{2n+2}\\ \notag
&=  \int_{\Omega}|\nabla  W|^{2n} \nabla  W \cdot \nabla  W_t + \mu \int_{\Omega}|\nabla  W|^{2n+2} \\\notag
\label{nw2p-1}
&=  \int_{\Omega}|\nabla  W|^{2n} \nabla  W \cdot \nabla[d\Delta  W-\lambda( U+ V)  W-\mu  W+r(x,t)] + \mu \int_{\Omega}|\nabla  W|^{2n+2} \\
& \leq d\int_{\Omega}|\nabla  W|^{2n} \nabla  W \cdot \nabla \Delta  W-\lambda \int_{\Omega}  W|\nabla  W|^{2n} \nabla  U \cdot \nabla  W \\ \notag
&\quad -\lambda \int_{\Omega}  W|\nabla  W|^{2n} \nabla  V \cdot \nabla  W + \int_{\Omega}|\nabla  W|^{2n} \nabla  W \cdot \nabla r(x,t) \\\notag
&=:  I_1+I_2+I_3+I_4.\notag
\end{align}
Using the identity $\nabla  W \cdot \nabla \Delta  W = \frac{1}{2}\Delta |\nabla  W|^2 - |D^2  W|^2$ and integrating the results by parts,  one has
\begin{equation}\label{nw2p-2*}
\begin{split}
I_1
&=\frac{d}{2}\int_\Omega |\nabla W|^{2n}\Delta |\nabla W|^2-d\int_\Omega |\nabla W|^{2n}|D^2W|^2\\
&=\frac{d}{2} \int_{\partial \Omega}|\nabla  W|^{2n} \frac{\partial|\nabla  W|^2}{\partial\nu} - \frac{dn}{2} \int_{\Omega}|\nabla  W|^{2n-2}\bigl|\nabla| \nabla  W|^2\bigr|^2  -d\int_{\Omega}|\nabla  W|^{2n}\left|D^2  W\right|^2.\\
\end{split}
\end{equation}
By \cite[Lemma 2.6]{Wang-2019} and using Young's inequality, there exists $c_1 > 0$ such that
\begin{equation}\label{wang}
\begin{split}
\int_{\partial \Omega}|\nabla  W|^{2n} \frac{\partial|\nabla  W|^2}{\partial\nu} \, dS
&\leq \frac{n}{2} \int_{\Omega}|\nabla  W|^{2n-2}\bigl|\nabla| \nabla  W|^2\bigr|^2+c_1 \int_{\Omega}|\nabla  W|^{2n+2}\\
&\leq \frac{n}{2} \int_{\Omega}|\nabla  W|^{2n-2}\bigl|\nabla| \nabla  W|^2\bigr|^2+c_2 \int_{\Omega}|\nabla  W|^{2n+4}+c_3.
\end{split}
\end{equation}
Then substituting \eqref{wang} into \eqref{nw2p-2*}, we derive 
\begin{equation}\label{nw2p-2}
\begin{aligned}
I_1 & \leq -d\int_{\Omega}|\nabla  W|^{2n}\left|D^2  W\right|^2-\frac{dn}{4} \int_{\Omega}|\nabla  W|^{2n-2}\bigl|\nabla| \nabla  W|^2\bigr|^2 + \frac{dc_2}{2} \int_{\Omega}|\nabla  W|^{2n+4} + \frac{dc_3}{2}.
\end{aligned}
\end{equation}
We integrate the term $I_2$ by parts, and use \eqref{UVW} along with the fact $|\Delta  W| \leq \sqrt{n}|D^2  W|$ to get
\begin{equation}\label{nw2p-3a}
\begin{aligned}
I_2 
&= \lambda \int_{\Omega}  U \nabla \cdot ( W|\nabla  W|^{2n} \nabla  W) \\
&=\lambda \int_{\Omega}  U |\nabla  W|^{2n+2}+\lambda \int_{\Omega}  UW |\nabla  W|^{2n}\Delta W+\lambda n\int_{\Omega}  UW |\nabla  W|^{2n-2}\nabla W\cdot \nabla|\nabla  W|^2\\
&\leq  c_4 r_*\int_{\Omega}  U |\nabla  W|^{2n} |D^2  W| + \lambda \int_{\Omega}  U |\nabla  W|^{2n+2}+c_4r_* \int_{\Omega}  U |\nabla  W|^{2n-1}| \nabla|\nabla  W|^2|.\\
\end{aligned}
\end{equation}
Then applying Young's inequality, one has 
\begin{equation}\label{nw2p-3b}
\begin{aligned}
c_4 r_* \int_{\Omega}  U |\nabla  W|^{2n} |D^2  W| 
& \leq \frac{d}{4} \int_{\Omega} |\nabla  W|^{2n} |D^2  W|^2 + \frac{c_4^2 r_*^2}{d} \int_{\Omega}  U^2 |\nabla  W|^{2n}\\
&\leq \frac{d}{4} \int_{\Omega} |\nabla  W|^{2n} |D^2  W|^2+c_5r_*^2 \int_\Omega U^{n+2}+c_5r_*^2\int_\Omega |\nabla W|^{2n+4},\\
\end{aligned}
\end{equation}
and 
\begin{equation}\label{nw2p-3b*}
\begin{aligned}
&\lambda \int_{\Omega}  U |\nabla  W|^{2n+2}+c_4 r_* \int_{\Omega}  U |\nabla  W|^{2n-1} | \nabla|\nabla  W|^2 |\\
& \leq \frac{dn}{16} \int_{\Omega} |\nabla  W|^{2n-2} | \nabla|\nabla  W|^2 |^2 + c_6 r_*^2 \int_{\Omega}  U^2 |\nabla  W|^{2n}+\lambda \int_{\Omega}  U |\nabla  W|^{2n+2}\\
&\leq \frac{dn}{16} \int_{\Omega} |\nabla  W|^{2n-2} | \nabla|\nabla  W|^2 |^2+c_7(r_*^2+1) \int_\Omega U^{n+2}+c_8(r_*^2+1)\int_\Omega |\nabla W|^{2n+4}.
\end{aligned}
\end{equation}
Substituting \eqref{nw2p-3b} and \eqref{nw2p-3b*} into \eqref{nw2p-3a} gives
\begin{equation}\label{nw2p-3}
\begin{aligned}
I_2&\leq \frac{d}{4} \int_{\Omega} |\nabla  W|^{2n} |D^2  W|^2 + \frac{dn}{16} \int_{\Omega} |\nabla  W|^{2n-2} | \nabla|\nabla  W|^2 |^2\\
&\ \ \ \ + c_9(r_*^2+1) \int_{\Omega} |\nabla  W|^{2n+4}+ c_9(r_*^2+1) \int_{\Omega}  U^{n+2}.\\
\end{aligned}
\end{equation}
Following the same arguments as in the proof of \eqref{nw2p-3}, we obtain
\begin{equation}\label{nw2p-4}
\begin{aligned}
I_3 &= \lambda \int_{\Omega}  V \nabla \cdot ( W|\nabla  W|^{2n} \nabla  W)\\
&\leq \frac{d}{4} \int_{\Omega} |\nabla  W|^{2n} |D^2  W|^2 + \frac{dn}{16} \int_{\Omega} |\nabla  W|^{2n-2}| \nabla|\nabla  W|^2|^2 \\
&\ \ \ \  + c_{10}(r_*^2+1) \int_{\Omega} |\nabla  W|^{2n+4} + c_{11}(r_*^2+1) \int_{\Omega}  V^{n+2}.
\end{aligned}
\end{equation}
Using the facts $0\leq r(x,t) \leq r_*$, $|\Delta  W| \leq \sqrt{n}|D^2  W|$  and Young's inequality, one derives
\begin{equation}\label{nw2p-5}
\begin{aligned}
I_4 &= - \int_{\Omega} r(x,t) |\nabla  W|^{2n} \Delta  W - n \int_{\Omega} r(x,t) |\nabla  W|^{2n-2} \nabla(|\nabla  W|^2) \cdot \nabla  W\\
&\leq \frac{d}{4} \int_{\Omega} |\nabla  W|^{2n} |D^2  W|^2 + \frac{dn}{16} \int_{\Omega} |\nabla  W|^{2n-2}| \nabla|\nabla  W|^2|^2 \\
& \ \ \ \ 
+ c_{12}r_*^2\int_{\Omega} |\nabla  W|^{2n+4} + c_{13}r_*^2.
\end{aligned}
\end{equation}
We substitute \eqref{nw2p-2}, \eqref{nw2p-3}, \eqref{nw2p-4} and \eqref{nw2p-5} into \eqref{nw2p-1} to obtain
\begin{equation}\label{nw2p-6}
\frac{d}{dt}\int_\Omega |\nabla W|^{2n+2}+2(n+1)\mu\int_\Omega |\nabla W|^{2n+2}
\le c_{14}(r_*^2+1)\int_\Omega(1+|\nabla W|^{2n+4}+U^{n+2}+V^{n+2}).
\end{equation}
Therefore, for any $0\le t_0<t\le T$, integrating \eqref{nw2p-6} over $(t_0,t)$ gives
\begin{equation}\label{nw2p-7}
\begin{aligned}
&\int_\Omega |\nabla W(\cdot,t)|^{2n+2}-\int_\Omega |\nabla W(\cdot,t_0)|^{2n+2}
+2(n+1)\mu\int_{t_0}^t\int_\Omega |\nabla W(\cdot,s)|^{2n+2}ds\\
&\le c_{14}(r_*^2+1)\int_{t_0}^t\int_\Omega(1+|\nabla W|^{2n+4}+U^{n+2}+V^{n+2})ds.
\end{aligned}
\end{equation}
Applying Lemma \ref{BDSS} to \eqref{nw2p-7} and using Lemma \ref{l2p}, we derive
\begin{equation*}
\sup_{t\in(0,T)}\int_\Omega |\nabla W|^{2n+2}\le c_{15}(r_*^{2n+6}+1).
\end{equation*}
This completes the proof of Lemma \ref{nw2p}.
\end{proof}
\begin{lemma}\label{u-infty}
Let the conditions in Lemma \ref{l2p} hold. Then there exists a positive constant $C_5$ independent of \(\sigma\), $r_*$ and $\varepsilon$, such that
\begin{equation}\label{ui0}
\sup_{t \in (0, T)} \| U(\cdot, t)\|_{L^\infty} \leq C_5(r_*^{2n+6}+1).
\end{equation}
\end{lemma}
\begin{proof}
Multiplying the first equation of \eqref{periodic3} by $q U^{q-1}(q\geq 2)$, and integrating the results over $\Omega$, then using Young's inequality, we end up with
\begin{equation*}
\begin{aligned}
& \frac{d}{dt} \int_{\Omega}  U^q +   q(q-1) \int_{\Omega}  U^{q-2}|\nabla  U|^2 + A_1(1-\sigma)q\int_{\Omega}  U^{q}+\sigma\varepsilon q\int_{\Omega}  U^{n+1+q}\\
&=  \sigma \varepsilon q \int_{\Omega}   U^{n+q} + \chi_1 q(q-1) \int_{\Omega}  U^{q-1} \nabla  U \cdot \nabla  W  \\
&\leq  \sigma \varepsilon q \int_{\Omega} ( U^{n+1+q} +  U^{q}) + \frac{  q(q-1)}{2} \int_{\Omega}  U^{q-2}|\nabla  U|^2 + \frac{\chi_1^2 q(q-1)}{2 } \int_{\Omega}  U^q |\nabla  W|^2,
\end{aligned}
\end{equation*}
which, together with the facts $\int_\Omega U^{q-2}|\nabla U|^2=\frac{4}{q^2}\int_\Omega|\nabla U^{\frac q2}|^2$, $\sigma\in(0,1]$ and $\varepsilon\in(0,1)$, gives
\begin{equation}\label{ULI}
\begin{aligned}
\frac{d}{dt} \int_{\Omega}  U^q +  \frac{2(q-1)}{q}\int_\Omega|\nabla U^{\frac q2}|^2 
&\leq  q \int_{\Omega} U^{q} + \frac{\chi_1^2 q(q-1)}{2 } \int_{\Omega}  U^q |\nabla  W|^2.\\
\end{aligned}
\end{equation}
Using the H\"{o}lder inequality, \eqref{nw2p-0} and Gagliardo-Nirenberg inequality, we have
\begin{equation}\label{uign1}
\begin{split}
 \frac{\chi_1^2 q(q-1)}{2 } \int_{\Omega}  U^q |\nabla  W|^2
 &\leq \frac{\chi_1^2 q(q-1)}{2 }  \left( \int_{\Omega}  U^{\frac{(n+1)q}{n}} \right)^{\frac{n}{n+1}} \left( \int_{\Omega} |\nabla  W|^{2(n+1)} \right)^{\frac{1}{n+1}} \\
 &\leq c_1(r_*^{2+\frac{4}{n+1}}+1)q(q-1) \left( \int_{\Omega}  U^{\frac{(n+1)q}{n}} \right)^{\frac{n}{n+1}}\\
 &=c_1(r_*^{2+\frac{4}{n+1}}+1)q(q-1)\| U^{\frac{q}{2}}\|_{L^{\frac{2n+2}{n}}}^2\\
&\leq c_2(r_*^{2+\frac{4}{n+1}}+1)q(q-1)(\|\nabla  U^{\frac{q}{2}}\|^{\frac{2n}{n+1}}_{L^2}\| U^{\frac{q}{2}}\|_{L^1}^{\frac{2}{n+1}}+\| U^{\frac{q}{2}}\|_{L^1}^2)\\
&\leq \frac{q-1}{2q}\|\nabla  U^{\frac{q}{2}}\|^{2}_{L^2}+ c_3(r_*^{2n+6}+1)q^{2(n+1)}\| U^{\frac{q}{2}}\|_{L^1}^2,
\end{split}
\end{equation}
and
\begin{equation}\label{uign2}
\begin{split}
(q+1) \| U\|_{L^q}^q \leq  2q \| U^{\frac{q}{2}}\|_{L^2}^2
&\leq c_4q( \|\nabla  U^{\frac{q}{2}}\|_{L^2}^{\frac{2n}{n+2}} \| U^{\frac{q}{2}}\|_{L^1}^{\frac{4}{n+2}} +\| U^{\frac{q}{2}}\|_{L^1}^2) \\
&\leq \frac{q-1}{2q}\|\nabla  U^{\frac{q}{2}}\|_{L^2}^2 + c_5q^{n+1}  \| U^{\frac{q}{2}}\|_{L^1}^2.     
\end{split}
\end{equation}
Here and all subsequent constants $c_i>0$ are further independent of $q$.
Substituting \eqref{uign1} and \eqref{uign2} into \eqref{ULI}, we have
\begin{equation}\label{uign2-1}
\frac{d}{dt} \int_{\Omega}  U^q+ \int_{\Omega}  U^q \leq c_6(r_*^{2n+6}+1)q^{2(n+1)}{\bigg(\int_\Omega U^{\frac{q}{2}}\bigg)^2}.
\end{equation}
Therefore, for any $0\le t_0<t\le T$, integrating \eqref{uign2-1} over $(t_0,t)$, we obtain
\begin{equation}\label{u-i3}
\int_\Omega U^q(x,t)-\int_\Omega U^q(x,t_0)
+\int_{t_0}^t\int_\Omega U^q(x,s)
\le c_6(r_*^{2n+6}+1)q^{2(n+1)}
\int_{t_0}^t\|U(\cdot,s)\|_{L^{\frac q2}}^q\,ds.
\end{equation}
Applying Lemma \ref{BDSS} to \eqref{u-i3} yields
\begin{equation*}
\begin{split}
\sup_{t\in(0,T)}\|U(\cdot,t)\|_{L^q}^q
&\le c_7(r_*^{2n+6}+1)q^{2(n+1)}
\int_0^T \|U(\cdot,t)\|_{L^{\frac q2}}^q\,dt\\
&\le c_8(r_*^{2n+6}+1)q^{2(n+1)}
\sup_{t\in(0,T)}\|U(\cdot,t)\|_{L^{\frac q2}}^q .
\end{split}
\end{equation*}

Define $M_j=\sup_{t\in(0,T)}\|U(\cdot,t)\|_{L^{q_j}}$ with $q_j=2^{j-1}$ ($j\in\mathbb{Z}^+$). Then, for $j\ge2$, we have
\begin{equation}\label{Moser}
\begin{split}
\sup_{t \in (0, T)} \| U(\cdot, t)\|_{L^{2^{j-1}}}=M_j
&\leq c_8^{\frac{1}{2^{j-1}}}(r_*^{2n+6}+1)^{\frac{1}{2^{j-1}}}2^{\frac{2(n+1)(j-1)}{2^{j-1}}}M_{j-1}\\
&\leq c_8^{\sum_{i=2}^{j}\frac{1}{2^{i-1}}}(r_*^{2n+6}+1)^{\sum_{i=2}^{j}\frac{1}{2^{i-1}}}\cdot2^{2(n+1)\sum_{i=2}^{j}\frac{i-1}{2^{i-1}}}M_1\\
&=c_8^{(1-\frac{1}{2^{j-1}})}(r_*^{2n+6}+1)^{(1-\frac{1}{2^{j-1}})}\cdot2^{2(n+1)(2-\frac{j+1}{2^{j-1}})}M_1.
\end{split}
\end{equation}
Letting $j\to \infty$ and noting \eqref{uv2p}, we deduce from  \eqref{Moser} that 
\begin{equation*}
\sup_{t \in (0,T)}\|U(\cdot,t)\|_{L^\infty}\leq c_9(r_*^{2n+6}+1).
\end{equation*}
This completes the proof of Lemma \ref{u-infty}.
\end{proof}

\begin{lemma}\label{Du2n+2}
Let the conditions in Lemma \ref{l2p} hold. Then we have
\begin{equation}\label{nu2p}
\int_0^T\int_\Omega|D^2  U|^{n+2} \leq C_6(r_*^{4(n+2)^3}+1),
\end{equation}
and 
\begin{equation}\label{0424*}
\int_0^T\int_\Omega |\nabla U|^{2n+4}\leq C_7(r_*^{4(n+2)^3}+1),
\end{equation}
where $C_6$ and $C_7$ are two positive constants  independent of \(\sigma\), $r_*$ and $\varepsilon$.
\end{lemma}
\begin{proof}
In the case $U\equiv 0$, \eqref{nu2p} and \eqref{0424*} hold trivially. We now turn to the nontrivial case $U\not\equiv 0$, hence $\| U(\cdot,t)\|_{L^\infty}\not=0$. We first rewrite $U$-equation in \eqref{periodic3} as
\begin{equation}\label{U*}
\begin{cases}
 U_t-\Delta  U+A_1 U=F_1( x, t), &x\in\Omega, \ t>0,\\
{\nabla U}\cdot\nu=0, &x\in\partial\Omega, \ t>0, \\
 U(x,t)= U(x,t+T), &x\in\Omega, t\geq0,
\end{cases}
\end{equation}
where $F_1( x,  t):=-\chi_1\nabla \cdot(  U\nabla W)+A_1\sigma U+ \sigma\varepsilon  U^{n+1}(1- U).$   Applying \eqref{LT1} with $p=n+2$ to the system \eqref{U*}, then using \eqref{dwp} and \eqref{ui0}, we have 
\begin{align}\notag
\int_0^T\int_\Omega |D^2 U|^{n+2}
&\leq \int_0^T\left(\chi_1\|\nabla \cdot ( U\nabla  W)\|_{L^{n+2}}^{n+2}+\|A_1\sigma U+\sigma\varepsilon U^{n+1}(1- U)\|_{L^{n+2}}^{n+2}\right)\\ \notag
&\leq  c_1\int_0^T\left(\|\nabla  U\cdot\nabla  W\|_{L^{n+2}}^{n+2}+\| U\|_{L^{\infty}}^{n+2}\|\Delta  W\|_{L^{n+2}}^{n+2}+\|A_1U+ U^{n+1}+ U^{n+2}\|_{L^\infty}^{n+2}\right)\\ 
\label{Du2n+2-1}
&\leq c_1\int_0^T\|\nabla  U\cdot\nabla  W\|_{L^{n+2}}^{n+2}+c_2\left(r_*^{3(n+2)^3}+1\right).
\end{align}
Applying the Gagliardo-Nirenberg inequality and \cite[Lemma 1]{1974-lemma1}, we can find a constant $C_N>0$ such that 
\begin{equation}\label{gn1}
\begin{aligned}
\|\nabla U(\cdot,t)\|_{L^{2n+4}}^{2n+4}
&\leq c_3\big(\|D^2 U(\cdot,t)\|_{L^{n+2}}^{\frac{1}{2}}\| U(\cdot,t)\|_{L^\infty}^{\frac{1}{2}}+\| U(\cdot,t)\|_{L^\infty}\big)^{2n+4}\\
&\leq C_N\| U(\cdot,t)\|_{L^\infty}^{n+2}\|D^2 U(\cdot,t)\|_{L^{n+2}}^{n+2}+C_N\| U(\cdot,t)\|_{L^\infty}^{2n+4}.
\end{aligned}
\end{equation}
Then we apply \eqref{gn1} and Young's inequality to derive
\begin{equation}\label{Du2n+2-2}
\begin{aligned}
c_1\|\nabla  U\cdot\nabla  W\|_{L^{n+2}}^{n+2} 
&= c_1\int_\Omega|\nabla U|^{n+2}|\nabla  W|^{n+2} \\
&\leq \frac{1}{2C_N\| U(\cdot,t)\|_{L^\infty}^{n+2}}\|\nabla U(\cdot,t)\|_{L^{2n+4}}^{2n+4}+\frac{c_1^2C_N\| U(\cdot,t)\|_{L^\infty}^{n+2}}{2}\|\nabla W(\cdot,t)\|_{L^{2n+4}}^{2n+4}\\
&\leq \frac{1}{2}\int_\Omega |D^2 U|^{n+2}+\frac{1}{2}\| U(\cdot,t)\|_{L^\infty}^{n+2}+\frac{c_1^2 C_N\| U(\cdot,t)\|_{L^\infty}^{n+2}}{2}\|\nabla W(\cdot,t)\|_{L^{2n+4}}^{2n+4}.
\end{aligned}
\end{equation}
Substituting \eqref{Du2n+2-2} into \eqref{Du2n+2-1}, then utilizing \eqref{nw2p+2} and \eqref{ui0}, we derive
\begin{equation}\label{0424}
\begin{aligned}
\frac{1}{2}\int_0^T\int_\Omega |D^2 U|^{n+2}
&\leq  \frac{T}{2}\sup\limits_{t\in (0,T)}\| U(\cdot,t)\|_{L^\infty}^{n+2}+ \frac{c_1^2C_N}{2}\sup\limits_{t\in (0,T)}\| U(\cdot,t)\|_{L^\infty}^{n+2}\int_0^T\int_\Omega|\nabla  W|^{2n+4}\\
&\quad + c_4(r_*^{3(n+2)^3}+1)\\
&\leq c_5\big(r_*^{2(n+3)(n+2)}
+ r_*^{2(n+2)(n+4)}
+ r_*^{3(n+2)^3}+1\big)\\
&\leq c_6\big(r_*^{3(n+2)^3}+1\big).
\end{aligned}
\end{equation}
With \eqref{0424} and \eqref{gn1}, we use Lemma \ref{u-infty} to get
\begin{equation}\label{0424-1}
\begin{aligned}
\int_0^T\int_\Omega |\nabla U|^{2n+4}
&\leq C_N\sup\limits_{t\in(0,T)}\| U(\cdot,t)\|_{L^\infty}^{n+2}\int_0^T\|D^2 U(\cdot,t)\|_{L^{n+2}}^{n+2}+C_N\sup\limits_{t\in(0,T)}\| U(\cdot,t)\|_{L^\infty}^{2n+4}\\
&\leq c_7 (r_*^{2n+6}+1)^{n+2}(r_*^{3(n+2)^3}+1\big)+c_8 (r_*^{2n+6}+1)^{2n+4}\\
&\leq c_8(r_*^{2(n+3)(n+2)+3(n+2)^3}+r_*^{4(n+3)(n+2)}+1)\\
&\leq c_{9}(r_*^{4(n+2)^3}+1).
\end{aligned}
\end{equation}
Then the combination of \eqref{0424} with \eqref{0424-1} finishes the proof of Lemma \ref{Du2n+2}.

\end{proof}
\begin{lemma}\label{Nu2p}
Let conditions in Lemma \ref{l2p} hold. Then there exist two positive constants $C_8$ and $C_9$ independent of \(\sigma\), $r_*$ and $\varepsilon$, such that
\begin{equation}\label{nu2n_est} 
\sup_{t \in (0,T)}\int_{\Omega} |\nabla  U|^{2n+2}\leq C_8(r_*^{5(n+2)^3}+1), 
\end{equation}
and 
\begin{equation}\label{vi0}
\sup_{t \in (0, T)} \| V(\cdot, t)\|_{L^\infty} \leq C_9(r_*^{5(n+2)^3}+1). 
\end{equation}
\end{lemma}
\begin{proof}
Using the integration by parts,  we deduce from \eqref{periodic3} that
\begin{equation} \label{nu2p-1*}
\begin{aligned}
\frac{1}{2(n+1)} \frac{d}{dt} \int_{\Omega} |\nabla U|^{2(n+1)}
&= \int_{\Omega} |\nabla U|^{2n} \nabla U \cdot \nabla U_t  \\
&= \int_{\Omega} |\nabla U|^{2n} \nabla U \cdot \nabla \big( \Delta U - \chi_1 \nabla \cdot ( U \nabla W) + F( U) \big)  \\
&= \int_{\Omega} |\nabla U|^{2n} \nabla U \cdot \nabla \Delta U
   - \chi_1 \int_{\Omega} |\nabla U|^{2n} \nabla U \cdot \nabla (\nabla \cdot ( U \nabla W))  \\
& \ \ \ \ + \int_{\Omega} |\nabla U|^{2n} \nabla U \cdot \nabla F(U),
\end{aligned}
\end{equation}
where 
$$F(U):=-A_1(1-\sigma)U+\sigma\varepsilon U^{n+1}(1-U).$$
Applying the equality $\nabla  U \cdot \nabla \Delta  U = \frac{1}{2}\Delta |\nabla  U|^2 - |D^2  U|^2$ and using the integration by parts again, we have
\begin{equation*}
\begin{split}
&\int_{\Omega} |\nabla U|^{2n} \nabla U \cdot \nabla \Delta U\\
&=\frac{1}{2} \int_{\Omega} |\nabla  U|^{2n} \Delta |\nabla  U|^2 - \int_{\Omega} |\nabla  U|^{2n} |D^2  U|^2\\
&=\frac{1}{2} \int_{\partial \Omega} |\nabla  U|^{2n} \frac{\partial|\nabla  U|^2}{\partial\nu} \, dS - \frac{n}{2} \int_{\Omega} |\nabla  U|^{2n-2}| \nabla |\nabla  U|^2|^2 - \int_{\Omega} |\nabla  U|^{2n} |D^2  U|^2,
\end{split}
\end{equation*}
which updates \eqref{nu2p-1*} as 
\begin{equation}\label{nu2p-1}
\begin{aligned}
&\frac{1}{2(n+1)} \frac{d}{dt} \int_{\Omega} |\nabla U|^{2(n+1)}+ \frac{n}{2} \int_{\Omega} |\nabla  U|^{2n-2}| \nabla |\nabla  U|^2|^2 +\int_{\Omega} |\nabla  U|^{2n} |D^2  U|^2\\
&= \frac{1}{2} \int_{\partial \Omega} |\nabla  U|^{2n} \frac{\partial|\nabla  U|^2}{\partial\nu} \, dS -\chi_1 \int_{\Omega} |\nabla U|^{2n} \nabla U \cdot \nabla (\nabla \cdot ( U \nabla W))\\
&\ \ \ \ + \int_{\Omega} |\nabla U|^{2n} \nabla U \cdot \nabla F( U)\\
&=:J_1+J_2+J_3.
\end{aligned}
\end{equation}
It follows from \cite[Lemma 2.6]{Wang-2019} that
\begin{equation}\label{nu2p-2}
\begin{aligned}
J_1
&\leq \frac{n}{4}\int_{\Omega}|\nabla  U|^{2n-2}\bigl|\nabla| \nabla  U|^2\bigr|^2 + c_{1} \int_{\Omega}|\nabla  U|^{2(n+1)}.
\end{aligned}
\end{equation}
Using the facts $ |\Delta \phi |\leq \sqrt{n}|D^2 \phi|$  and $\nabla|\nabla U|^2=2\nabla U \cdot D^2U$, we estimate the term $J_2$ as follows:
\begin{equation}\label{nu2p-J2-1}
\begin{aligned}
J_2&= \chi_1 \int_{\Omega} \nabla \cdot (|\nabla  U|^{2n} \nabla  U) \nabla \cdot ( U \nabla  W) \\
&\leq \chi_1 \int_{\Omega} \left( |\nabla  U|^{2n}|\Delta  U| + n|\nabla  U|^{2n-1} \left|  \nabla |\nabla  U|^2 \right| \right)\left( |\nabla  U| |\nabla  W| +  U |\Delta W| \right)\\
&\leq \chi_1 (2n+\sqrt{n}) \int_{\Omega} |\nabla  U|^{2n} |D^2  U| \left( |\nabla  U| |\nabla  W| +  U \sqrt{n}|D^2  W| \right) \\
&\leq \frac{1}{2} \int_{\Omega} |\nabla  U|^{2n} |D^2  U|^2 + c_{2} \int_{\Omega} |\nabla  U|^{2(n+1)} |\nabla  W|^2  +c_2  \int_{\Omega} U^2|\nabla  U|^{2n} |D^2  W|^2.
\end{aligned}
\end{equation}
As for the term $J_3$, we derive that  
\begin{equation}\label{nu2p-4}
\begin{split}
J_3 &= \int_{\Omega} |\nabla  U|^{2n} \nabla  U \cdot [-A_1(1-\sigma)\nabla U+\sigma\varepsilon (n+1) U^{n}\nabla U-\sigma\varepsilon (n+2)U^{n+1}\nabla U]\\
&\leq c_3 \int_{\Omega} (1+U^{n}+U^{n+1})|\nabla  U|^{2(n+1)}. 
\end{split}
\end{equation}
Substituting \eqref{nu2p-2} - \eqref{nu2p-4} into \eqref{nu2p-1}, and applying \eqref{ui0} along with Young's inequality yield
\begin{equation}\label{nu2p-5}
\begin{aligned}
&\frac{d}{dt} \int_{\Omega} |\nabla U|^{2(n+1)} + 2(n+1)\int_{\Omega} |\nabla U|^{2(n+1)} \\
&\le c_4 \int_{\Omega} |\nabla W|^2 |\nabla U|^{2(n+1)}
+c_4 \int_{\Omega} U^2|\nabla U|^{2n}|D^2W|^2
+c_4\int_{\Omega}(1+U^{n}+U^{n+1})|\nabla U|^{2(n+1)} \\
&\le c_4 \int_{\Omega} |\nabla W|^2|\nabla U|^{2(n+1)}
+c_5(r_*^{4(n+3)}+1)\int_{\Omega}|\nabla U|^{2n}|D^2W|^2\\
&\quad +c_6(r_*^{2(n+3)(n+1)}+1)\int_{\Omega}|\nabla U|^{2(n+1)} \\
&\le c_7\int_{\Omega}|\nabla W|^{2n+4}
+c_{11}(r_*^{4(n+3)}+1)\int_{\Omega}|D^2W|^{n+2}
+c_{13}(r_*^{2(n+3)(n+1)}+1) \\
&\quad +c_{12}(r_*^{2(n+3)(n+1)}+r_*^{4(n+3)}+1)\int_{\Omega}|\nabla U|^{2n+4}\\
&=:h(t).
\end{aligned}
\end{equation}
For any $0\le t_0<t\le T$, integrating \eqref{nu2p-5} over $(t_0,t)$ gives
\begin{equation}\label{0701-1}
\begin{aligned}
\int_\Omega |\nabla U(\cdot,t)|^{2(n+1)}-\int_\Omega |\nabla U(\cdot,t_0)|^{2(n+1)}
+2(n+1)\int_{t_0}^t\int_\Omega |\nabla U(\cdot,s)|^{2(n+1)}ds\le \int_{t_0}^t h(s)ds.
\end{aligned}
\end{equation}
We use \eqref{dwp}, \eqref{nw2p+2} and \eqref{0424*} to get
\begin{equation}\label{0707}
\begin{split}
\int_{0}^T h(s)ds& \leq c_{14}\big( r_*^{2(n+2)}+(r_*^{4(n+3)}+1)r_*^{n+2}+r_*^{2(n+3)(n+1)}+1\\
&\quad +(r_*^{2(n+3)(n+1)}+r_*^{4(n+3)}+1)(r_*^{4(n+2)^3}+1)\big)\\
&\leq  c_{15} (r_*^{5(n+2)^3}+1).
\end{split}
\end{equation}
Applying Lemma \ref{BDSS} to \eqref{0701-1} and using \eqref{0707} yield
\begin{equation*}
\sup_{t\in(0,T)}\int_\Omega |\nabla U|^{2(n+1)}
\le c_{16}(r_*^{5(n+2)^3}+1),
\end{equation*}
which gives \eqref{nu2n_est}.

Finally, with \eqref{nu2n_est}, following the similar Moser iteration arguments as in the proof of Lemma \ref{u-infty}, we derive 
\begin{equation*}
\sup_{t \in (0, T)} \| V(\cdot, t)\|_{L^\infty} \leq c_{17}(r_*^{5(n+2)^3}+1).
\end{equation*}
which gives \eqref{vi0}. Thus, the proof of Lemma \ref{Nu2p} is complete.
\end{proof}

\begin{lemma}\label{NV2n+2}
Let the conditions in Lemma \ref{l2p} automatically hold. Then there exists a positive constant $C_{10}$ independent of \(\sigma\), $r_*$ and $\varepsilon$, such that
\begin{equation}\label{0424*1}
\int_0^T\int_\Omega |\nabla V|^{2n+4}\leq C_{10}(r_*^{10(n+2)^5}+1).
\end{equation}
\end{lemma}
\begin{proof}
If $V\equiv 0$, \eqref{0424*1} holds. It remains to consider $V\not\equiv 0$, for which $\| V(\cdot,t)\|_{L^\infty}\not=0$.  We rewrite the $V$-equation in \eqref{periodic3} as
\begin{equation}\label{V*}
\begin{cases}
V_t-\Delta V+A_2V=F_2(x,t), & x\in\Omega,\ t>0,\\
\nabla V\cdot\nu=0, & x\in\partial\Omega,\ t>0,\\
V(x,t)=V(x,t+T), & x\in\Omega,\ t\geq 0,
\end{cases}
\end{equation}
where
\begin{equation*}
F_2(x,t)=-\chi_2\nabla V\cdot\nabla U-\chi_2V\Delta U+A_2\sigma V+\sigma\varepsilon V^{n+1}(1-V).
\end{equation*}
Taking $p=n+2$ and applying Lemma \ref{LT} to \eqref{V*}, we have
\begin{equation}\label{d2v_est}
\begin{split}
\int_0^T\int_\Omega |D^2V|^{n+2}
\leq& c_1\int_0^T
\Big(
\|\nabla V\cdot\nabla U\|_{L^{n+2}}^{n+2}
+\|V\Delta U\|_{L^{n+2}}^{n+2}\Big)dt\\
&+ c_1\int_0^T
\Big(\|A_2V+V^{n+1}+V^{n+2}\|_{L^{n+2}}^{n+2}
\Big)\,dt .
\end{split}
\end{equation}
We now estimate the two terms inside the first integral for each $t\in(0,T)$. By Young's inequality and the Gagliardo--Nirenberg inequality used in \eqref{gn1}, we have 
\begin{equation}\label{grad_est}
\begin{split}
c_1\|\nabla V\cdot\nabla U\|_{L^{n+2}}^{n+2}
&=c_1\int_\Omega |\nabla V|^{n+2}|\nabla U|^{n+2} \\
&\leq \frac{1}{2C_N\|V(\cdot,t)\|_{L^\infty}^{n+2}}\|\nabla V(\cdot,t)\|_{L^{2n+4}}^{2n+4}+c_2\|V(\cdot,t)\|_{L^\infty}^{n+2}\|\nabla U(\cdot,t)\|_{L^{2n+4}}^{2n+4} \\
&\leq \frac12\|D^2V(\cdot,t)\|_{L^{n+2}}^{n+2}+\frac12\|V(\cdot,t)\|_{L^\infty}^{n+2}+c_2\|V(\cdot,t)\|_{L^\infty}^{n+2}\|\nabla U(\cdot,t)\|_{L^{2n+4}}^{2n+4}.
\end{split}
\end{equation}
For the second integrand, one has
\begin{equation}\label{lap_est}
\|V\Delta U\|_{L^{n+2}}^{n+2} \leq \|V(\cdot,t)\|_{L^\infty}^{n+2}\|\Delta U(\cdot,t)\|_{L^{n+2}}^{n+2}.
\end{equation}
Substituting \eqref{grad_est} and \eqref{lap_est} into \eqref{d2v_est}, and using the facts  $|\Delta U|\leq \sqrt n\, |D^2U|,$  \eqref{nu2p}, \eqref{0424-1} and \eqref{vi0}, we obtain
\begin{equation}\label{0424d2v}
\begin{split}
\int_0^T\int_\Omega |D^2V|^{n+2}
&\leq c_3\sup_{t\in(0,T)}\|V(\cdot,t)\|_{L^\infty}^{n+2} \int_0^T\int_\Omega \left(|\nabla U|^{2n+4}+|\Delta U|^{n+2}\right) \\
&\quad +c_3\sup_{t\in(0,T)}\big(\|V(\cdot,t)\|_{L^\infty}^{n+2}+\|V(\cdot,t)\|_{L^\infty}^{(n+1)(n+2)}+\|V(\cdot,t)\|_{L^\infty}^{(n+2)^2}\big)\\
&\leq c_4\big(r_*^{5(n+2)^5}+1\big).
\end{split}
\end{equation}
Next, applying the Gagliardo-Nirenberg inequality to $V$ and using \eqref{vi0} with \eqref{0424d2v}, we have 
\begin{equation*}
\begin{split}
\int_0^T\|\nabla V(\cdot,t)\|_{L^{2n+4}}^{2n+4} 
&\leq c_5\int_0^T \left(\|V(\cdot,t)\|_{L^\infty}^{n+2}\|D^2V(\cdot,t)\|_{L^{n+2}}^{n+2}+\|V(\cdot,t)\|_{L^\infty}^{2n+4}\right)\\
&\leq c_6\big(r_*^{10(n+2)^5}+1\big),
\end{split}
\end{equation*}
which gives \eqref{0424*1}. Hence we complete the proof of Lemma \ref{NV2n+2}.
\end{proof}

\begin{lemma}\label{W21p_global}
Let conditions in Lemma \ref{l2p} hold. Then there exists a constant $M\geq 1$ independent of $\varepsilon$ and $\sigma$ but increasing in $r_*$ such that
\begin{equation}\label{c2a}
    \|(U,V,W)\|_{C^{2+\alpha, 1+\frac{\alpha}{2}}(\overline{Q}_T)} \leq M,
\end{equation}
for some constant $\alpha\in (0,1)$.

\end{lemma}
\begin{proof}
In view of \eqref{ui0}, \eqref{vi0}, \eqref{UVW} and (H), applying \eqref{LT1} to the third equation of \eqref{periodic3}, we obtain
\begin{equation}\label{c2a-1}
\begin{split}
\| W\|_{W^{2,1}_{\tilde{p}}(Q_T)}
&\leq c_1\|-\lambda( U+ V) W+r(x,t)\|_{L^{\tilde{p}}(Q_T)}\\
&\leq c_1r_*\|\lambda (U+V)/\mu+1\|_{L^{\tilde{p}}(Q_T)}\\
&\leq c_2r_*(r_*^{2n+6}+r_*^{5(n+2)^3}+1)\\
&\leq c_3(r_*^{6(n+2)^3}+1),  \ \ \forall \tilde{p}>n+2.
\end{split}
\end{equation}
By the Sobolev embedding theorem, $W^{2,1}_{\tilde{p}}(Q_T) \hookrightarrow C^{1+\alpha_1, \frac{1+\alpha_1}{2}}(\overline{Q}_T)$ with $0 < \alpha_1 < 1 - \frac{n+2}{\tilde{p}}$, we obtain
\begin{equation}\label{c2a-2}
    \|\nabla  W\|_{L^\infty(Q_T)} \leq \| W\|_{ C^{1+\alpha_1, \frac{1+\alpha_1}{2}}(\overline{Q}_T)}\leq c_4(r_*^{6(n+2)^3}+1).
\end{equation}
Applying Lemma \ref{LT} to \eqref{U*},  using  \eqref{ui0}, \eqref{c2a-1}, \eqref{c2a-2} and \eqref{0424*}, for $\tilde{p}=2n+4$, one has
\begin{equation*}
\begin{split}
\| U\|_{W^{2,1}_{\tilde{p}}(Q_T)}
&\leq c_5\|f\|_{L^{\tilde{p}}(Q_T)}\\
&=c_5\|-\chi_1\nabla \cdot(  U\nabla W)+A_1\sigma U+ \sigma\varepsilon  U^{n+1}(1- U)\|_{L^{\tilde{p}}(Q_T)}\\
&\leq c_6 (r_*^{2n+6}+1)^{n+2}+c_6(r_*^{2n+6}+1)\|\Delta W\|_{L^{\tilde{p}}}+c_6\|\nabla W\|_{L^\infty(Q_T)}\|\nabla U\|_{L^{\tilde{p}}(Q_T)}\\
&\leq c_7(r_*^{8(n+2)^3}+1).
\end{split}
\end{equation*}
Similarly, we further obtain 
\begin{equation}\label{c2a-3}
 \|\nabla  U\|_{L^\infty(Q_T)}\leq \| U\|_{ C^{1+\alpha_2, \frac{1+\alpha_2}{2}}(\overline{Q}_T)} \leq c_8 (r_*^{8(n+2)^3}+1),
\end{equation}
and
\begin{equation*}
\begin{split}
\| V\|_{W^{2,1}_{\tilde{p}}(Q_T)} 
&\leq \|-\chi_2\nabla \cdot(  V\nabla U)+A_2\sigma V+ \sigma\varepsilon  V^{n+1}(1- V)\|_{L^{\tilde{p}}(Q_T)}\\
&\leq c_9 \sup\limits_{t\in(0,T)}\|V\|_{L^\infty}^{n+2}+c_9\sup\limits_{t\in(0,T)}\|V\|_{L^\infty}\|\Delta U\|_{L^{\tilde{p}}(Q_T)}+c_9\|\nabla U\|_{L^\infty(Q_T)}\|\nabla V\|_{L^{\tilde{p}}(Q_T)}\\
&\leq c_{10} (r_*^{13(n+2)^4}+1),
\end{split}
\end{equation*}
which yields
\begin{equation}\label{c2a-4}
\|\nabla  V\|_{L^\infty(Q_T)}\leq \| V\|_{ C^{1+\alpha_3, \frac{1+\alpha_3}{2}}(\overline{Q}_T)} \leq c_{11} (r_*^{13(n+2)^4}+1).
\end{equation}
In the sequel, we denote by $c_i(r_*)$ generic constants which are increasing in $r_*$ and satisfy $c_i(r_*)\geq 1$. Thanks to \eqref{c2a-2}-\eqref{c2a-4}, and the assumption $\operatorname{(H)}$, we apply the Schauder estimate (see Lemma \ref{LT}) to the $W$-equation in \eqref{periodic3} and get
\begin{equation}\label{c2a-w}
\|W\|_{C^{2+\alpha_4, 1+\frac{\alpha_4}{2}}(\overline{Q}_T)}\leq c_{12}(r_*),
\end{equation}
where $\alpha_4:=\min\{\alpha_0,\alpha_1,\alpha_2,\alpha_3\}$. Now we rewrite the $U$-equation in \eqref{periodic3} as
\begin{equation}\label{U*H}
\begin{cases}
 U_t-\Delta  U+A_1U =f( x, t), &x\in\Omega, t>0,\\
{\nabla U}\cdot\nu=0, &x\in\partial\Omega, t>0, \\
 U(x,t)= U(x,t+T), &x\in\Omega, t\geq0,
\end{cases}
\end{equation}
where $$f( x, t):=-\chi_1\nabla U\cdot \nabla W-\chi_1 U\Delta W+A_1\sigma U+ \sigma\varepsilon  U^{n+1}(1- U).$$  
Due to \eqref{c2a-3} and \eqref{c2a-w}, it holds that
$$\|f(x,t)\|_{C^{\alpha_4, \frac{\alpha_4}{2}}(\overline{Q}_T)}\leq c_{13}(r_*).$$
Then applying Lemma \ref{LT} to \eqref{U*H} gives
\begin{equation}\label{c2a-u}
\|U\|_{C^{2+\alpha_4, 1+\frac{\alpha_4}{2}}(\overline{Q}_T)}\leq c_{14}(r_*).
\end{equation}
Similarly, by \eqref{c2a-4} and \eqref{c2a-u}, we apply Lemma \ref{LT} to the $V$-equation in \eqref{periodic3} and obtain 
\begin{equation}\label{c2a-v}
\|V\|_{C^{2+\alpha_4, 1+\frac{\alpha_4}{2}}(\overline{Q}_T)}\leq c_{15}(r_*).
\end{equation}
Then \eqref{c2a} follows from \eqref{c2a-w}, \eqref{c2a-u} and \eqref{c2a-v}.
\end{proof}

\vspace{2.5mm}
In what follows, using the topological degree method, we establish the existence of non-constant positive $T$-periodic solution $(U,V,W)$ to \eqref{periodic2}. To this end, we consider the following problem:
\begin{equation}\label{linearized_prob}
\begin{cases}
 U_t - \Delta  U + \chi_1 \nabla \cdot ( U \nabla  W) + A_1  U =  A_1\widetilde u_+ + \varepsilon \widetilde u_+^{n+1}(1- U)+\varepsilon\eta,  &x\in\Omega, \ t>0, \\
 V_t - \Delta  V + \chi_2 \nabla \cdot ( V \nabla  U) + A_2  V =  A_2\widetilde v_+ + \varepsilon \widetilde v_+^{{n+1}}(1- V)+\varepsilon\eta,  &x\in\Omega, \ t>0, \\
\nabla{ U}\cdot{\nu}=\nabla{V}\cdot{\nu}=0, &x\in\partial\Omega, \ t>0, \\
 U(x,t)= U(x,t+T), \  V(x,t)= V(x,t+T), &x\in\Omega,t\geq 0,
\end{cases}
\end{equation}
where $\eta \in [0, 1]$, $(\widetilde{u},\widetilde{v})\in \mathcal{X}$ are any given functions, $W\geq 0$ is the unique solution satisfying \eqref{lw}. Analogous to \eqref{lu} and \eqref{lv}, the system \eqref{linearized_prob} admits a unique nonnegative $T$-periodic solution $(U, V)\in\mathcal{X}$. 

Therefore, for any $(\widetilde{u}, \widetilde{v}) \in \mathcal{X}$ and $\eta \in [0, 1]$, we can define a new operator $\mathcal{H} : \mathcal{X} \times [0, 1] \to \mathcal{X}$ by 
$$\mathcal{H}((\widetilde{u}, \widetilde{v}), \eta) = ( U,  V).$$

Then we have the following results.
\begin{lemma}\label{uv>0}
Let $\operatorname{(H)}$ hold and $(U,V,W)$ be the solution of $\mathcal{H}(( U,  V), \eta)=(U,V)$. Then it holds that
\begin{itemize}
\item[(i)] For any $\eta\in [0,1]$, the solution $(U,V,W)$ is nonnegative. Moreover,  there exists a positive constant $M_0\geq1$ increasing in $r_*$ but independent of $\varepsilon$ and $\eta$ such that
\begin{equation*}
\|(U,V,W)\|_{C^{2+\alpha, 1+\frac{\alpha}{2}}(\overline{Q}_T)} \leq M_0.
\end{equation*}
\item[(ii)] For any $\eta\in(0,1]$, the solution $(U,V)$ satisfies
\begin{equation}\label{1/2}
\| U\|_{C^{\alpha,\frac{\alpha}{2}}(\overline{Q}_T)}>\frac{1}{2},\  \ \text{and}\ \  \| V\|_{C^{\alpha,\frac{\alpha}{2}}(\overline{Q}_T)}>\frac{1}{2};
\end{equation}
\item [(iii)] If $\eta=0$, $\| U\|_{C^{\alpha,\frac{\alpha}{2}}(\overline{Q}_T)}\leq \frac{1}{2}$ $(\text{resp.}\  \| V\|_{C^{\alpha,\frac{\alpha}{2}}(\overline{Q}_T)}\leq \frac{1}{2})$ holds if and only if $U\equiv 0$ $(\text{resp.}\ V\equiv0)$.
\end{itemize}
\end{lemma}

\begin{proof}
By the same arguments as in the proof of Lemma \ref{W21p_global}, we can prove the statement (i) directly. The details of proof are omitted for convenience.

Next, we shall prove the statement (ii) by contradiction. Suppose that the statement (ii) does not hold, then for some $\eta\in (0,1]$, there exists a solution $( U, V)$ satisfying
\begin{equation}\label{1/2-1}
\| U\|_{C^{\alpha,\frac{\alpha}{2}}(\overline{Q}_T)} \leq \frac{1}{2}\ \ \ \text{or}\ \ \ \| V\|_{C^{\alpha,\frac{\alpha}{2}}(\overline{Q}_T)}\leq \frac{1}{2}.
\end{equation}
If $\| U\|_{C^{\alpha,\frac{\alpha}{2}}(\overline{Q}_T)} \leq \frac{1}{2}$, we integrate $U$-equation of $\mathcal{H}(( U,  V), \eta)=(U,V)$, then apply \eqref{1/2-1} and the non-negativity of $ U$ to get 
\begin{equation*}
0\leq \frac{1}{2}\int_0^T\int_\Omega  U^{n+1} \leq \int_0^T\int_\Omega  U^{n+1}(1- U)=-\eta |\Omega|T<0,
\end{equation*}
which is a contradiction. Similarly,  if $\| V\|_{C^{\alpha,\frac{\alpha}{2}}(\overline{Q}_T)}\leq \frac{1}{2}$, we derive 
\begin{equation*}
0\leq \frac{1}{2}\int_0^T\int_\Omega  V^{n+1} \leq \int_0^T\int_\Omega  V^{n+1}(1- V)=-\eta |\Omega|T<0,
\end{equation*}
which  yields a contradiction again. Thus, \eqref{1/2} follows.

When $\eta=0$, if $\| U\|_{C^{\alpha,\frac{\alpha}{2}}(\overline{Q}_T)}\leq \frac{1}{2} ~(\text{resp.}~\| V\|_{C^{\alpha,\frac{\alpha}{2}}(\overline{Q}_T)}\leq \frac{1}{2})$, we similarly derive
$$
\int_0^T\int_\Omega  U^{n+1} \leq 0  \quad \bigg(\text{resp.} \quad \int_0^T\int_\Omega  V^{n+1} \leq 0\bigg), 
$$
which along with the facts $U\geq 0\ (\text{resp.}\ V\geq0)$ yields $U\equiv0\ (\text{resp.}\ V\equiv0)$. The converse is trivial. This proves the statement (iii). Thus, we complete the proof of Lemma \ref{uv>0}.
\end{proof}

\vspace{1.5mm}
With the preceding results, we now establish the existence of nonnegative  $T$-periodic solutions for the system \eqref{periodic2}.
\begin{proposition}\label{e-p2}
Let the conditions in Theorem \ref{GBSp} hold. Then the system \eqref{periodic2} admits at least one nonnegative $T$-periodic solution $( U_\varepsilon,  V_\varepsilon,  W_\varepsilon)$ satisfying
\begin{equation}\label{limit1}
\|(U_\varepsilon, V_\varepsilon, W_\varepsilon)(\cdot,t)\|_{C^{2+\alpha,1+\frac{\alpha}{2}}(\overline{\Omega}\times[0,\infty))}\le M,      
\end{equation}
where the constant $M>0$ is independent of $\varepsilon$ and increases in $r_*$.
\end{proposition}
\begin{proof}  
    We employ the topological degree theory to prove the existence of positive $T$-periodic solutions to \eqref{periodic2}. For any constant $R>0$, let $B_R$ denote the open ball in $\mathcal{X}$ defined by
\begin{equation*}
    B_R:=\big\{(U,V)\in \mathcal{X}\ \big|\ \|U\|_{C^{\alpha,\frac{\alpha}{2}}(\overline{Q}_T)} + \|V\|_{C^{\alpha,\frac{\alpha}{2}}(\overline{Q}_T)} < R\big\}.
\end{equation*}
Let $\widetilde{M}:= \max\{M, M_0\}\geq 1$. Moreover, for any $r>0$, we introduce an open subset $D_r\subset B_{\widetilde{M}+1}$ defined by
\begin{equation*}
    D_r:=\big\{(U,V)\in B_{\widetilde{M}+1}\ \big |\  \min\big\{\|U\|_{C^{\alpha,\frac{\alpha}{2}}(\overline{Q}_T)},\|V\|_{C^{\alpha,\frac{\alpha}{2}}(\overline{Q}_T)}\big\}<r\big\}.
\end{equation*}
Then from  Lemma \ref{W21p_global}, we get that any possible solution to $\mathcal{F}(( U,  V), \sigma)= ( U,  V)$ satisfies
\begin{equation*}
\|U\|_{C^{\alpha, \frac{\alpha}{2}}(\overline{Q}_T)}+\|V\|_{C^{\alpha, \frac{\alpha}{2}}(\overline{Q}_T)}\leq M  < \widetilde{M}+1,
\end{equation*}
which implies that there are no solutions on $\partial B_{\widetilde{M}+1}$. Thus, 
\begin{equation}\label{boundary_cond}
    0 \notin [\mathcal{I} - \mathcal{F}(\cdot, \sigma)](\partial B_{\widetilde{M}+1}), \quad \forall \sigma \in [0, 1],
\end{equation}
where $\mathcal{I}$ represents the identity operator. Then \eqref{boundary_cond} ensures that the topological degree $\deg(\mathcal{I} - \mathcal{F}(\cdot, \sigma), B_{\widetilde{M}+1}, 0)$ is well-defined. Using the homotopy invariance of the topological degree, and recalling $\mathcal{F}(\cdot, 0) = (0,0)$ (see Lemma \ref{C12}), we derive 
\begin{equation}\label{deg_large}
    \deg(\mathcal{I} - \mathcal{F}(\cdot, 1), B_{\widetilde{M}+1}, 0) = \deg(\mathcal{I} - \mathcal{F}(\cdot, 0), B_{\widetilde{M}+1}, 0) = \deg(\mathcal{I}, B_{\widetilde{M}+1}, 0) = 1.
\end{equation}
From Lemma \ref{uv>0}, we know the equation $\mathcal{H}(( U,  V), \eta) = ( U,  V)$ admits no solution on $\partial D_{1/2}$ and hence
\begin{equation}\label{boundary_H}
    0 \notin [\mathcal{I} - \mathcal{H}(\cdot, \eta)](\partial D_{1/2}), \quad \forall \eta \in [0, 1].
\end{equation}
Noting the fact $\mathcal{H}(\cdot, 0) = \mathcal{F}(\cdot, 1)$, it follows from \eqref{boundary_H} that
\begin{equation*}
    0 \notin [\mathcal{I} - \mathcal{F}(\cdot, 1)](\partial D_{1/2}).
\end{equation*}
Consequently, $\deg(\mathcal{I} - \mathcal{H}(\cdot, \eta), D_{1/2}, 0)$ and $\deg(\mathcal{I} - \mathcal{F}(\cdot, 1), D_{1/2}, 0)$ are well-defined. In view of \eqref{1/2}, the equation $\mathcal{H}(( U,  V), 1) = ( U,  V)$ admits no solution in $D_{1/2}$, which implies
\begin{equation*}
    \deg(\mathcal{I} - \mathcal{H}(\cdot, 1), D_{1/2}, 0) = 0.
\end{equation*}
By the homotopy invariance and the fact that $\mathcal{H}(\cdot, 0) = \mathcal{F}(\cdot, 1)$, we obtain
\begin{equation}\label{deg_small}
    \deg(\mathcal{I} - \mathcal{F}(\cdot, 1), D_{1/2}, 0) = \deg(\mathcal{I} - \mathcal{H}(\cdot, 0), D_{1/2}, 0) = \deg(\mathcal{I} - \mathcal{H}(\cdot, 1), D_{1/2}, 0) = 0.
\end{equation}
Using the excision property of the topological degree together with \eqref{deg_large} and \eqref{deg_small} gives
\begin{equation*}
    \deg(\mathcal{I} - \mathcal{F}(\cdot, 1), B_{\widetilde{M}+1} \setminus \overline{D}_{1/2}, 0) = \deg(\mathcal{I} - \mathcal{F}(\cdot, 1), B_{\widetilde{M}+1}, 0) - \deg(\mathcal{I} - \mathcal{F}(\cdot, 1), D_{1/2}, 0) = 1.
\end{equation*}
Therefore, in view of \eqref{UVW}, for any $\varepsilon \in (0, 1)$, the equation $$\mathcal{F}(( U,  V), 1) = ( U,  V),$$
which is equivalent to the system \eqref{periodic2}, admits at least one nonnegative solution $( U_\varepsilon,  V_\varepsilon,  W_\varepsilon)$ in $B_{\widetilde{M}+1} \setminus \overline{D}_{1/2}$ such that
\begin{equation}\label{limit2}
    \frac{1}{2} \leq \min\big\{\|U_\varepsilon\|_{C^{\alpha,\frac{\alpha}{2}}(\overline{Q}_T)},\|V_\varepsilon\|_{C^{\alpha,\frac{\alpha}{2}}(\overline{Q}_T)}\big\} \quad \text{and} \quad \|U_\varepsilon\|_{C^{\alpha,\frac{\alpha}{2}}(\overline{Q}_T)}+\|V_\varepsilon\|_{C^{\alpha,\frac{\alpha}{2}}(\overline{Q}_T)} \leq \widetilde{M},
\end{equation}
which yields $U_\varepsilon\not\equiv0 $ and $ V_\varepsilon\not\equiv0$. Furthermore, Lemma \ref{W21p_global} yields
\begin{equation*}
    \|(U_\varepsilon,V_\varepsilon,W_\varepsilon)\|_{C^{2+\alpha, 1+\frac{\alpha}{2}}(\overline{Q}_T)} \leq M,
\end{equation*}
where the constant $M\geq 1$ is independent of $\varepsilon$ and increasing in $r_*$. This completes the proof of Proposition \ref{e-p2}.
\end{proof}

Based on Proposition \ref{e-p2}, we prove the existence of non-constant positive $T$-periodic solutions for the system \eqref{periodic}.
\begin{proof}[{\bf Proof of Theorem \ref{GBSp}}]
We divide the proof into two steps.

{\bf Step 1: Existence and positivity.}  By Proposition \ref{e-p2} and the compact embedding $$C^{2+\alpha, 1+\frac{\alpha}{2}}(\overline{Q}_T)\hookrightarrow C^{2, 1}(\overline{Q}_T),$$
there exist a subsequence (still denoted by $\varepsilon$) and a limit nonnegative function $( U,  V,  W)\in C^{2,1}(\overline{Q}_T)$ such that as $\varepsilon \to 0$, one has
\allowdisplaybreaks
\begin{align*}
    ( U_\varepsilon,  V_\varepsilon,  W_\varepsilon) &\to ( U,  V,  W), \quad &&\text{uniformly in } \overline{Q}_T, \\
    (\nabla  U_\varepsilon, \nabla  V_\varepsilon, \nabla  W_\varepsilon) &\to (\nabla  U, \nabla  V, \nabla  W), \quad &&\text{uniformly in } \overline{Q}_T, \\
    (\Delta  U_\varepsilon, \Delta  V_\varepsilon, \Delta  W_\varepsilon) &\to (\Delta  U, \Delta  V, \Delta  W), \quad &&\text{uniformly in } \overline{Q}_T, \\
    (\partial_t  U_\varepsilon, \partial_t  V_\varepsilon, \partial_t  W_\varepsilon) &\to ( U_t,  V_t,  W_t), \quad &&\text{uniformly in } \overline{Q}_T.
\end{align*}
Therefore, passing to the limit in the approximate system shows that $(U,V,W)\in C^{2,1}(\overline{Q}_T)$ is a nonnegative classical $T$-periodic solution of \eqref{periodic}. Moreover, \eqref{limit1} and the lower semicontinuity of the H\"{o}lder norm imply
\begin{equation*}
    \|(U,V,W)\|_{C^{2+\alpha,1+\frac{\alpha}{2}}(\overline{Q}_T)}\leq M.
\end{equation*}
Furthermore, passing to the limit in \eqref{limit2}, we obtain
$$\min\left\{
\|U\|_{C^{\alpha,\frac{\alpha}{2}}(\overline{Q}_T)},
\|V\|_{C^{\alpha,\frac{\alpha}{2}}(\overline{Q}_T)}
\right\}\geq\frac12.$$
which implies $U\not\equiv0$ and $V\not \equiv 0$. Then by the strong maximum principle, we conclude $ U,  V,  W > 0$ in $\Omega \times [0, T]$.

{\bf Step 2: Spatial inhomogeneity.} Let $(U,V,W)$ be a positive $T$-periodic solution of \eqref{periodic}. We claim that $U, V, W$ are spatially inhomogeneous. To prove this, we integrate the first two equations of \eqref{periodic} over $\Omega$ to obtain
\begin{equation*}
    \frac{d}{dt}\int_\Omega U = 0 \quad \text{and} \quad \frac{d}{dt}\int_\Omega V = 0.
\end{equation*}
Hence, if $U$ is spatially homogeneous, it must be a positive constant. The same conclusion holds for $V$.

With this observation, suppose, for contradiction, that at least one component is spatially homogeneous, i.e., $\nabla U \equiv {\bf0}$, $\nabla V \equiv {\bf0}$, or $\nabla W \equiv {\bf0}$. We now consider the three cases:

\textbf{Case 1:} If $\nabla U \equiv {\bf0}$, then $U \equiv P_1$ for some constant $P_1 > 0$. The first equation of \eqref{periodic} reduces to 
$$ \Delta W = 0,$$
which, combined with $\nabla W\cdot\nu\big|_{\partial\Omega}=0$, gives 
$$\int_\Omega |\nabla W|^2=0.$$
This implies $W\equiv P(t)$ for some positive $T$-periodic solution. Substituting $U \equiv P_1$ into the second equation of \eqref{periodic} yields $V_t = \Delta V$. Multiplying this equality by $V$ and integrating the results by parts over $\Omega$, we obtain
\begin{equation*}
    \frac{1}{2}\frac{d}{dt}\int_\Omega V^2 = -\int_\Omega |\nabla V|^2\leq 0,
\end{equation*}
which, along with the fact $V(x, t) = V(x, t+T)$, gives
$$ \int_t^{t+T}\int_\Omega |\nabla V|^2=0,$$
and hence 
\begin{equation*}
    \frac{d}{dt}\int_\Omega V^2=\int_\Omega |\nabla V|^2 = 0.
\end{equation*}
This indicates that $V \equiv P_2$ for some constant $P_2 > 0$.

\textbf{Case 2:} If $\nabla V \equiv {\bf0}$, then $V \equiv P_2 > 0$. We get from the second equation of \eqref{periodic} that $\Delta U = 0$, leading to $U \equiv P_1 > 0$. This subsequently forces $\Delta W = 0$, and thus $W = P(t)$.

\textbf{Case 3:} If $\nabla W \equiv {\bf0}$, then $W = P(t)$. The first equation of \eqref{periodic} reduces to $U_t = \Delta U$. By the same energy estimate and $T$-periodicity used in Case 1, we obtain $\nabla U \equiv {\bf 0}$, hence $U \equiv P_1>0$. This in turn forces $V_t = \Delta V$, yielding $V \equiv P_2 > 0$.

In all cases, we obtain $(U,V,W) \equiv (P_1, P_2, P(t))$. Substituting this into the third equation of \eqref{periodic} yields
\begin{equation*}
    r(x,t) =  P'(t) + [\lambda(P_1 + P_2) + \mu]P(t).
\end{equation*}
This indicates that $r(x,t)$ is independent of $x$, which contradicts the hypothesis (H) that $r(x,t)$ is spatially heterogeneous. Thus, the components $U, V, W$ are spatially inhomogeneous.

{\bf Step 3: Temporal inhomogeneity.} Assume that $U(x,t)=U(x), V(x,t)=V(x)$ and $W(x,t)=W(x)$. Then the $W$-equation implies
$$ r(x,t)=-d\Delta W(x)+[\lambda(U(x)+V(x))+\mu]W(x),$$
which shows that $r(x,t)$ is independent of $t$. Hence, this contradicts the temporal inhomogeneity of $r(x,t)$ in the hypothesis (H). Thus, $(U, V, W)$ cannot be stationary.

Then, the combination of Step 1, Step 2 with Step 3 finishes the proof of Theorem \ref{GBSp}.
\end{proof}

\subsection{Uniqueness and global asymptotic stability}
In this subsection, we investigate the uniqueness and global asymptotic  stability  of positive non-constant $T$-periodic solution $( U,  V,  W)$  for the system \eqref{periodic} under some smallness assumptions on $r_*$. Without loss of generality, we assume $0<r_*\leq 1$ in the following. Then  from \eqref{ui0}, \eqref{vi0} and Lemma \ref{W21p_global}, there exists a constant $K_1>0$ independent of $r_*$ such that 
\begin{equation}\label{UWB}
\|U(\cdot,t)\|_{L^\infty}+\|V(\cdot,t)\|_{L^\infty}+\|\Delta  W(\cdot, t)\|_{L^\infty}\leq K_1.
\end{equation}
To prove the uniqueness and stability, we first establish an $r_*$-dependent upper bound for $\|\nabla U\|_{L^\infty}$.
\begin{lemma}\label{grad_u_decay}
Let $(U,V,W)$ be the positive non-constant $T$-periodic solution to \eqref{periodic} obtained in Theorem \ref{GBSp}. Then there exists a constant $C_{11} > 0$ independent of $r_*$ such that 
\begin{equation}\label{grad_u_rate}
\sup_{t \in (0, T)} \|\nabla  U(\cdot, t)\|_{L^\infty} \leq C_{11} r_*^{\frac{1}{n+2}}.
\end{equation}
\end{lemma}
\begin{proof}
Multiplying the first equation of \eqref{periodic} by $U$, integrating the results by parts and using Young's inequality and \eqref{UWB}, we obtain
\begin{equation}\label{u_L2_est}
\frac{1}{2}\frac{d}{dt}\int_\Omega U^2 + \frac{1}{2}\int_\Omega |\nabla U|^2 \leq \frac{\chi_1^2}{2} \| U\|_{L^\infty}^2 \int_\Omega |\nabla  W|^2\leq \frac{\chi_1^2K_1^2}{2}\int_\Omega |\nabla  W|^2.
\end{equation}
Using \eqref{wbd}, \eqref{UWB} and the Gagliardo-Nirenberg inequality derives 
\begin{equation}\label{lwrbd}
\|\nabla  W(\cdot, t)\|_{L^\infty} \leq c_1( \| W(\cdot, t)\|_{L^\infty}^{\frac{1}{2}} \|\Delta  W(\cdot, t)\|_{L^\infty}^{\frac{1}{2}} + \| W(\cdot, t)\|_{L^\infty}) \leq c_2 r_*^{\frac{1}{2}}.   
\end{equation}
Substituting \eqref{lwrbd} into \eqref{u_L2_est}, one gets
\begin{equation}\label{ul2est}
\frac{d}{dt}\int_\Omega U^2 + \int_\Omega |\nabla U|^2 \leq\chi_1^2 K_1^2 |\Omega| c_2^2 r_*.   
\end{equation}
Integrating \eqref{ul2est} over $(0,T)$ directly yields
\begin{equation}\label{grad_u_int}
\int_0^T \int_\Omega |\nabla  U|^2  \leq \chi_1^2 K_1^2 |\Omega|T c_2^2 r_*=: c_3r_*,
\end{equation}
and hence there exists some $t_0\in [0,T]$  such that
\begin{equation}\label{grad_u_int*}
\int_\Omega |\nabla U(\cdot,t_0)|^2\leq \frac{c_3}{T}r_*.
\end{equation}

Next, we multiply the first equation of \eqref{periodic} by $-\Delta  U$ and integrate the results by parts to obtain
\begin{equation*}
\begin{aligned}
\frac{1}{2}\frac{d}{dt}\int_\Omega |\nabla  U|^2 + \int_\Omega |\Delta  U|^2 &= \chi_1 \int_\Omega \nabla \cdot ( U \nabla  W) \Delta  U \\
&\leq \frac{1}{2}\int_\Omega |\Delta  U|^2 + \frac{\chi_1^2}{2} \int_\Omega \left( |\nabla  U|^2 |\nabla  W|^2 +  U^2 |\Delta  W|^2 \right),
\end{aligned}
\end{equation*}
which, together with $\| U\|_{L^\infty} \leq K_1$ in \eqref{UWB} and \eqref{lwrbd}, gives
\begin{equation}\label{grad_u_H1}
\begin{split}
\frac{d}{dt}\int_\Omega |\nabla  U|^2 
&\leq \chi_1^2 \|\nabla  W\|_{L^\infty}^2 \int_\Omega |\nabla  U|^2 + \chi_1^2 \| U\|_{L^\infty}^2 \int_\Omega |\Delta  W|^2\\
&\leq c_2^2\chi_1^2 r_* \int_\Omega |\nabla  U|^2+K_1^2\chi_1^2 \int_\Omega |\Delta  W|^2.
\end{split}
\end{equation}
Noting the second inequality in \eqref{c2a-1} and $0<r_*\leq 1$, we apply the H\"{o}lder inequality to get
\begin{equation}\label{grad_u_H1-1}
\int_0^T\int_\Omega |\Delta W|^2\leq \|\Delta W\|_{L^{\tilde{p}}(Q_T)}^2(|\Omega|T)^{\frac{\tilde{p}-2}{\tilde{p}}}\leq c_4 r_*.
\end{equation}
Integrating \eqref{grad_u_H1} over $(t_0,t)$ for any $0\leq t_0<t\leq T$, and using \eqref{grad_u_int}, \eqref{grad_u_int*} and \eqref{grad_u_H1-1}, one derives

\begin{equation}\label{grad_u_sup_L2}
\begin{split}
\sup_{t \in (0, T)} \|\nabla  U(\cdot, t)\|_{L^2}^2
&\leq
c_5r_*\int_0^T\int_\Omega |\nabla U|^2+\chi_1^2K_1^2\int_0^T\int_\Omega |\Delta W|^2+\int_\Omega |\nabla U(x,t_0)|^2\\
&\leq c_3c_5r_*^2+\chi_1^2K_1^2c_4r_*+\frac{c_3}{T}r_*\leq c_{6} r_*.
\end{split}
\end{equation}
Finally, by \eqref{c2a}, \eqref{grad_u_sup_L2} and the Gagliardo-Nirenberg inequality, we have
\begin{equation*}
\|\nabla  U\|_{L^\infty} \leq c_{7}( \|\Delta U\|_{L^\infty}^{\frac{n}{n+2}} \|\nabla  U\|_{L^2}^{\frac{2}{n+2}}+\|\nabla  U\|_{L^2})\leq c_{8}r_*^\frac{1}{n+2}.
\end{equation*}
This completes the proof of Lemma \ref{grad_u_decay}.
\end{proof}

With Lemma \ref{grad_u_decay}, we now prove the uniqueness and global stability for the positive $T$-periodic solution.
\begin{lemma}\label{uni1}
Let $(U,V,W)$ be the positive non-constant $T$-periodic solution to \eqref{periodic} obtained in Theorem \ref{GBSp} satisfying
\begin{equation}\label{TMC}
\int_\Omega U(\cdot,t)=M_U,\quad
\int_\Omega V(\cdot,t)=M_V,
\end{equation}
where $M_U$ and $M_V$ are two positive constants. Then there exists a small constant $r_1>0$ such that if $0<r_*\leq r_1$, the non-constant positive $T$-periodic solution $(U,V,W)(x,t)$ is unique subject to the mass constraints \eqref{TMC}.
\end{lemma}
\begin{proof}
 Let $(U_1,V_1,W_1)$ and $(U_2,V_2,W_2)$ be two solutions to \eqref{periodic} satisfying \eqref{TMC}, and denote $(\widetilde{U},\widetilde{V},\widetilde{W}):=(U_1-U_2,V_1-V_2,W_1-W_2)$. Then we have
\begin{equation}\label{error_s2}
\begin{cases}
\widetilde{U}_t = \Delta \widetilde{U} - \chi_1 \nabla \cdot (U_1 \nabla \widetilde{W} + \widetilde{U} \nabla W_2), & x \in \Omega, \ t > 0, \\
\widetilde{V}_t = \Delta \widetilde{V} - \chi_2 \nabla \cdot (V_1 \nabla \widetilde{U} + \widetilde{V} \nabla U_2), & x \in \Omega, \ t > 0, \\
\widetilde{W}_t = d \Delta \widetilde{W} - [\lambda(U_1+V_1) + \mu]\widetilde{W} - \lambda W_2(\widetilde{U} + \widetilde{V}), & x \in \Omega, \ t > 0, \\
\nabla\widetilde{U}\cdot{\nu} =\nabla \widetilde{V}\cdot \nu = \nabla \widetilde{W}\cdot{\nu} = 0, & x \in \partial \Omega, \ t > 0, \\
(\widetilde{U},\widetilde{V},\widetilde{W})(x,t) = (\widetilde{U},\widetilde{V},\widetilde{W})(x,t+T), & x \in \Omega,t\geq 0.
\end{cases}   
\end{equation}
We integrate the first two equations of \eqref{error_s2} over $\Omega$, then apply \eqref{TMC} to get
\begin{equation}\label{mass_bar}
\int_\Omega \widetilde{U}(\cdot,t) = \int_\Omega \widetilde{V}(\cdot,t) = 0, \quad \forall t \ge 0.
\end{equation}
By \eqref{mass_bar} and the Poincar\'e inequality, there exists a constant $c_1>0$ such that
\begin{equation}\label{poincare_bar}
\int_\Omega \widetilde{U}^2 \le c_1 \int_\Omega |\nabla \widetilde{U}|^2 \quad \text{and} \quad \int_\Omega \widetilde{V}^2 \le c_1 \int_\Omega |\nabla \widetilde{V}|^2.
\end{equation}
Multiplying the first equation in \eqref{error_s2} by $\widetilde{U}$, integrating the results by parts over $\Omega$, and then applying Young's inequality yield
\begin{equation*}
\begin{aligned}
\frac{1}{2}\frac{d}{dt}\int_\Omega \widetilde{U}^2 + \int_\Omega |\nabla \widetilde{U}|^2 
&= \chi_1 \int_\Omega U_1 \nabla \widetilde{W} \cdot \nabla \widetilde{U} + \chi_1 \int_\Omega \widetilde{U} \nabla W_2 \cdot \nabla \widetilde{U} \\
&\le \frac{1}{2} \int_\Omega |\nabla \widetilde{U}|^2 + \chi_1^2 \|U_1\|_{L^\infty}^2 \int_\Omega |\nabla \widetilde{W}|^2 + \chi_1^2 \|\nabla W_2\|_{L^\infty}^2 \int_\Omega \widetilde{U}^2,
\end{aligned}
\end{equation*}
which, together with 
 $\|U_1\|_{L^\infty} \le K_1$ (see \eqref{UWB}) and \eqref{poincare_bar}, gives
\begin{equation}\label{est_U_bar}
\frac{d}{dt}\int_\Omega \widetilde{U}^2 + \frac{1}{2c_1}\int_\Omega \widetilde{U}^2 + \frac{1}{2}\int_\Omega |\nabla \widetilde{U}|^2 \le 2\chi_1^2 K_1^2 \int_\Omega |\nabla \widetilde{W}|^2 + 2\chi_1^2 \|\nabla W_2\|_{L^\infty}^2 \int_\Omega \widetilde{U}^2.
\end{equation}
Similarly, we multiply the second equation of \eqref{error_s2} by $\widetilde{V}$, then use the facts $\|V_1\|_{L^\infty} \le K_1$ (see \eqref{UWB}) and \eqref{poincare_bar} to obtain
\begin{equation}\label{est_V_bar}
\frac{1}{2}\frac{d}{dt}\int_\Omega \widetilde{V}^2 + \frac{1}{2c_1}\int_\Omega \widetilde{V}^2 \le \chi_2^2 K_1^2 \int_\Omega |\nabla \widetilde{U}|^2 +\chi_2^2 \|\nabla U_2\|_{L^\infty}^2 \int_\Omega \widetilde{V}^2.
\end{equation}
We multiply the third equation of \eqref{error_s2} by $\widetilde{W}$, apply Young's inequality, and use the non-negativity of $U_1, V_1$ to obtain
\begin{equation}\label{est_W_bar}
\begin{aligned}
\frac{1}{2}\frac{d}{dt}\int_\Omega \widetilde{W}^2 + d \int_\Omega |\nabla \widetilde{W}|^2 + \mu \int_\Omega \widetilde{W}^2 
&\le -\int_\Omega \lambda(U_1+V_1)\widetilde{W}^2 - \lambda \int_\Omega W_2(\widetilde{U} + \widetilde{V})\widetilde{W} \\
&\le \frac{\mu}{2} \int_\Omega \widetilde{W}^2 + \frac{\lambda^2}{\mu} \|W_2\|_{L^\infty}^2 \left(\int_\Omega \widetilde{U}^2 + \int_\Omega \widetilde{V}^2\right).
\end{aligned}
\end{equation}

Multiplying \eqref{est_U_bar} by $c_2 := 4\chi_2^2 K_1^2 + 1$ and \eqref{est_W_bar} by $c_3 := \frac{2c_2 \chi_1^2 K_1^2}{d} + 1$, and adding the results to \eqref{est_V_bar}, then using \eqref{UVW}, \eqref{grad_u_rate} and \eqref{lwrbd}, one derives 
\begin{equation*}
\begin{aligned}
& \frac{1}{2}\frac{d}{dt}\int_\Omega (c_2\widetilde{U}^2 + \widetilde{V}^2 + c_3\widetilde{W}^2) + \frac{c_2}{4c_1}\int_\Omega\widetilde{U}^2 + \frac{1}{2c_1}\int_\Omega \widetilde{V}^2 + \frac{c_3 \mu}{2} \int_\Omega \widetilde{W}^2 \\
& \le \left( c_2 \chi_1^2 \|\nabla W_2\|_{L^\infty}^2 + c_3 \frac{\lambda^2}{\mu} \|W_2\|_{L^\infty}^2 \right) \int_\Omega \widetilde{U}^2  + \left( \chi_2^2 \|\nabla U_2\|_{L^\infty}^2 + c_3 \frac{\lambda^2}{\mu} \|W_2\|_{L^\infty}^2 \right) \int_\Omega \widetilde{V}^2 \\
& \le (c_4 r_*+c_5 r_*^2 )\int_\Omega \widetilde{U}^2+(c_6r_*^{\frac{2}{n+2}}+c_7r_*^2)\int_\Omega \widetilde{V}^2,
\end{aligned}
\end{equation*}
which gives
\begin{equation}\label{G1}
\begin{split}
 &\frac{1}{2}\frac{d}{dt}\int_\Omega (c_2\widetilde{U}^2 + \widetilde{V}^2 + c_3\widetilde{W}^2) + \left(\frac{c_2}{4c_1}-c_4 r_*-c_5 r_*^2\right)\int_\Omega\widetilde{U}^2
 \\ 
 &+ \left(\frac{1}{2c_1}-c_6r_*^{\frac{2}{n+2}}-c_7r_*^2\right)\int_\Omega \widetilde{V}^2 + \frac{c_3 \mu}{2} \int_\Omega \widetilde{W}^2\leq 0.\\
 \end{split}
\end{equation}
Since above constants $c_i~(i=1,2,\cdots,7)$ are independent of $r_*$, then there exists a constant $r_1>0$ such that for any $r_*<r_1$, we have
\begin{equation*}
\frac{c_2}{4c_1}-c_4 r_*-c_5 r_*^2\geq \frac{c_2}{8c_1}\ \ \mathrm{and} \ \ \frac{1}{2c_1}-c_6r_*^{\frac{2}{n+2}}-c_7r_*^2\geq \frac{1}{4c_1},
\end{equation*}
which, substituted into \eqref{G1}, gives
\begin{equation}\label{est_diff_decay*}
\frac{d}{dt}\int_\Omega (c_2\widetilde{U}^2 + \widetilde{V}^2 +  c_3\widetilde{W}^2) + \frac{c_2}{4c_1}\int_\Omega \widetilde{U}^2 +\frac{1}{2c_1}\int_\Omega \widetilde{V}^2 + c_3\mu \int_\Omega \widetilde{W}^2 \leq 0.
\end{equation}
Let $\tilde{E}_{1}(t):=\int_\Omega (c_2\widetilde{U}^2 + \widetilde{V}^2 + c_3\widetilde{W}^2)$, and $c_8:=\min\{\frac{1}{2c_1},\mu\}$, then \eqref{est_diff_decay*} yields that
\begin{equation}\label{est_diff_decay}
\frac{d}{dt}\tilde{E}_{1}(t) + c_8 \tilde{E}_{1}(t) \leq 0, \quad \forall t \geq 0.
\end{equation}
Integrating \eqref{est_diff_decay} over $(t,t+T)$, one has 
$$\int_t^{t+T}\int_\Omega (c_2\widetilde{U}^2 + \widetilde{V}^2 + c_3\widetilde{W}^2)\leq 0,$$
which entails $\widetilde{U} \equiv\widetilde{V} \equiv \widetilde{W}\equiv0$ in $\Omega \times (0,\infty)$. Hence, we obtain $U_1\equiv U_2, V_1\equiv V_2$ and $ W_1\equiv W_2$ in $\Omega \times (0,\infty)$.
This finishes the proof of uniqueness.
\end{proof}

\begin{proof}[{\bf Proof of Theorem \ref{GBSps}}]
The uniqueness directly follows from Lemma \ref{uni1}. In the sequel, we shall establish the global stability. By Lemma \ref{lem:holder}, we know that for $0<r_*\leq1$,
there exist positive constants $c_*$ and $\tilde{\alpha}$ independent of $r_*$ and $t$ such that
\begin{equation}\label{CS-B}
\|u\|_{{C^{\tilde{\alpha}}(\overline{\Omega})}}+\|v\|_{C^{\tilde{\alpha}}(\overline{\Omega})}\leq c_*,\ t\geq2.
\end{equation}
Denote $\widetilde{U}:= u - U$, $\widetilde{V}:= v - V$, and $\widetilde{W} = w - W$, then the combination of \eqref{periodic} and  \eqref{system} gives
\begin{equation}\label{error_system}
\begin{cases}
\widetilde{U}_t = \Delta \widetilde{U} - \chi_1 \nabla \cdot (u \nabla \widetilde{W} + \widetilde{U} \nabla W), & x \in \Omega, \ t > 0, \\
\widetilde{V}_t = \Delta \widetilde{V} - \chi_2 \nabla \cdot (v \nabla \widetilde{U} + \widetilde{V} \nabla U), & x \in \Omega, \ t > 0, \\
\widetilde{W}_t = d \Delta \widetilde{W} - [\lambda(u+v) + \mu]\widetilde{W} - \lambda W(\widetilde{U} + \widetilde{V}), & x \in \Omega, \ t > 0, \\
\nabla\widetilde{U}\cdot{\nu} =\nabla \widetilde{V}\cdot \nu = \nabla \widetilde{W}\cdot{\nu} = 0, & x \in \partial \Omega, \ t > 0, \\
\widetilde{U}(x,0) = u_0(x) - U(x,0), \quad \widetilde{V}(x,0) = v_0(x) - V(x,0), & x \in \Omega, \\
\widetilde{W}(x,0) = w_0(x) - W(x,0), & x \in \Omega.
\end{cases}
\end{equation}
By the mass conservation of \eqref{system} and \eqref{periodic} alongside \eqref{IC-A}, we deduce
\begin{equation}\label{ZMC}
 \int_\Omega \widetilde{U}(\cdot,t) = \int_\Omega \widetilde{V}(\cdot,t) = 0,\ \forall t \ge 0.
\end{equation}
Then noting \eqref{ZMC} and applying the Poincar\'e inequality, there exists a constant $c_1>0$ such that 
\begin{equation}\label{poincare}
\int_\Omega \widetilde{U}^2 \le c_1 \int_\Omega |\nabla \widetilde{U}|^2 \quad \text{and} \quad \int_\Omega \widetilde{V}^2 \le c_1 \int_\Omega |\nabla \widetilde{V}|^2.
\end{equation}
Multiplying the first equation of \eqref{error_system} by $\widetilde{U}$, and integrating the results by parts over $\Omega$, then applying Young's inequality and \eqref{CS-B}, we obtain
\begin{equation*}
\begin{aligned}
\frac{1}{2}\frac{d}{dt}\int_\Omega \widetilde{U}^2 + \int_\Omega |\nabla \widetilde{U}|^2 
&= \chi_1 \int_\Omega u \nabla \widetilde{W} \cdot \nabla \widetilde{U} + \chi_1 \int_\Omega \widetilde{U} \nabla W \cdot \nabla \widetilde{U} \\
&\le \frac{1}{2} \int_\Omega |\nabla \widetilde{U}|^2 + \chi_1^2 c_*^2 \int_\Omega |\nabla \widetilde{W}|^2 + \chi_1^2 \|\nabla W\|_{L^\infty}^2 \int_\Omega \widetilde{U}^2,
\end{aligned}
\end{equation*}
which together with \eqref{poincare} gives
\begin{equation}\label{est_phi}
\frac{d}{dt}\int_\Omega \widetilde{U}^2 + \frac{1}{2c_1}\int_\Omega \widetilde{U}^2 +\frac{1}{2}\int_\Omega |\nabla \widetilde{U}|^2\le 2\chi_1^2 c_*^2 \int_\Omega |\nabla \widetilde{W}|^2 + 2\chi_1^2 \|\nabla W\|_{L^\infty}^2 \int_\Omega \widetilde{U}^2.
\end{equation}
Similarly, we derive from the  second equation of \eqref{error_system} that 
\begin{equation}\label{est_psi}
\frac{1}{2}\frac{d}{dt}\int_\Omega \widetilde{V}^2 + \frac{1}{2c_1}\int_\Omega \widetilde{V}^2 \le \chi_2^2 c_*^2 \int_\Omega |\nabla \widetilde{U}|^2 + \chi_2^2 \|\nabla U\|_{L^\infty}^2 \int_\Omega \widetilde{V}^2.
\end{equation}
Next, multiplying the third equation of \eqref{error_system} by $\widetilde{W}$, and integrating the results over $\Omega$, then applying Young's inequality, one derives 
\begin{equation}\label{est_theta}
\frac{1}{2}\frac{d}{dt}\int_\Omega \widetilde{W}^2 + d \int_\Omega |\nabla \widetilde{W}|^2 + \frac{\mu}{2} \int_\Omega \widetilde{W}^2 \le \frac{\lambda^2}{\mu} \|W\|_{L^\infty}^2 \left(\int_\Omega \widetilde{U}^2 + \int_\Omega \widetilde{V}^2\right).
\end{equation}

Multiplying \eqref{est_phi} by $c_2:= 4\chi_2^2 c_*^2 + 1$ and \eqref{est_theta} by $c_3: = \frac{2c_2 \chi_1^2 c_*^2}{d} + 1$, and adding the results to \eqref{est_psi}, we obtain for any $r_*\in (0,1]$ that
\begin{equation}\label{est_total}
\begin{aligned}
& \frac{1}{2}\frac{d}{dt}\int_\Omega (c_2\widetilde{U}^2 + \widetilde{V}^2 + c_3\widetilde{W}^2) + \frac{c_2}{4c_1}\int_\Omega\widetilde{U}^2 + \frac{1}{2c_1}\int_\Omega \widetilde{V}^2 + \frac{c_3 \mu}{2} \int_\Omega \widetilde{W}^2\\
& \le \left( c_2 \chi_1^2 \|\nabla W\|_{L^\infty}^2 + c_3 \frac{\lambda^2}{\mu} \|W\|_{L^\infty}^2 \right) \int_\Omega \widetilde{U}^2  + \left( \chi_2^2 \|\nabla U\|_{L^\infty}^2 + c_3 \frac{\lambda^2}{\mu} \|W\|_{L^\infty}^2 \right) \int_\Omega \widetilde{V}^2 \\
& \le (c_{12} r_*+c_{13} r_*^2 )\int_\Omega \widetilde{U}^2+(c_{14}r_*^{\frac{2}{n+2}}+c_{15}r_*^2)\int_\Omega \widetilde{V}^2,
\end{aligned}
\end{equation}
where we have used \eqref{UVW}, \eqref{grad_u_rate} and \eqref{lwrbd}.
From \eqref{est_total}, there exists a constant $0<r_0\leq r_1$ such that for any $r_*<r_0$, it holds that
\begin{equation}\label{est_total_decay*}
\frac{1}{2}\frac{d}{dt}\int_\Omega (c_2\widetilde{U}^2 + \widetilde{V}^2 + c_3\widetilde{W}^2) + \frac{c_2}{8c_1}\int_\Omega \widetilde{U}^2 + \frac{1}{4c_1}\int_\Omega \widetilde{V}^2 + \frac{c_3 \mu}{2} \int_\Omega \widetilde{W}^2 \leq 0.
\end{equation}
Denote $E_2(t):= \int_\Omega (c_2\widetilde{U}^2 + \widetilde{V}^2 + c_3\widetilde{W}^2)$ and $c_4:=\min\big\{\frac{1}{2c_1},\mu\big\}$. We follow from \eqref{est_total_decay*} that
\begin{equation*}
\frac{d}{dt}E_2(t)+c_4 E_2(t)\leq 0,
\end{equation*}
which together with Gronwall's inequality gives
\begin{equation}\label{l2_stable_tau1}
\|\widetilde{U}(\cdot,t)\|_{L^2} + \|\widetilde{V}(\cdot,t)\|_{L^2} + \|\widetilde{W}(\cdot,t)\|_{L^2} \leq c_5 e^{-\frac{c_4}{2} t}.
\end{equation}
Let $\hat\alpha:=\min\{\alpha,\tilde{\alpha}\}$. Recalling the interpolation inequality (see \cite[(3.62)]{JinLW-CVPDE-2025}), 
\[
\|f\|_{L^\infty(\Omega)} \leq c_6 \|f\|_{C^{\hat\alpha}(\overline{\Omega})}^{\frac{n}{n+\hat\alpha}} \|f\|_{L^1(\Omega)}^{\frac{\hat\alpha}{n+\hat\alpha}}
\quad \text{for } f \in L^1(\Omega) \cap C^{\hat\alpha}(\overline{\Omega}),
\]
and combing this with \eqref{CS-B} and \eqref{l2_stable_tau1}, we obtain
\begin{equation}\label{uid}
\|\widetilde{U}(\cdot,t)\|_{L^\infty} \leq c_7 \|\widetilde{U}(\cdot,t)\|_{C^{\hat\alpha}(\overline{\Omega})}^{\frac{n}{n+\hat\alpha}} \|\widetilde{U}(\cdot,t)\|_{L^1}^{\frac{\hat\alpha}{n+\hat\alpha}} \leq c_8\|\widetilde{U}(\cdot,t)\|_{L^2}^{\frac{\hat\alpha}{n+\hat\alpha}}\leq c_9 e^{-c_{10} t},\ \forall t\geq2.
\end{equation}
Applying the same interpolation technique yields
\begin{equation*}
\|\widetilde{V}(\cdot,t)\|_{L^\infty} + \|\widetilde{W}(\cdot,t)\|_{L^\infty} \leq c_{11} e^{-c_{12}t},\forall t\geq2.
\end{equation*}
This together with \eqref{uid} completes the proof of Theorem \ref{GBSps}.
\end{proof}

\section{Homogeneous environments: proof of Theorem \ref{LHP} }\label{HP}
This subsection will prove the existence of time-periodic solutions of \eqref{system} under homogeneous environment via Hopf bifurcation theory. To facilitate the comparison between heterogeneous (i.e., $T$-periodic $r(x,t)$) and homogeneous (i.e., constant $r$) environments, we take $r$ as the bifurcation parameter.
\subsection{Linearized stability analysis} 
For the subsequent Hopf bifurcation analysis, we first perform a linear stability analysis with respect to the resource renewal rate  $r$. To this end, we linearize the system \eqref{system} at the constant steady state $(\bar{u}_0,\bar{v}_0,w_c)$ to obtain
\begin{equation*}
\begin{cases}
\Psi_t=\mathcal{A}\Delta\Psi+\mathcal{B}\Psi,&x\in\Omega,t>0,\\
\nabla\Psi\cdot\nu=0,&x\in\partial\Omega,t>0,\\
\int_\Omega \psi_1(x,t)dx=\int_\Omega\psi_2(x,t)dx=0,&x\in \Omega,t>0,\\
\Psi(x,0)=(u-\bar{u}_0,v-\bar{v}_0,w-w_c)^{\mathcal{T}}, &x\in\Omega,
\end{cases}
\end{equation*}
where $\mathcal{T}$ denotes the transpose, $\Psi=(\psi_1,\psi_2, \psi_3)^\mathcal{T}$ and
\begin{equation*}
	\Psi=(\psi_1,\psi_2,\psi_3)^\mathcal{T}:=\begin{pmatrix}
		u-\bar{u}_0\\[2mm]
		v-\bar{v}_0\\[2mm]
		w-w_c
	\end{pmatrix},\ 
	\mathcal{A}=\begin{pmatrix}
		1 & 0 & -\chi_1\bar{u}_0\\[2mm]
		-\chi_2\bar{v}_0 & 1 & 0\\[2mm]
		0 & 0 & d
	\end{pmatrix},\ 
	\mathcal{B}=\begin{pmatrix}
		0 & 0 & 0\\[2mm]
		0 & 0 & 0\\[2mm]
		-\lambda w_c & -\lambda w_c & B_{33}
	\end{pmatrix}
\end{equation*}
with $B_{33}:=-\lambda(\bar{u}_0+\bar{v}_0)-\mu$. Then the corresponding elliptic eigenvalue problem is given as below:
\begin{equation}\label{le}
\begin{cases}
\mathcal{A}\Delta\Phi+\mathcal{B}\Phi=\rho\Phi ,&x\in\Omega,t>0,\\
\nabla\Phi\cdot\nu=0,&x\in\partial\Omega,t>0,\\
\int_\Omega \phi(x)=\int_\Omega \psi(x)=0, & x\in\Omega,
\end{cases}
\end{equation}
where $\Phi:=(\phi(x),\psi(x),\varphi(x))^\mathcal{T}$. The linear stability of $(\bar{u}_0,\bar{v}_0,w_c)$ is determined by the eigenvalues of the matrix $(-\sigma_m\mathcal{A}+\mathcal{B})$, which satisfies the following characteristic equation:
\begin{equation}\label{ef}
	 \rho^3+A_1(r,\sigma_m) \rho^2+A_2(r,\sigma_m) \rho+A_3(r,\sigma_m)=0,
\end{equation}
where $A_i(r,\sigma_m)=:A_i$ ($i=1,2,3$) and
\begin{equation}\label{A123}
	\begin{aligned}
		A_1=&\sigma_m(2+d)+\lambda (\bar{u}_0+\bar{v}_0)+\mu>0,\\
		A_2=&\sigma_m^2(1+2d)+\sigma_m\left\{2[\lambda(\bar{u}_0+\bar{v}_0)+\mu]+\lambda \chi_1\bar{u}_0w_c\right\}>0,\\
         A_3=&\sigma_m^3d+\sigma_m^2[\lambda(\bar{u}_0+\bar{v}_0)+\mu+\chi_1\lambda\bar{u}_0w_c+\chi_1\chi_2\lambda\bar{u}_0\bar{v}_0w_c]>0.\\
    \end{aligned}
\end{equation}
Noting $A_1>0, A_3>0$ for all $m\in\mathbb{Z}^+$, using the Routh-Hurwitz criterion \cite[Appendix B.1]{Murray}, we know that  $(\bar{u}_0,\bar{v}_0,w_c)$ is linearly stable if and only if 
\begin{equation*}\label{criterion}
A_1A_2-A_3>0, \ \forall m\in\mathbb{Z}^+.
\end{equation*}
After some calculations, we derive that 
\begin{equation*}
	A_1A_2-A_3=\sigma_m (\sigma_m^2 B_1 + \sigma_m B_2 + B_3)=:\sigma_mP (r,\sigma_m),
\end{equation*}
where
\begin{equation}\label{P}
	\begin{aligned}
          &P(r,\sigma_m):=\sigma_m^2 B_1 + \sigma_m B_2 + B_3,\\
		&B_1:=2(d+1)^2>0,\\
        &B_2:=4 \left[ \lambda (\bar{u}_0 + \bar{v}_0) + \mu \right] (d+1) + \lambda \chi_1 \bar{u}_0 w_c \left( 1 + d - \chi_2 \bar{v}_0 \right),\\
        &B_3:=[ \lambda (\bar{u}_0 + \bar{v}_0) + \mu ]\{2[ \lambda (\bar{u}_0 + \bar{v}_0) + \mu ]+\lambda\chi_1\bar{u}_0w_c\}>0.
    \end{aligned}
\end{equation} 
Thus, the sign of $P(r,\sigma_m)$ completely governs stability/instability of $(\bar{u}_0,\bar{v}_0,w_c)$. More precisely, we have the following results. 
\begin{lemma}\label{criterion1}
The constant steady state $(\bar{u}_0,\bar{v}_0,w_c)$ is linearly stable if $P(r,\sigma_m)>0$ for all $ m\in \mathbb{Z}^+$, and linearly unstable if $P(r,\sigma_m)<0$ for some $m\in\mathbb{Z}^+$.
\end{lemma}

With Lemma \ref{criterion},  we shall establish stability results subject to appropriate constraints on the model parameters. 
\begin{lemma}\label{RE-LS}
The constant steady state $(\bar{u}_0,\bar{v}_0,w_c)$ is linearly stable, if one of the following condition holds
\begin{itemize}
\item[(a)] $\chi_2\bar{v}_0\leq d+1$;
\item [(b)]$\chi_2\bar{v}_0>d+1\ \ \text{with}\ \ r\leq \frac{4(d+1)[\lambda(\bar{u}_0+\bar{v}_0)+\mu]^2}{(\chi_2\bar{v}_0-d-1)\lambda\chi_1\bar{u}_0}.$
\end{itemize}
\end{lemma}
\begin{proof}
To prove Lemma \ref{RE-LS}, we first define
\begin{equation}\label{JI}
I:=\chi_2\bar{v}_0-d-1>0,\ \ J:=\lambda(\bar{u}_0+\bar{v}_0)+\mu, 
\end{equation}
then 
$$ B_2^2-4B_1B_3=\frac{\lambda \chi_1\bar{u}_0w_c}{J}(\lambda \chi_1 \bar{u}_0rI^2-8J^2(d+1)\chi_2\bar{v}_0)=: \frac{\lambda \chi_1\bar{u}_0w_c}{J}F(\chi_2,r).$$
Denote
\begin{equation*}
\chi_2^{\pm}(r):=\frac{d+1}{\bar{v}_0}
\Bigl(1+\frac{4J^2}{\lambda\chi_1\bar{u}_0r}
\pm\frac{2J}{\lambda\chi_1\bar{u}_0r}\sqrt{4J^2+2\lambda\chi_1\bar{u}_0r}\,\Bigr) >0, 
\end{equation*}
one can check that  $\chi_2^{\pm}(r)$ are solutions to $F(\chi_2,r)=0.$ After some calculations, we derive
\begin{equation}\label{DELTAIFF}
 B_2^2-4B_1B_3<0\ \ \text{iff}\ \chi_2^-<\chi_2<\chi_2^+ ,
\end{equation}
and
\begin{equation*}
B_2<0\ \  \text{iff}\ \ \chi_2>\frac{d+1}{\bar{v}_0}\big(\frac{4J^2}{\lambda\chi_1\bar{u}_0r}+1\big)=:\chi_2^*,  
\end{equation*}
as well as
\begin{equation}\label{CHI*}
\chi_2^-<\chi_2^*<\chi_2^+.
\end{equation}
Under the condition (a) or (b), calculating directly gives $\chi_2<\chi_2^+$. When $\chi_2\leq \chi_2^-$, \eqref{DELTAIFF}-\eqref{CHI*} imply 
$$B_2>0\ \  \text{and} \ \ B_2^2-4B_1B_3\geq0.$$
This along with the fact $B_2^2-4B_1B_3<0$ for $ \chi_2^-<\chi_2<\chi_2^+$ (see \eqref{DELTAIFF}) yields  that for $\chi_2<\chi_2^+$
\begin{equation}\label{PP}
P(r,\sigma_m)>0, \ \forall m\in \mathbb{Z}^+.
\end{equation}
Then \eqref{PP} together with Lemma \ref{criterion1} finishes the proof of Lemma \ref{RE-LS}.
\end{proof}

\vspace{1.5mm}
Hence, it remains to consider the left regime
\begin{equation}\label{C-HOPF}
\chi_2\bar{v}_0>d+1 \ \mathrm{and} \ \
r>\frac{4(d+1)[\lambda(\bar{u}_0+\bar{v}_0)+\mu]^2}{(\chi_2\bar{v}_0-d-1)\lambda\chi_1\bar{u}_0}.
\end{equation} 
Since $A_3>0$, the matrix $(-\sigma_m\mathcal{A}+\mathcal{B})$ has no zero eigenvalue, no steady state bifurcation occurs. Then the boundary between the stability and instability regime is therefore given by
$$P(r,\sigma_m)=0.$$
Solving $P(r,q)=0$ for $r$ yields $r=r^H(q),$ where $r^H(q)$ is defined in \eqref{rhm}. 

The following lemma gives the monotonicity of $r^H(q)$ with respect to $q>0$.
\begin{lemma}\label{ddx}
Let $\chi_1,\chi_2,\lambda,\mu,\bar{v}_0,\bar{v}_0$ be fixed, and set $q_0:=\frac{\lambda(\bar{u}_0+\bar{v}_0)+\mu}{\chi_2\bar{v}_0-d-1}$. Then for any $q>q_0$, $r^H(q)$ is strictly decreasing for $q\in(q_0,q_*)$, and is strictly increasing for $q\in(q_*,\infty)$ with $q_*$ defined in \eqref{q*}. Moreover, it holds that
\begin{equation}\label{gm}
\min\limits_{q>q_0}\{r^H(q)\}= r^H(q_*)=\frac{8J^2(d+1)\chi_2\bar{v}_0}{\lambda\chi_1\bar{u}_0I^2}=\frac{8(d+1)\chi_2\bar{v}_0[\lambda(\bar{u}_0+\bar{v}_0)+\mu]^2}{(\chi_2\bar{v}_0-d-1)^2\lambda\chi_1\bar{u}_0}.
\end{equation}
\end{lemma}
\begin{proof}
By the definitions of $J$ and $I$ in \eqref{JI}, we rewrite $r^H(q)$ as
$$r^H(q)=\frac{2J}{\lambda\chi_1\bar{u}_0}\frac{[(d+1)q+J]^2}{qI-J}=:\frac{2J}{\lambda\chi_1\bar{u}_0}f(q).$$
A simple calculation gives 
  $$f'(q)=\frac{[(d+1)q+J](d+1)I}{(qI-J)^2}\left(q-\frac{[2(d+1)+I]J}{(d+1)I}\right),$$
which immediately gives the monotonicity and the minimum.
\end{proof}

We now show the following stability/instability results.
\begin{lemma}\label{LSUS1*}
Let $\chi_1,\chi_2,\lambda,\mu,d,\bar{v}_0,\bar{u}_0$ be fixed, and $r_m^H:=r^H(\sigma_m)$ defined in \eqref{rhm}, the set $Q_0$ defined in \eqref{Q0*}. Assume \eqref{C-HOPF} holds. 
Then 
\begin{itemize}
\item[(i)]  $(\bar{u}_0,\bar{v}_0,w_c)$ is linearly stable if  $r<\min\limits_{m_0\in Q_0}\{r_{m_0}^H\}$;
\item[(ii)] $(\bar{u}_0,\bar{v}_0,w_c)$ is linearly unstable if $r>\min\limits_{m_0\in Q_0}\{r_{m_0}^H\}$.
\end{itemize}
\end{lemma}
\begin{proof}
We first show that $\min\limits_{m_0\in Q_0}\{r_{m_0}^H\}$ exists. Note that the sequence $\{\sigma_m\}_{m=1}^\infty$ is increasing with respect to $m$ and satisfies $$\sigma_1>0 \ \  \mathrm{and}\ \lim\limits_{m\rightarrow\infty}\sigma_m=\infty.$$ Then there exists a unique integer $m_*\in\mathbb{Z}^+$ satisfying $\sigma_{m_*}>q_0$ such that $q_*\in[\sigma_{m_*}, \sigma_{m_*+1}]$ or $q_*<\sigma_{m_*}$. Hence, we deduce from Lemma \ref{ddx} that
\begin{equation}\label{min}
\min\limits_{m_0\in Q_0}\{r_{m_0}^H\}=
\begin{cases}
\min\{r^H(\sigma_{m_*}),r^H(\sigma_{m_*+1})\},&\text{if}~\sigma_{m_*}\neq q_*~\text{and}~\sigma_{m_*+1}\neq q_*\\[1.5mm]
r^H(q_*),&\text{if}~\sigma_{m_*}=q_*~\text{or}~\sigma_{m_*+1}=q_*,
\end{cases}
\end{equation}
which yields $\min\limits_{m_0\in Q_0}\{r_{m_0}^H\}$ exists and $\min\limits_{m_0\in Q_0}\{r_{m_0}^H\}\geq r^H(q_*).$ Then we follow from \eqref{gm} that for each $m_0\in Q_0$, it holds that
\begin{equation}\label{basic}
 r_{m_0}^H\geq \min\limits_{m_0\in Q_0}\{r_{m_0}^H\}\geq r^H(q_*)>\frac{4(d+1)[\lambda(\bar{u}_0+\bar{v}_0)+\mu]^2}{(\chi_2\bar{v}_0-d-1)\lambda\chi_1\bar{u}_0}.
\end{equation}
Now we prove the linear stability of $(\bar{u}_0,\bar{v}_0,w_c)$ if  $r<\min\limits_{m_0\in Q_0}\{r_{m_0}^H\}$. Denote the set
$$S:=\mathbb{Z}^+\backslash Q_0\subset \mathbb{Z}^+.$$
If $S\ne\emptyset$, then for each $m\in S$, one has $\sigma_m\leq q_0$. The definition of $P(r,\sigma_m)$ in \eqref{P} along with $\chi_2\bar{v}_0>d+1$ in \eqref{C-HOPF} gives
\begin{equation}\label{PL0}
P(r,\sigma_m)=2[\sigma_m(d+1)+\lambda(\bar{u}_0+\bar{v}_0)+\mu]^2+\frac{r\lambda\chi_1\bar{u}_0}{\lambda(\bar{u}_0+\bar{v}_0)+\mu}(\chi_2\bar{v}_0-d-1)(q_0-\sigma_m)>0.
\end{equation}
Moreover, for each $m_0\in Q_0$, the condition $r<\min\limits_{m_0\in Q_0}\{r_{m_0}^H\}$ also shows \eqref{PL0} holds. 

If $S=\emptyset$, this means that for each $m\in\mathbb{Z}^+$, we have $\sigma_m>q_0.$  Similarly, we also get \eqref{PL0}. Therefore, Lemma \ref{criterion1} implies Lemma \ref{LSUS1*}-(i).

We next show the linear instability $(\bar{u}_0,\bar{v}_0,w_c)$ if  $r>\min\limits_{m_0\in Q_0}\{r_{m_0}^H\}$. By \eqref{min}, we know for some $1\leq m^*$ satisfying $\sigma_{m^*}>q_0$ it holds that 
$$\min\limits_{m_0\in Q_0}\{r_{m_0}^H\}=r^H(\sigma_{m^*}).$$ 
Then for all $r>r^H(\sigma_{m^*})$, we derive $P(r,\sigma_{m^*})<0$, which together with Lemma \ref{criterion1} finishes the proof of Lemma \ref{LSUS1*}-(ii). In fact,  as long as $r>r_{m_0}^H$ for some $m_0\in Q_0$, then 
$$r>r_{m_0}^H\geq r^H(\sigma_{m^*}),$$ 
and one has $P(r,\sigma_{m_0})<0$, which also yields the linear instability of $(\bar{u}_0,\bar{v}_0,w_c)$.
\end{proof}

\subsection{Hopf bifurcation} In this subsection, we prove the existence of periodic solutions via Hopf bifurcation theory. We first identify the potential Hopf bifurcation points.
\begin{lemma}\label{imaginary}
Let $\chi_1,\chi_2,\lambda,\mu,d,\bar{v}_0,\bar{u}_0$ be fixed, and \eqref{C-HOPF} holds. Then the matrix $-\sigma_m\mathcal{A}+\mathcal{B}$ admits a pair of purely imaginary eigenvalues $ \rho=\pm i\xi$ if and only if  $r=r_{m_0}^H$ for some $m_0\in Q_0$. 
\end{lemma}
\begin{proof}
Suppose that \eqref{ef} has eigenvalues $-i\xi$, $i\xi$ and $\eta$ with $\xi,\eta\in\mathbb{R}\setminus\{0\}$. Then the Routh-Hurwitz criterion yields
$$ A_1=-( \rho_1+ \rho_2+ \rho_3)=-\eta>0,\ A_2= \rho_1 \rho_2+ \rho_1 \rho_3+ \rho_2 \rho_3=\xi^2>0,\ A_3=- \rho_1 \rho_2 \rho_3=-\xi^2\eta>0,$$
which implies $\eta<0$ and $A_1A_2-A_3=0$. This along with $m\in\mathbb{Z}^+$ gives
\begin{equation*}
P(r,\sigma_m)=2[\sigma_m(d+1)+\lambda(\bar{u}_0+\bar{v}_0)+\mu]^2+\frac{r\lambda\chi_1\bar{u}_0}{\lambda(\bar{u}_0+\bar{v}_0)+\mu}(\chi_2\bar{v}_0-d-1)(q_0-\sigma_m)=0,
\end{equation*}
which implies $r=r_{m_0}^H$ for $m_0\in Q_0$.

On the other hand, if $r=r_{m_0}^H$ for some $m_0\in Q_0$, then we obtain $A_1A_2-A_3=0$ for some $m_0\in Q_0$. Thus \eqref{ef} can be rewritten as
$$( \rho+A_1)( \rho^2+A_2)=0,$$
which gives 
\begin{equation}\label{IM}
 \rho_1=-A_1<0,\ \  \rho_2=\pm i\sqrt{A_2}.
\end{equation}
Moreover, there is no eigenvalue with the form $ki\sqrt{A_2}$ for $k\in \mathbb{Z}\backslash\{\pm1\}$. The proof of Lemma \ref{imaginary} is complete.
\end{proof}

The following lemma verifies the transversality condition at the potential Hopf bifurcation point $r_{m_0}^H$.
\begin{lemma}\label{TC}
Let the assumptions in Lemma \ref{LSUS1*} hold, and $m_0\in Q_0$ with $Q_0$ defining in \eqref{Q0*}. Then for $r$ near $r_{m_0}^H$, the matrix $-\sigma_m\mathcal{A}+\mathcal{B}$ admits a unique eigenvalue $\operatorname{Re}[ \rho(r)]+i\operatorname{Im}[ \rho(r)]$ such that 
$$\operatorname{Re}[ \rho(r_{m_0}^H)]=0,\ \operatorname{Im}[ \rho(r_{m_0}^H)]>0,~\text{and}~ \frac{d\operatorname{Re}[ \rho(r)]}{dr}\bigg|_{r=r_{m_0}^H}>0.$$
\end{lemma}
\begin{proof}
By Lemma \ref{imaginary} and the continuous dependence of eigenvalues on the parameter $r$, there exists a complex eigenvalue of the form $\sigma(r)+i\xi(r)$ satisfying $$\sigma(r_{m_0}^H)=0\ \ \text{and}\ \ \xi(r_{m_0}^H)>0.$$
Denote other two roots of \eqref{ef} by $\sigma(r)-i\xi(r)$ and $\eta(r)$. Then the Routh-Hurwitz criterion yields
\begin{equation}\label{B123-RH}
-A_1=2\sigma(r)+\eta(r),\ A_2=2\sigma(r)\eta(r)+\sigma^2(r)+\xi^2(r),-A_3=\eta(r)[\sigma^2(r)+\xi^2(r)].
\end{equation}
Differentiating \eqref{B123-RH} with respect to $r$ and applying \eqref{A123} give
\begin{equation}\label{B123-RH1}
\begin{split}
&2\sigma'(r)+\eta'(r)=0,\\
&2\sigma(r)\eta'(r)+2\sigma'(r)\eta(r)+2\sigma(r)\sigma'(r)+2\xi(r)\xi'(r)=\ell_0>0,\\
&2\eta(r)[\sigma(r)\sigma'(r)+\xi(r)\xi'(r)]+\eta'(r)[\sigma^2(r)+\xi^2(r)]=\ell_1<0,
\end{split}
\end{equation}
where
\begin{equation}\label{ell12}
\ell_0:=\frac{\sigma_m\lambda\chi_1\bar{u}_0}{\lambda(\bar{u}_0+\bar{v}_0)+\mu},\ \ell_1:=-\frac{\sigma_m^2\lambda\chi_1\bar{u}_0(1+\chi_2\bar{v}_0)}{\lambda(\bar{u}_0+\bar{v}_0)+\mu}.
\end{equation}
Given $\sigma(r_{m_0}^H)=0$, we deduce from \eqref{B123-RH} that $\eta(r_{m_0}^H)=-A_1<0$ and $\xi^2(r_{m_0}^H)=A_2>0.$ At $r=r_{m_0}^H$, we follow from \eqref{B123-RH1} that
\begin{equation*}
2\sigma'(r)\eta(r)+2\xi(r)\xi'(r)=\ell_0\ \ \mathrm{and}\ \ 2\eta(r)\xi(r)\xi'(r)-2\sigma'(r)\xi^2(r)=\ell_1<0,
\end{equation*}
which implies 
$$ 2\sigma'(r)[\eta^2(r)+\xi^2(r)]=\ell_0\eta(r)-\ell_1\ \text{at}~r=r_{m_0}^H.$$
This along with the fact $\eta(r_{m_0}^H)=-A_1<0$ and $\xi^2(r_{m_0}^H)=A_2>0$ yields 
$$\sigma'(r)\big|_{r=r_{m_0}^H}=-\frac{A_1\ell_0+\ell_1}{2(A_1^2+A_2)}.$$
By the definitions of $\ell_0$, $\ell_1$ in \eqref{ell12}, and the conditions \eqref{C-HOPF} and $\sigma_{m_0}>q_0$,    we obtain
$$-(A_1\ell_0+\ell_1)=\frac{\sigma_{m_0}\lambda\chi_1\bar{u}_0(\chi_2\bar{v}_0-d-1)}{\lambda(\bar{u}_0+\bar{v}_0)+\mu} (\sigma_{m_0}-q_0)>0.$$
Therefore, $\sigma'(r_{m_0}^H)>0.$ We complete the proof of Lemma \ref{TC}.
\end{proof}
\begin{proof}[{\bf Proof of Theorem \ref{LHP}}] The stability and instability results directly follow from Lemma \ref{LSUS1*}. In the sequel, we show the existence of Hopf bifurcations. When $r=r_j^H$, it follows from \eqref{basic} that
 $$ r>\frac{4(d+1)[\lambda(\bar{u}_0+\bar{v}_0)+\mu]^2}{(\chi_2\bar{v}_0-d-1)\lambda\chi_1\bar{u}_0}.$$
By Lemma \ref{imaginary}, the matrix $(-\sigma_j\mathcal{A}+\mathcal{B})$ has a pair of purely imaginary eigenvalues $\pm i\xi_0$ at $r=r_j^H$, and $\xi_0=\sqrt{A_2}$ (see \eqref{IM}). Let $E_j^\pm\ne (0,0,0)$ be the associated eigenvectors: 
$$(-\sigma_j\mathcal{A}+\mathcal{B})E_j^\pm=\pm i\xi_0E_j^\pm,$$
Then $\pm i\xi_0$ are also purely imaginary eigenvalues  of \eqref{le} at  $r=r_j^H$, with corresponding eigenfunction $E_j^\pm z_j(x)$. Here $z_j(x)$ is the eigenfunction corresponding to $\sigma_j$.

By the assumption $\sigma_j>q_*$, we deduce from Lemma \ref{ddx} that $r_m^H\not=r_j^H$ if $m\not= j$. And note the conditions that $\sigma_j$ is simple eigenvalue of $-\Delta$ under Neumann boundary conditions, then $\pm i\xi_0$ are simple eigenvalues for the linearized eigenvalue problem \eqref{le}.

Moreover, the proof of Lemma \ref{imaginary} shows that $(-\sigma_j\mathcal{A}+\mathcal{B})$ has no  eigenvalues with the form $ki\sqrt{A_2}$ for $k\in \mathbb{Z}\backslash\{\pm1\}$, hence neither does \eqref{le}. 

Finally, Lemma \ref{TC} implies that for $r$ near $r_{j}^H$, \eqref{le} has a unique eigenvalue $\operatorname{Re}[\rho(r)]+i\operatorname{Im}[\rho(r)]$ satisfying 
$$\operatorname{Re}[\rho(r_j^H)]=0,\ i\operatorname{Im}[\rho(r_j^H)]=i\xi_0,~\text{and}~ \frac{d\operatorname{Re}[\rho(r)]}{dr}\bigg|_{r=r_j^H}\ne0.$$
Therefore, applying \cite[Theorem 1]{Amann-Hopf-1990}(or see \cite[Theorem 6.1]{Liu-Shi-Wang-DCDSB}) directly yields the desired results. 
\end{proof}

\section{Numerical simulations}\label{NS}
In Section \ref{timeperiodic} and Section \ref{HP}, we study the non-constant positive time-periodic solution in homogeneous environments and heterogeneous settings, respectively. Specifically, in Section \ref{timeperiodic}, we first use the topological degree theory to establish the existence of non-constant positive time-periodic solutions when the resource renewal rate $r(x,t)$ is non-constant time-periodic, then  prove their uniqueness and  global asymptotic stability via coupled energy estimates when $r_*$ is small. In Section \ref{HP}, we establish the existence of non-constant time-periodic solution in homogeneous environments by Hopf bifurcation theory. However, there still remain some interesting problems worthy of further study:
\begin{itemize}
\item[(Q1)] For  non-constant time-periodic function $r(x,t)$, do the periodic solutions remain globally stable if $r_*$ is not small?
\item[(Q2)] For a constant $r(x,t)$, how about the global/local dynamics of the periodic solution induced by Hopf bifurcation?
\item[(Q3)] For a spatially heterogeneous $r(x,t)\equiv r(x)$, do periodic solutions exist besides non-constant steady states? If so, what mechanism triggers them?

\end{itemize}
In this section, we shall numerically explore these problems. For this purpose, in all simulations,  we choose $\Omega=(0,3\pi)$ and set
\begin{equation}\label{pv}
\chi_1=d=\lambda=\mu=1,\ \chi_2=18,\ \bar{u}_0=\bar{v}_0=1.
\end{equation}

\subsection{Time-periodic environments} 
\begin{figure}[t]
\centering
\begin{minipage}{0.8\textwidth}
\includegraphics[width=\linewidth]{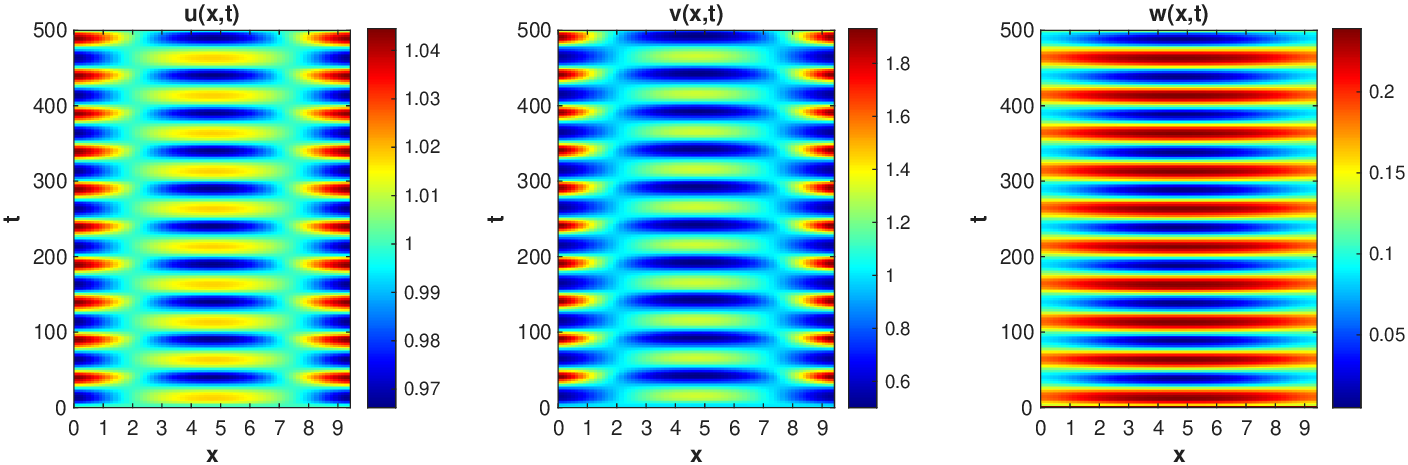}  
\end{minipage}  
\begin{minipage}{0.8\textwidth}
\includegraphics[width=\linewidth]{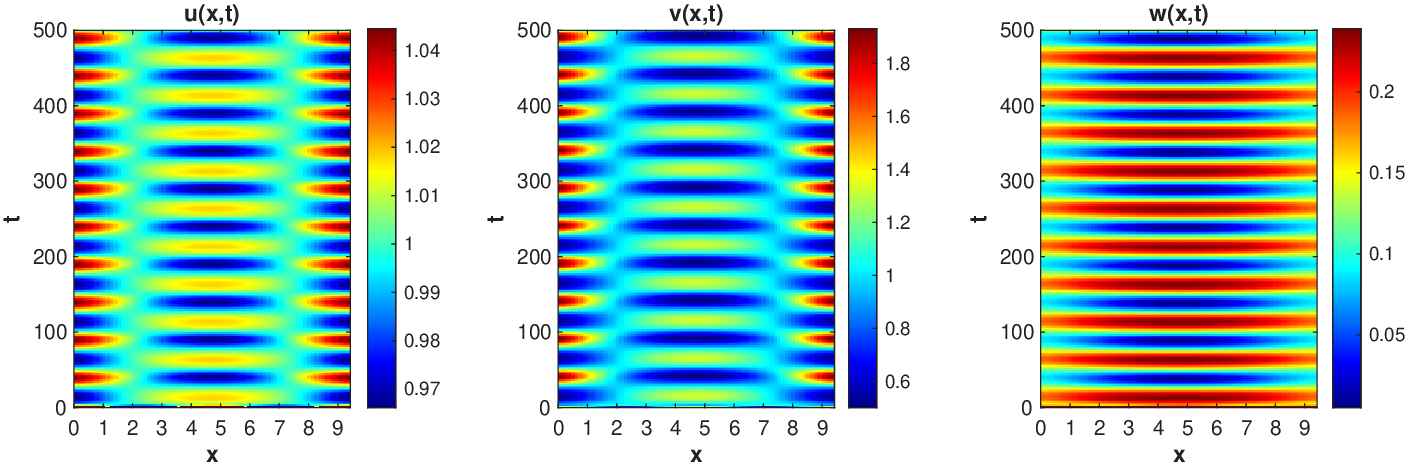}  
\end{minipage}
\captionsetup{width=\textwidth}
\caption{Numerical simulations of  solution profiles of \eqref{system} with $ r_*=0.8$ and other parameters set as in \eqref{pv}. The initial data $(u_0,v_0,w_0)$ is taken as: $(1, 1, 0.8/3)$ for the first row; $(1+0.3\cos(\frac{4 x}{3}), 1+0.3\cos(\frac{4 x}{3}), 0.8/3+5)$ for the second row.} \label{fig-rxt-rs}
\end{figure}
\begin{figure}[htbp]
\centering
\begin{minipage}{0.8\textwidth}
\includegraphics[width=\linewidth]{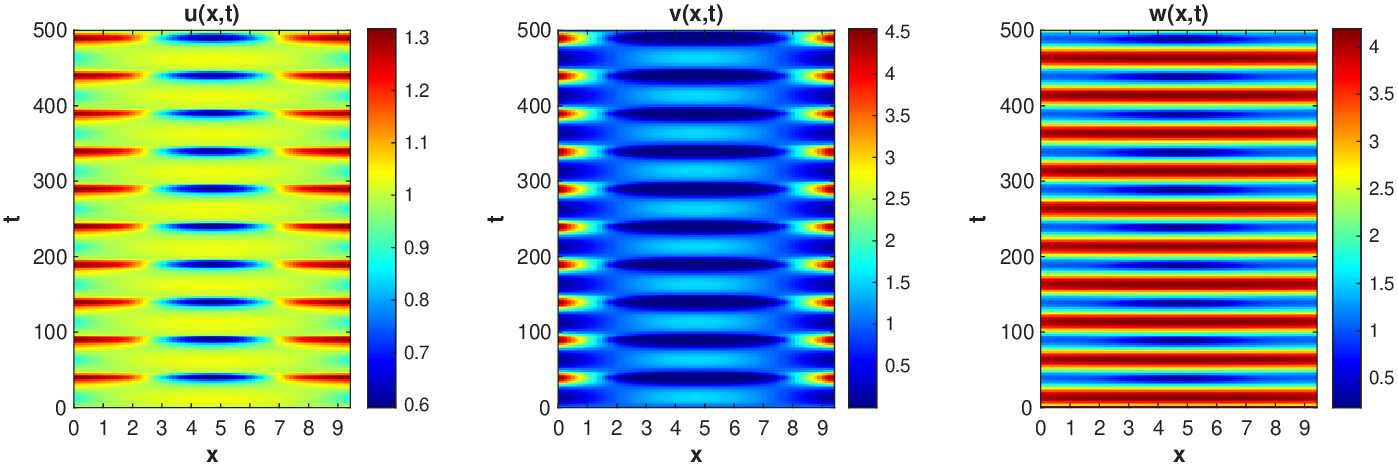}  
\end{minipage}  
\begin{minipage}{0.8\textwidth}
\includegraphics[width=\linewidth]{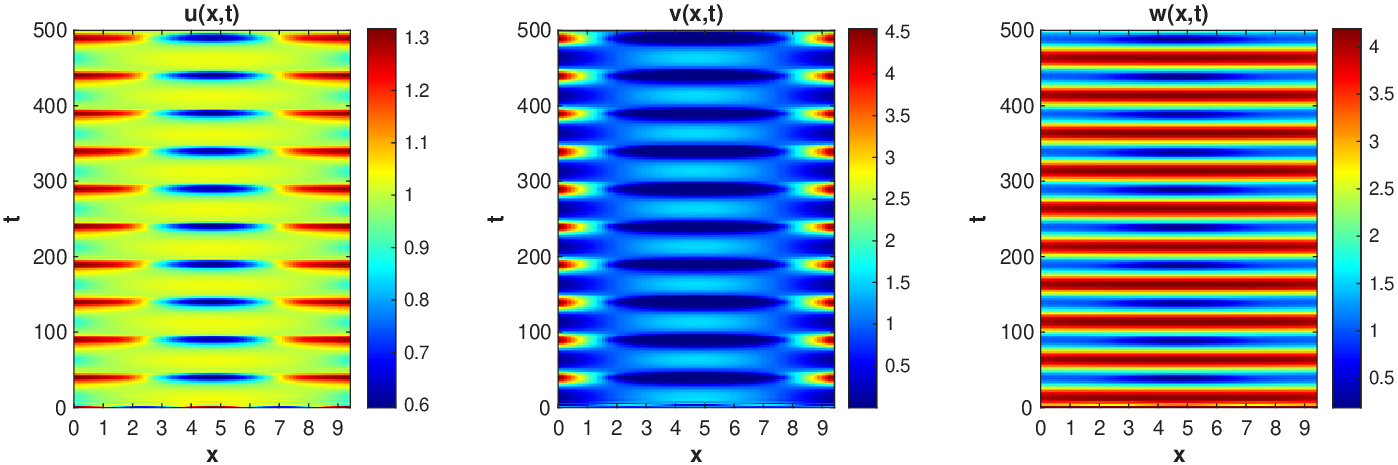}  
\end{minipage} 
\captionsetup{width=\textwidth}
\caption{Numerical simulations of  solution profiles of \eqref{system} with $ r_*=15$ and other parameters set as in \eqref{pv}. The initial data $(u_0,v_0,w_0)$ is taken as: $(1, 1, 5)$ for the first row; $(1+0.3\cos(\frac{4 x}{3}), 1+0.3\cos(\frac{4 x}{3}), 10)$ for the second row.} \label{fig-rxt-rl}
\end{figure}

We first verify the uniqueness and global stability of the non-constant periodic solution in Theorem \ref{GBSps} when $r_*$ is small, then we address  the problem (Q1). Our findings numerically give a positive answer to (Q1). To this end, we consider the following function:
$$r(x,t)=\frac{r_*}{2}\bigg(1+\frac{4\sin (\pi t/25)(3\pi x-x^2)}{9\pi^2}\bigg).$$
As shown in the first row of Figure \ref{fig-rxt-rs}, for small $r_*=0.8$,  the solution initiated from $(1,1,0.8/3)$ develops into a periodic solution. To further examine its attractivity, we also perform simulations with several other initial data satisfying \eqref{pv}, all of which eventually approach the periodic solution displayed in the first row. For brevity, we present only the simulation initiated from  $\bigl(1+0.3\cos\tfrac{4x}{3},\; 1+0.3\cos\tfrac{4x}{3},\; \tfrac{0.8}{3}+5\bigr)$, shown in the second row of Figure \ref{fig-rxt-rs}. These observations are consistent with Theorem~\ref{GBSps}. For large $r_*=15$, we observe a similar phenomenon. Solutions originating from a variety of  initial data satisfying \eqref{pv} all converge to a single  periodic solution; two examples are displayed in Figure \ref{fig-rxt-rl}. Motivated by these observations, we may conjecture that for any fixed $r(x,t)$ satisfying (H0), the non-constant positive periodic solution is globally asymptotically stable.

\subsection{Homogeneous environments (i.e., constant $r$)}\label{NRH}
\begin{figure}[t!]
\centering
\begin{minipage}[b]{0.96\textwidth}
    \includegraphics[width=\linewidth]{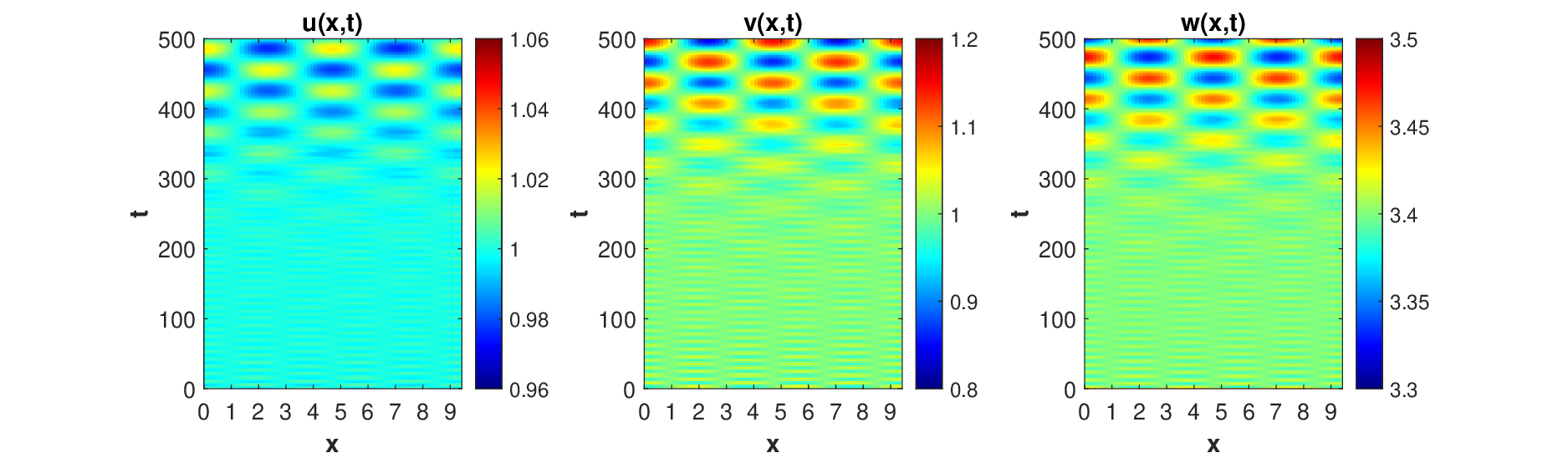}
\end{minipage}
\begin{minipage}[b]{0.96\textwidth}
    \includegraphics[width=\linewidth]{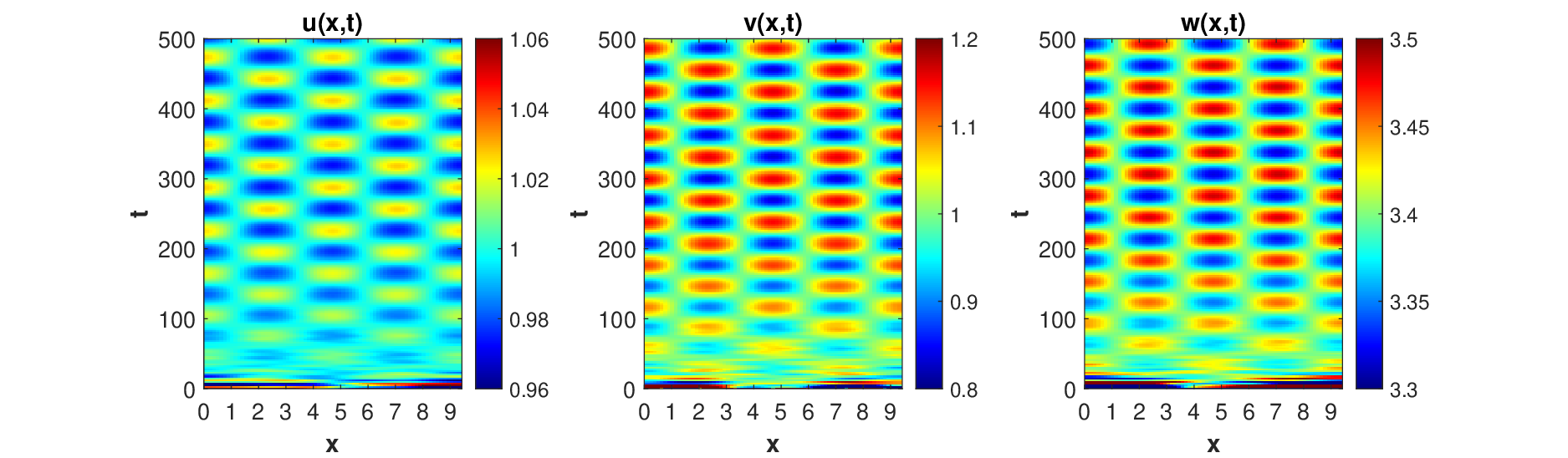}
\end{minipage}
\captionsetup{width=\textwidth}
\caption{Numerical simulations of  solution profiles of \eqref{system} with $ r=10.2$ and other parameters set as in \eqref{pv}. The initial data $(u_0,v_0,w_0)$ is taken as:  $(1+0.01\cos(\frac{5 x}{3}), 1+0.01\cos(\frac{5 x}{3}), \frac{10.3}{2}+ 0.01\cdot\text{rand})$ for the first row; $(1.2-\frac{0.4x}{3\pi}, 1+0.9\cos(\frac{5 x}{3}),xe^{-(x-1)^2})$ for the second row.} 
\label{fig-rc=10.2}
\end{figure}
\begin{figure}[t!]
\centering
\begin{minipage}[b]{0.96\textwidth}
    \includegraphics[width=\linewidth]{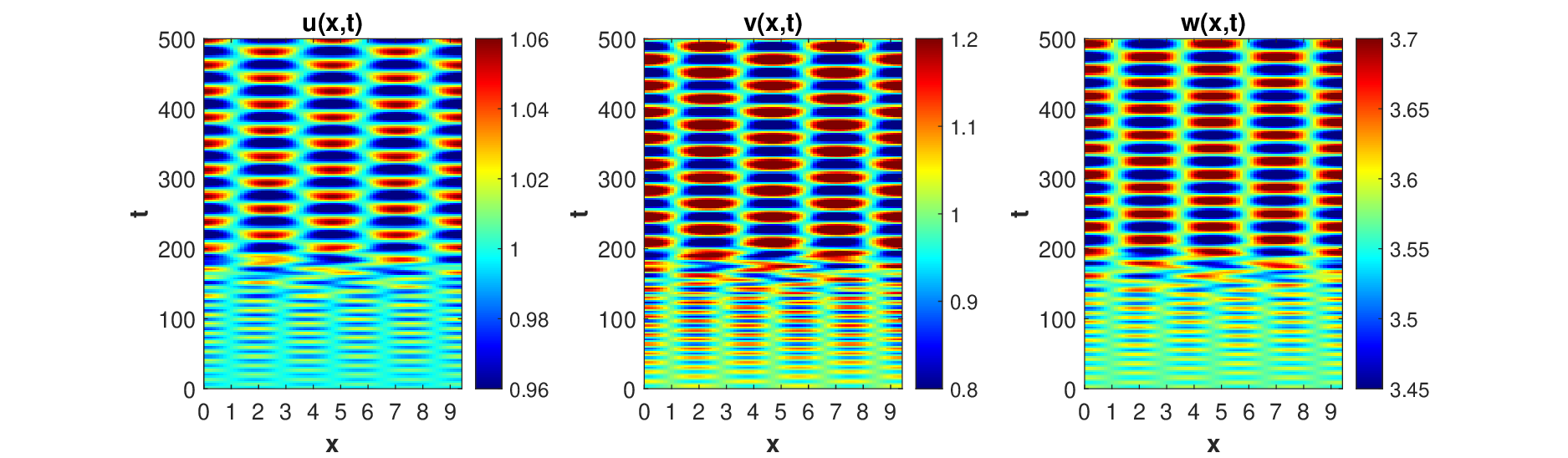}
\end{minipage}
\begin{minipage}[b]{0.96\textwidth}
    \includegraphics[width=\linewidth]{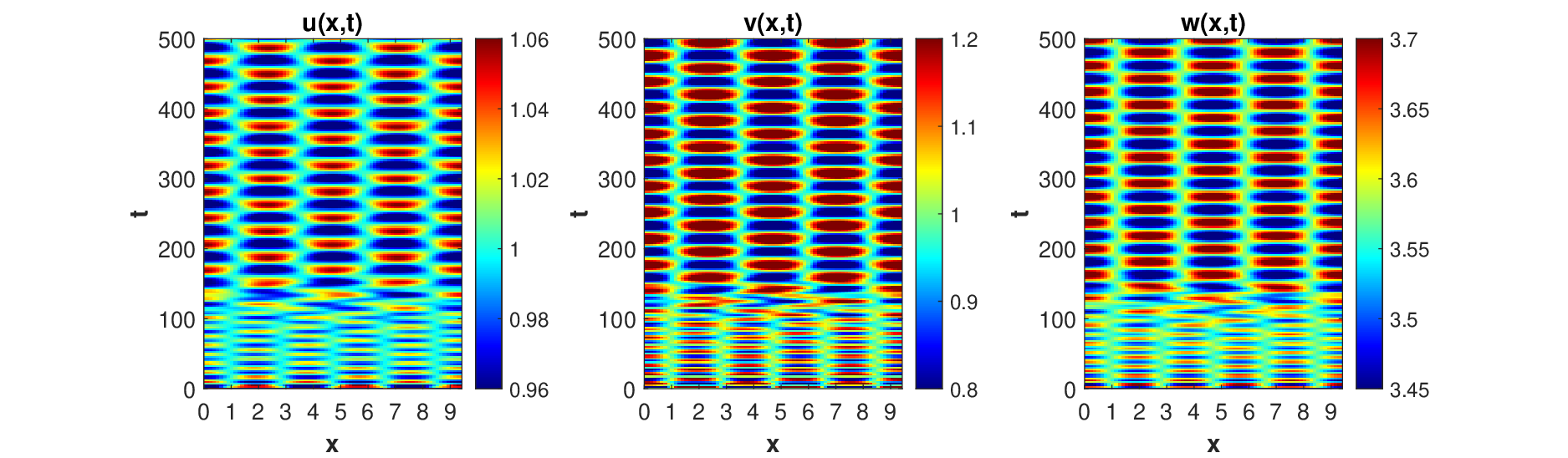}
\end{minipage}
\captionsetup{width=\textwidth}
\caption{Numerical simulations of  solution profiles of \eqref{system} with $ r=10.7$ and other parameters set as in \eqref{pv}. The initial data $(u_0,v_0,w_0)$ is taken as:  $(1+0.01\cos(\frac{5 x}{3}), 1+0.01\cos(\frac{7x}{3}), \frac{10.7}{2}+ 0.01\cdot\text{rand})$ for the first row; $(1+0.9\cos(\frac{5 x}{3}), 1,\text{rand})$ for the second row.} 
\label{fig-rc=10.7}
\end{figure}

\begin{figure}[h!]
\centering
\begin{minipage}[b]{0.96\textwidth}
    \includegraphics[width=\linewidth]{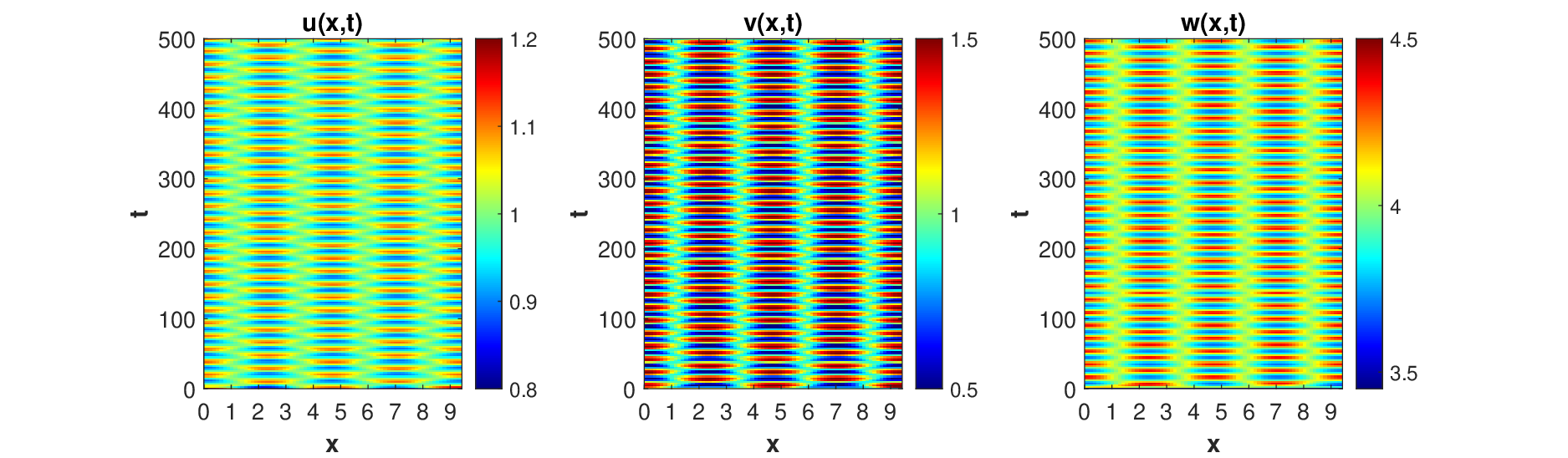}  
\end{minipage} 
\begin{minipage}[b]{0.96\textwidth}
    \includegraphics[width=\linewidth]{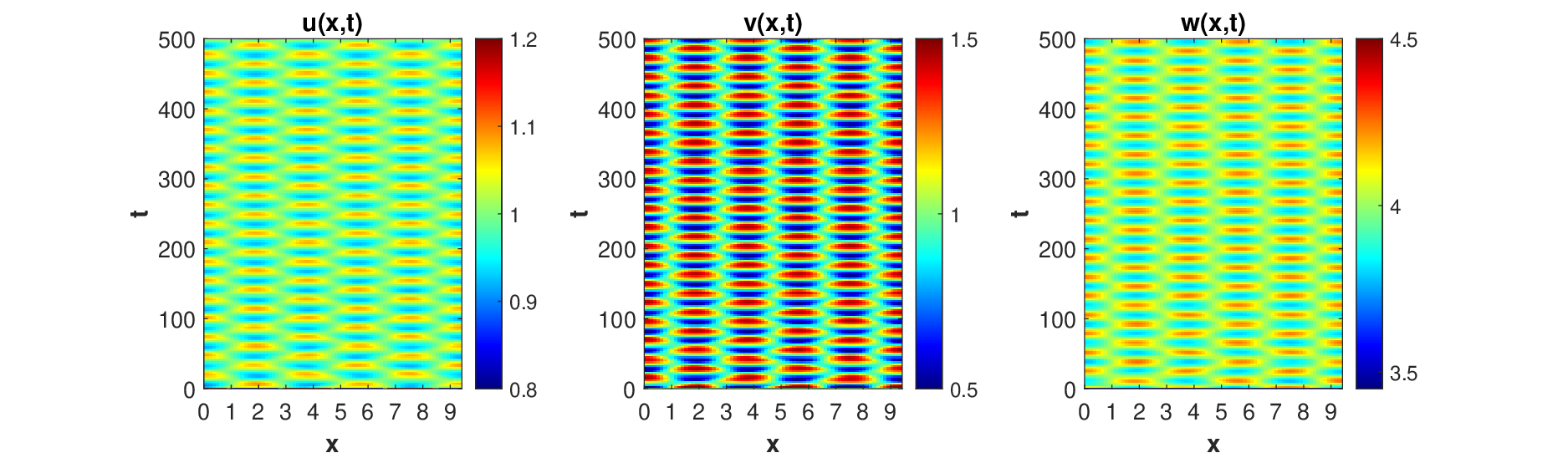}  
\end{minipage}
\captionsetup{width=\textwidth}
\caption{Numerical simulations of  solution profiles of \eqref{system} with $ r=12$ and other parameters set as in \eqref{pv}. The initial data $(u_0,v_0,w_0)$ is taken as $(1+0.2\cos(\frac{4 x}{3}), 1.2-\frac{0.4x}{3\pi}, \frac{x}{3})$ for the first row; $(1+0.9\cos(\frac{5 x}{3}), 1,e^{-(x-1)^2}(x+0.1))$ for the second row.} \label{bi}
\end{figure}

This subsection first examines spatiotemporal patterns in the parameter regimes identified in Theorem \ref{LHP}, then investigates global/local of periodic patterns as stated in problem (Q2). Interestingly, we find that the global attractivity and bistability of periodic solutions exhibit sensitive dependence on the parameter $r$, which sharply differs from the scenario in time-periodic environments.

It follows from \eqref{pv} that  $\sigma_m=\frac{m^2}{9}$ and $\bar{u}_0=\bar{v}_0=1$. Thus, by \eqref{rhm}, we obtain
\begin{equation}\label{cv}
r_m^H=\frac{3(\frac{8m^4}{9}+24m^2+162)}{16m^2-27},
\end{equation}
and
$$q_0=\frac{\lambda(\bar{u}_0+\bar{v}_0)+\mu}{\chi_2\bar{v}_0-1-d}=\frac{3}{16}, \ q_*=\frac{[2(d+1)+\chi_2\bar{v}_0-d-1][\lambda(\bar{u}_0+\bar{v}_0)+\mu]}{(d+1)(\chi_2\bar{v}_0-d-1)}=\frac{15}{8},$$
which implies the sets
$$Q_0:=\{m_0\in\mathbb{Z}^+|\sigma_{m_0}>q_0\}=\{2,3,4,\cdots\},\ Q_*:=\{j\in\mathbb{Z}^+|\sigma_{j}>q_*\}=\{5,6,7,\cdots\}.$$
By Lemma \ref{LSUS1*} and \eqref{cv}, we get
$$\min\limits_{m_0\in Q_0}\{r_{m_0}^H\}=\min\{r^H(\sigma_{4}),r^H(\sigma_{5})\}=r^H(\sigma_{4}), \ \min\limits_{j\in Q_*}\{r_{j}^H\}=r^H(\sigma_{5}),$$
hence the critical value classifying for stability and instability is $r^H(\sigma_{4})\approx 10.13$. And it follows from Theorem \ref{LHP} that $r_5^H\approx 10.60$, $r_6^H\approx 11.90$, $r_7^H\approx 13.76$, $r_8^H\approx 16.06$ $\cdots$ are bifurcation points.

As shown in the first row of Figure \ref{fig-rc=10.2}, a small perturbation of $(1,1,10.2/3)$ with $r=10.2>r_4^H\approx 10.13$ evolves into a stable time‑periodic solution  after $t=300$. This suggests that the condition $\sigma_j\geq q_*$ in Theorem \ref{LHP} may be a sufficient condition, and  $r_4^H$ is the actual minimal Hopf bifurcation point that sharply separates stability and instability. Moreover, simulations initialized with far-from-equilibrium $(1.2-\frac{0.4x}{3\pi}, 1+0.9\cos(\frac{5x}{3}),xe^{-(x-1)^2})$ also converge to the same periodic profile as shown in the second row in Figure \ref{fig-rc=10.2}. Additional simulations with different initial data exhibit the same asymptotic periodic behavior, although the time required to reach the periodic state may vary. For example, starting from the initial state \((1+0.9\cos\frac{5x}{3},\,1+0.6\cos\frac{5x}{3},\,2.6)\), the solution approaches the same periodic profile after approximately after \(t=3000\) (figures are omitted for brevity). These numerical findings provide further evidence for the global attractivity of the underlying periodic solution.

For $r=10.7$, slightly above the theoretical Hopf point $r_5^H\approx 10.60$ predicted by Theorem \ref{LHP}, a small perturbation of the steady state  $(1,1,10.7/2)$ evolves into a stable periodic pattern (first row, Figure \ref{fig-rc=10.7}), which is consistent with theoretical bifurcation analysis. Moreover, simulation initiated from several distinct initial data all asymptotically approach the same periodic solution. One representative example is shown in the second row, where $(u_0,v_0,w_0)=\bigl(1+0.9\cos\frac{5x}{3},1, \text{rand}\bigr)$. This observation provides further numerical evidence for the global attractivity of the periodic solution.

By contrast, increasing the parameter to $r=12>r_6^H\approx 11.9$ leads to bistability, indicating  the loss of global attractivity (Figure \ref{bi}). Specifically, as illustrated in the first row of Figure \ref{bi}, the solution initiated from $(1+0.2\cos(\frac{4 x}{3}), 1.2-\frac{0.4x}{3\pi}, \frac{x}{3})$ converges to a five-peaked periodic solution, whereas the solution initiated from $(1+0.9\cos(\frac{5 x}{3}), 1,e^{-(x-1)^2}(x+0.1))$ evolves toward a distinct six-peaked periodic solution.

\subsection{Spatially heterogeneous environments (i.e., $r(x,t)=r(x)$)} 
\begin{figure}[t!]
\centering
\begin{minipage}[b]{0.96\textwidth}
    \includegraphics[width=\linewidth]{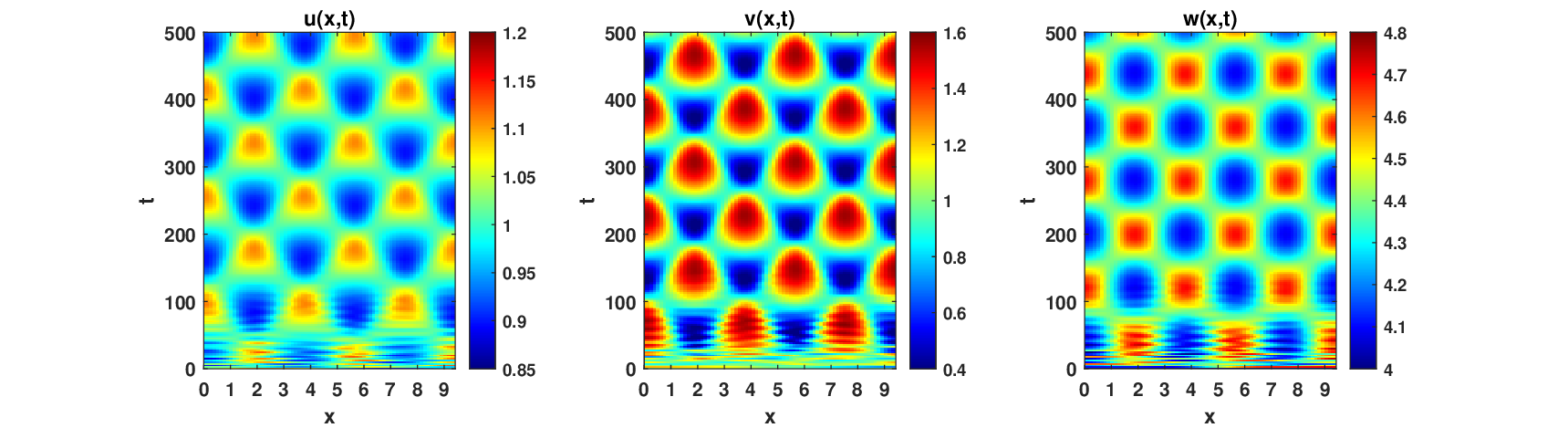}  
    \subcaption{$r(x)=13.08$}
\end{minipage} 
\begin{minipage}[b]{0.96\textwidth}
    \includegraphics[width=\linewidth]{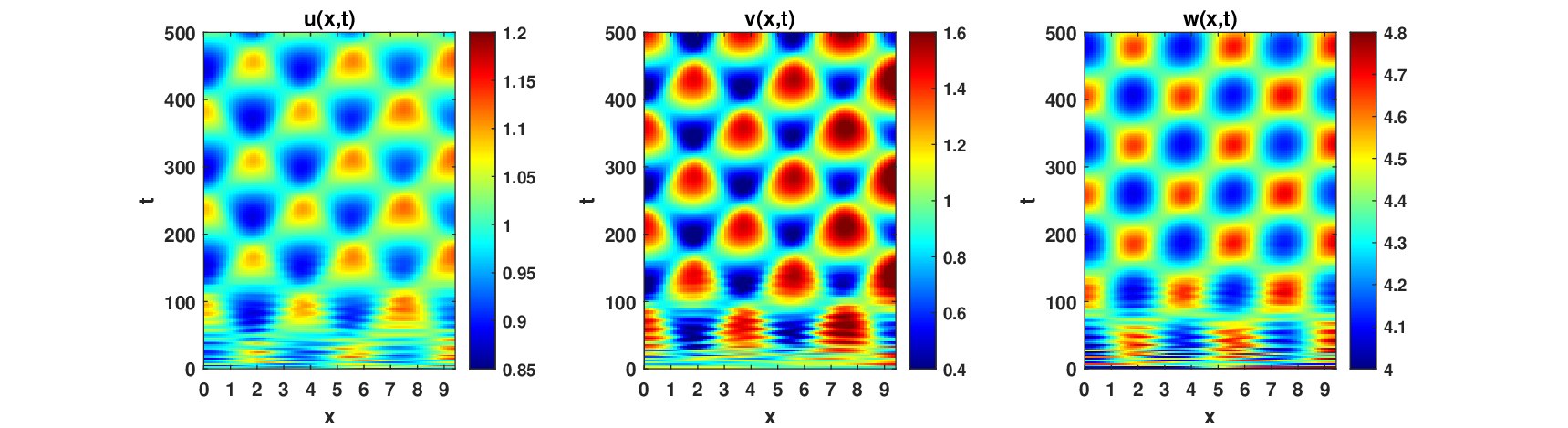}  
    \subcaption{$r(x)=r_1(x)$}
\end{minipage}  
\begin{minipage}[b]{0.96\textwidth}
    \includegraphics[width=\linewidth]{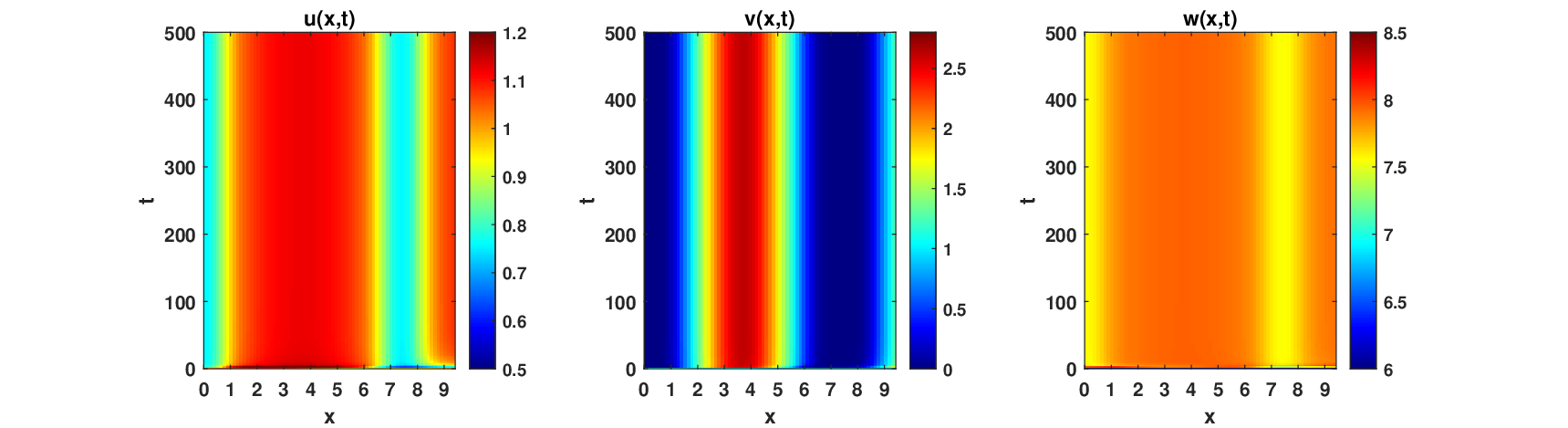}  
    \subcaption{$r(x)=r_2(x)$}
\end{minipage}  
\captionsetup{width=\textwidth}
\caption{Numerical simulations of  solution profiles of \eqref{system} and $r_1(x), r_2(x)$ are taken in \eqref{r12}. The other parameters set as in \eqref{pv}, and initial data $(u_0,v_0,w_0)$ is taken as $(1+0.1\sin(\frac{2 x}{3}), 1+0.1\sin(\frac{2 x}{3}), \frac{8}{3})$.} \label{fig-sp}
\end{figure}
\begin{figure}[t!]
\centering
\begin{subfigure}[b]{0.32\textwidth}
    \includegraphics[width=\linewidth]{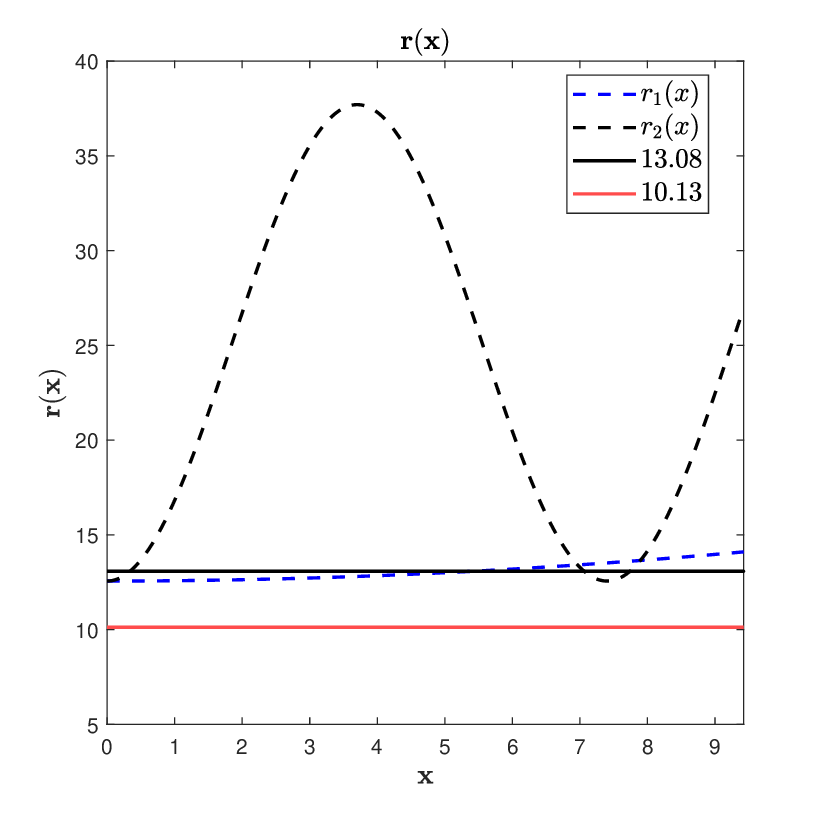}
    \caption{}
\end{subfigure}
\hfill
\begin{subfigure}[b]{0.32\textwidth}
    \includegraphics[width=\linewidth]{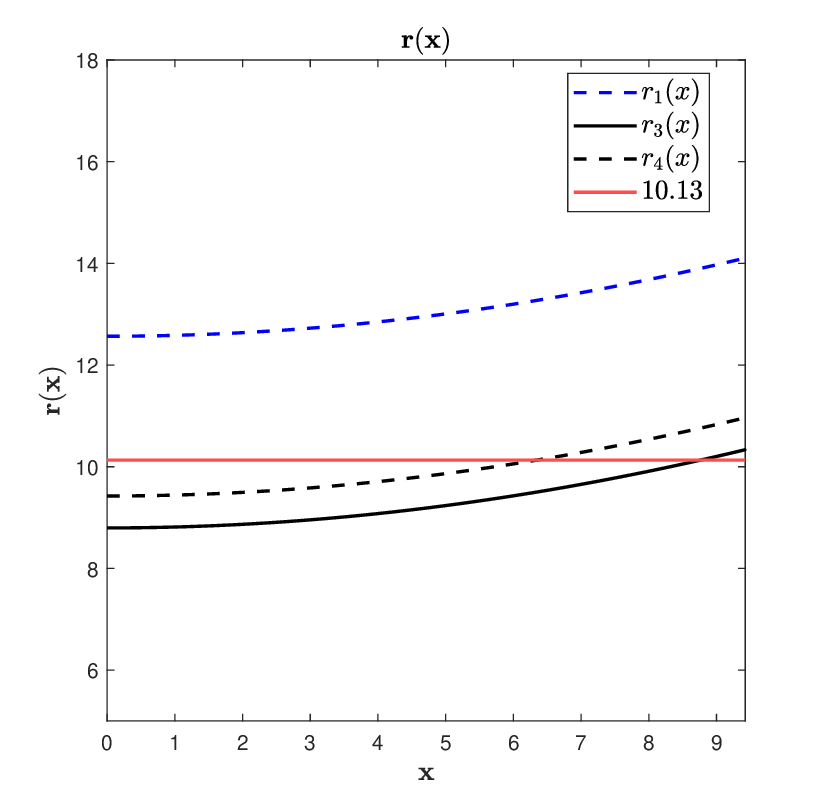}
    \caption{}
\end{subfigure}
\hfill
\begin{subfigure}[b]{0.32\textwidth}
    \includegraphics[width=\linewidth]{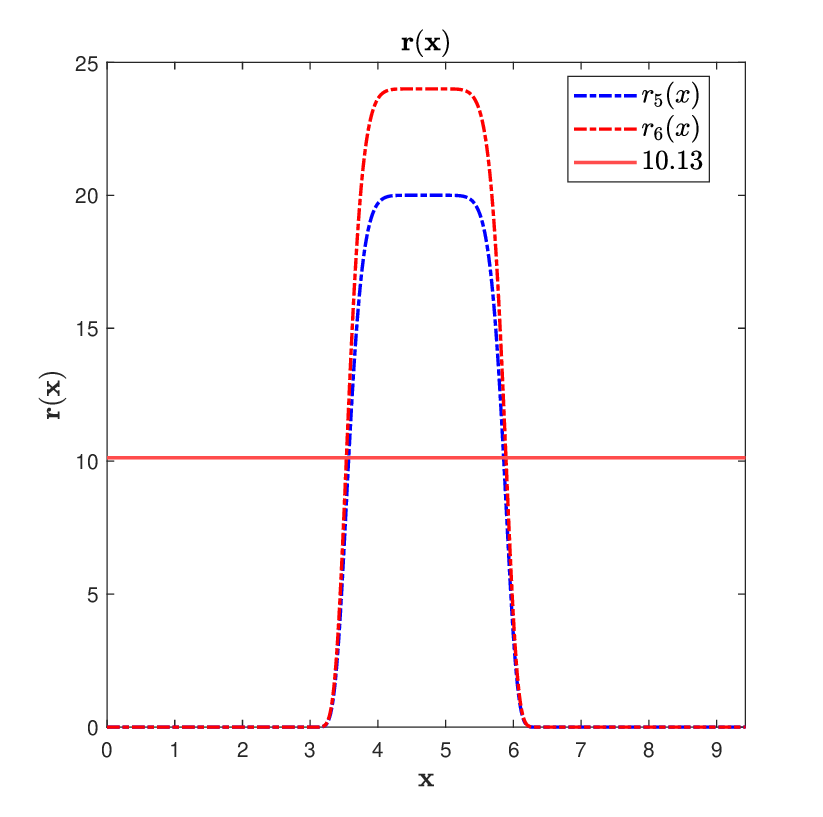}
    \caption{}
\end{subfigure}
\caption{Spatial profiles of the resource input functions $r_1(x)-r_6(x)$.}\label{5678}
\end{figure}

Unlike the case of time-periodic resource renewal $r(x,t)$, the system \eqref{system} with $r(x)$ lacks external time-periodic forcing. Compared with the homogeneous constant-resource case, the system \eqref{system} with $r(x)$ admits non-constant positive steady states, and this non-constant positive steady state is unique and globally stable when $r^*$ is suitably small by applying the same proof procedure as in Theorems \ref{GBSp}-\ref{GBSps}. Consequently, the classical Hopf bifurcation approach cannot be applied directly, which substantially complicates the analysis of positive periodic solutions. We thus turn to numerical simulations to explore possible periodic patterns and their underlying mechanisms  when $r^*$ is not small, thereby addressing question (Q3).

Numerical simulations reveal two distinct dynamical regimes for stationary spatially heterogeneous resources. When the domain-averaged resource exceeds \(10.13\), strong spatial heterogeneity suppresses temporal oscillations, whereas moderately uniform resource distributions facilitate temporal oscillations. Conversely, when the global average below \(10.13\), oscillation generation is no longer governed by the total resource amount. Instead, time-periodic dynamics are triggered by locally concentrated resource regions. Moreover, a lower overall resource level requires stronger local resource accumulation to sustain oscillations. Detailed numerical findings are presented below.

\vspace{2mm}
{\bf  Case 1: Sufficient resource input (i.e., $\bar{r}(x)>10.13$).} 
Our numerical results indicate that spatial heterogeneity can suppress temporal oscillations. Moreover, once the domain-averaged resource is sufficiently large, the emergence of time-periodic solutions is determined primarily by the spatial distribution of resources rather than their total abundance. To illustrate this phenomenon, we consider two heterogeneous resource distributions:

\begin{align}
\label{r12}
&r_1(x):=8\pi-4\pi\cos\Big(\frac{x}{6\pi}\Big),\ r_2(x):=8\pi-4\pi\cos\Big(\frac{8x}{3\pi}\Big),
\end{align}
whose spatial profiles are displayed in Figure \ref{5678}(a). The corresponding resource statistics are computed as 
\begin{align*}
&\max\limits_{x\in [0,3\pi]} r_1(x)\approx14.11,\ \bar{r}_1(x)\approx 13.08>10.13,\ \min\limits_{x\in [0,3\pi]}r_1(x)=4\pi\approx 12.57,\\
&\max\limits_{x\in [0,3\pi]} r_2(x)\approx37.70,\ \bar{r}_2(x)\approx 23.58>10.13,\ \min\limits_{x\in [0,3\pi]}r_2(x)=4\pi\approx 12.57.
\end{align*}
For comparison, we also examine the homogeneous resource case $r\equiv 13.08$, which coincides with the spatial average of $r_1(x)$. 

As shown in Figure \ref{fig-sp}, although the average resource level of $r_2(x)$ ($\bar{r}_2\approx 23.58$) is much higher than that of both the homogeneous resource distribution $r\equiv 13.08$ and the weakly heterogeneous case $r_1(x)$ ($\bar{r}_1\approx 13.08$), and the minima of all three cases remains above $10.13$, only the homogeneous distribution $r\equiv 13.08$ and the weakly heterogeneous $r_1(x)$ generate stable time-periodic patterns (Figure \ref{fig-sp}(a)--(b)). In contrast, the strongly heterogeneous distribution $r_2(x)$ leads to a non-constant steady state (Figure \ref{fig-sp}(c)). These results indicate that, once the average resource level is sufficiently high,  the stronger spatial heterogeneity suppresses temporal oscillations. However, when the resource input becomes insufficient, the situation is more complicated, as discussed below.

\vspace{2mm}
{\bf Case 2: Insufficient resource supply (i.e., $\bar{r}(x)\leq10.13$).}
To characterize the periodic dynamics under insufficient resource supply, we define the local average resource level
$$\bar{r}_i^l(x)=\frac{1}{|\Omega'|}\int_{\Omega'} r_i(x)dx \  \ \text{with} \ \  \Omega':=\{x\in [0,3\pi]\ |\ r_i(x)\geq 10.13\},$$
which represents the average resource level over the regions where the resource level exceeds $10.13$. Numerical results reveal that the local average resource abundance plays a crucial role in triggering temporal oscillations when the global resource supply is inadequate. In this scenario, spatially uniform resource distributions are less favorable for the emergence of periodic solutions, which is in contrast to the homogeneity effect observed under sufficient resource conditions.

To illustrate these observations, we consider two classes of resource functions defined on $[0,3\pi]$: the first class has the same derivatives (see Figure \ref{5678}(b)), while the second class has identical spatial profiles up to a multiplicative constant (see Figure \ref{5678}(c)). The specific resource distributions are given by:
\begin{figure}[t!]
\centering
\begin{minipage}[b]{0.96\textwidth}
    \includegraphics[width=\linewidth]{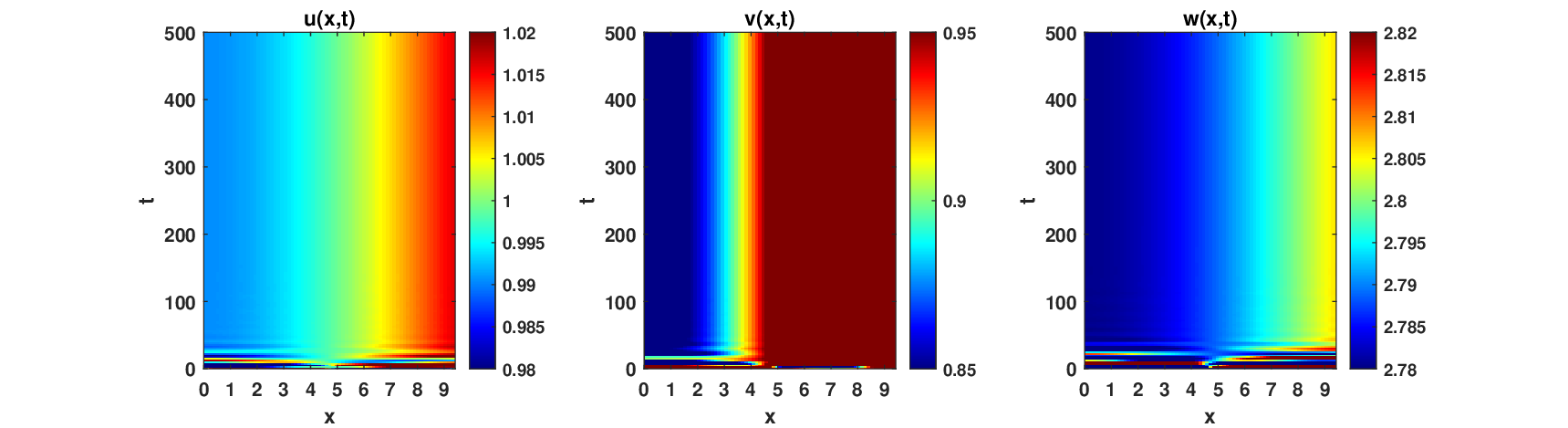}  
    \subcaption{$r(x)=r_3(x)$}
\end{minipage} 
\begin{minipage}[b]{0.96\textwidth}
    \includegraphics[width=\linewidth]{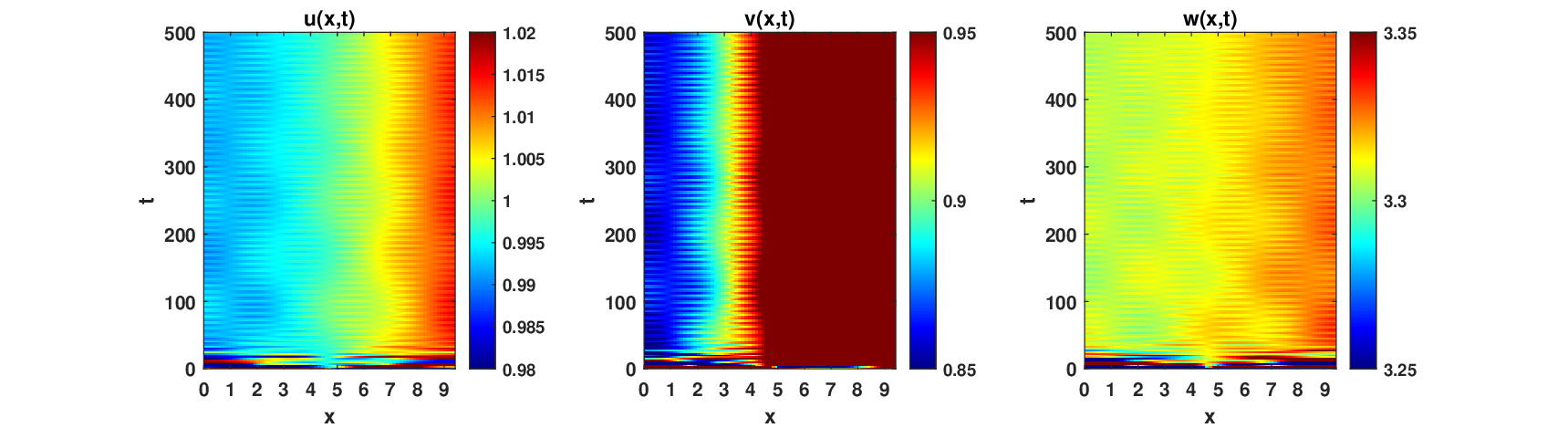}  
    \subcaption{$r(x)=r_4(x)$}
\end{minipage}  
\captionsetup{width=\textwidth}
\caption{Numerical simulations of  solution profiles of \eqref{system} and $r_3(x), r_4(x)$ are taken in \eqref{r56}. The other parameters set as in \eqref{pv}, and initial data $(u_0,v_0,w_0)$ is taken as $(1+0.1\sin(\frac{2 x}{3}), 1+0.1\sin(\frac{2 x}{3}), \frac{8}{3})$.} \label{fig-56}
\end{figure}
\begin{figure}[t!]
\centering
\begin{minipage}[b]{0.96\textwidth}
    \includegraphics[width=\linewidth]{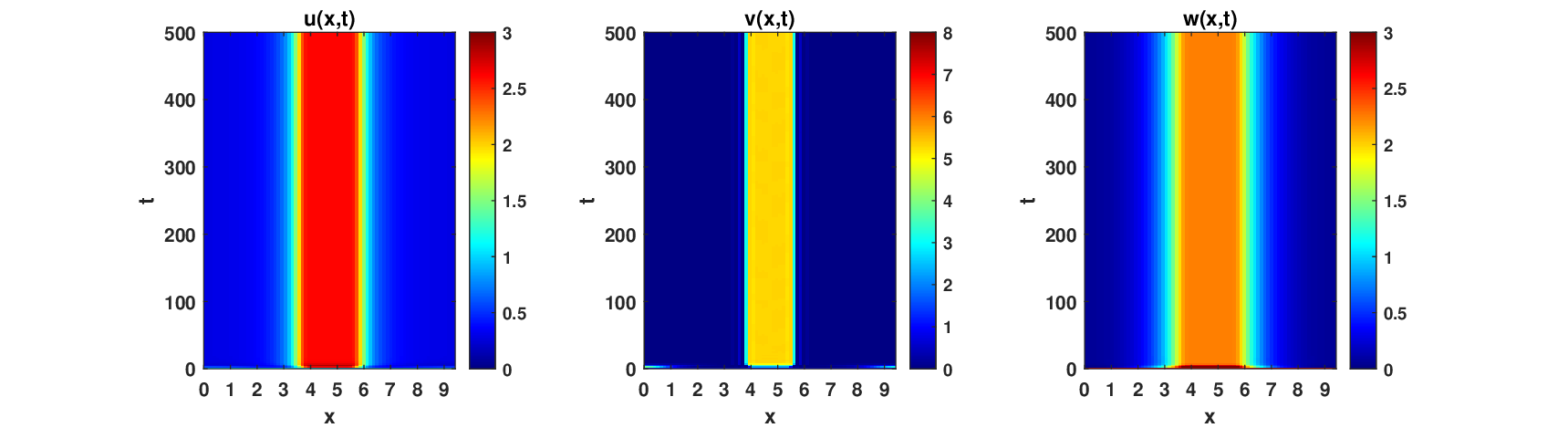}  
    \subcaption{$r(x)=r_5(x)$}
\end{minipage} 
\begin{minipage}[b]{0.96\textwidth}
    \includegraphics[width=\linewidth]{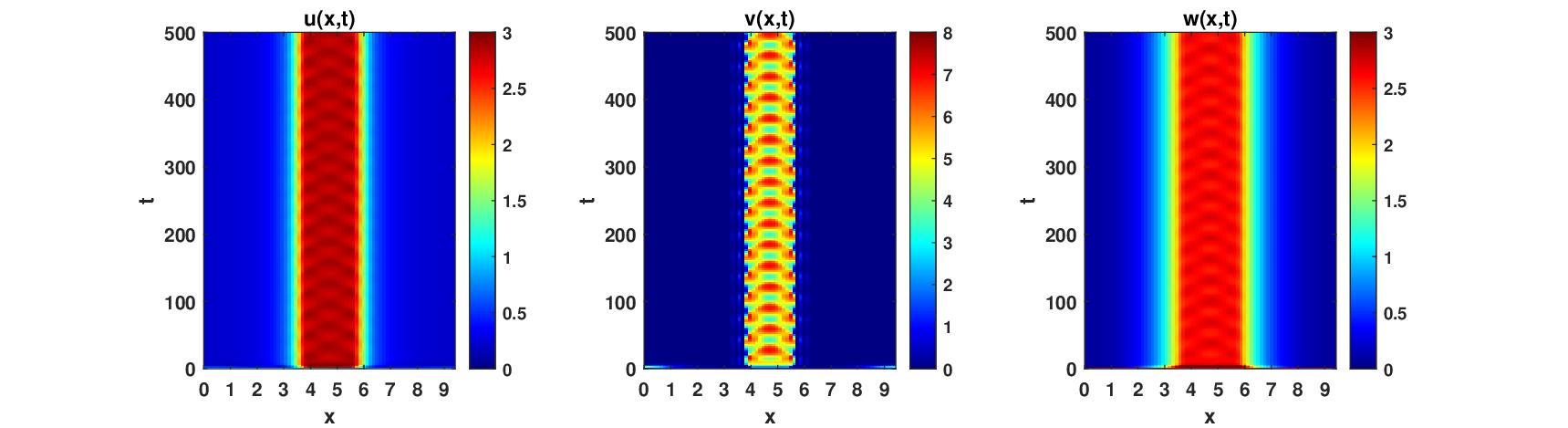}  
    \subcaption{$r(x)=r_6(x)$}
\end{minipage}  
\captionsetup{width=\textwidth}
\caption{Numerical simulations of  solution profiles of \eqref{system} and $r_5(x), r_6(x)$ are taken in \eqref{r78}. The other parameters set as in \eqref{pv}, and initial data $(u_0,v_0,w_0)$ is taken as $(1+0.1\sin(\frac{2 x}{3}), 1+0.1\sin(\frac{2 x}{3}), \frac{8}{3})$.} \label{fig-78}
\end{figure}
\begin{align}
\label{r56}
&r_3(x):=6.8\pi-4\pi\cos\Big(\frac{x}{6\pi}\Big),\ r_4(x):=7\pi-4\pi\cos\Big(\frac{x}{6\pi}\Big), \\
\label{r78}
&r_5(x) := 20 e^{-\left(\frac{x - 1.5\pi}{1.2}\right)^8} , \quad \quad \quad r_6(x) := 24 e^{-\left(\frac{x - 1.5\pi}{1.2}\right)^8} ,
\end{align}
with the corresponding quantitative characteristics calculated as
\begin{align*}
&\max\limits_{x\in [0,3\pi]} r_3(x)\approx10.33,\ \bar{r}_3(x)\approx 9.31<10.13,\ \min\limits_{x\in [0,3\pi]}r_3(x)\approx 8.80, \ \bar{r}_3^l(x) \approx 10.23 >10.13,\\
&\max\limits_{x\in [0,3\pi]} r_4(x)\approx10.96,\ \bar{r}_4(x)\approx 9.94<10.13,\ \min\limits_{x\in [0,3\pi]}r_4(x)\approx 9.42,\ \bar{r}_4^l(x) \approx 10.52 >10.13,\\
&\max\limits_{x\in [0,3\pi]} r_5(x)=20,\ \ \bar{r}_5(x)\approx 4.80<10.13,\ \ \quad\min\limits_{x\in [0,3\pi]}r_5(x)\approx 0.00, \ \bar{r}_5^l(x) \approx 42.82 >10.13,\\
&\max\limits_{x\in [0,3\pi]} r_6(x)=24,\ \ \bar{r}_6(x)\approx 5.76<10.13,\ \ \quad\min\limits_{x\in [0,3\pi]}r_6(x)\approx 0.00,\ \bar{r}_6^l(x) \approx 52.14>10.13.\\
\end{align*}

As illustrated in Figure \ref{fig-56}, for resource distributions with same derivative, comparison between Figure \ref{fig-56}(a) ($\bar{r}_3^l\approx 10.23$) and Figure \ref{fig-56}(b) ($\bar{r}_4^l\approx 10.52$) demonstrates that sufficiently high local average resource level can induce temporal oscillations. The numerical results in Figure \ref{fig-78} further confirm that temporal periodic patterns can still emerge under globally insufficient resource conditions, provided that the local average resource level is sufficiently high.  Furthermore, comparing $r_4(x)$ and $r_6(x)$ shows that a lower global average resource (e.g., $\bar{r}_6\approx 5.76$ for $r_6(x)$, which is smaller than $\bar{r}_4\approx 9.94$ for $r_4(x)$) requires a much higher local average resource level ($\bar{r}_6^l\approx 52.14 > \bar{r}_4^l\approx 10.52$) to trigger periodic dynamics.  These results suggest that increasing the global average resource level lowers the local resource level required to sustain periodic dynamics.

\subsection{Discussion}
This section summarizes the numerical results related to questions (Q1)–(Q3) by systematically comparing the dynamical behaviors of the system \eqref{system} under time-periodic, homogeneous, and stationary spatially heterogeneous resource environments. A rigorous mathematical analysis of these numerical observations remains a direction for future work.

In time-periodic settings with fixed initial biomass, the non-constant positive periodic solution remains unique and globally stable for both small and large values of $r_*$ under the assumption (H0). These numerical results demonstrate that the theoretically established small-$r_*$ condition is sufficient but not necessary for global stability. This can be attributed to the dominant influence of external periodic forcing, which synchronizes the system dynamics and maintains coherent temporal oscillations regardless of the resource level.  

In homogeneous environments (i.e., $r>0$ is a constant), the system exhibits mode-dependent Hopf bifurcation.  When $r$ slightly exceeds the critical threshold $r_4^H\approx10.13$, only low-order spatial modes trigger stable, globally attractive single periodic patterns. As $r$ increases beyond $r_6^H\approx11.9$,  multiple unstable modes emerge, and  nonlinear mode interaction gives rise to bistability, characterized by the coexistence of stable five-peaked and six-peaked periodic solutions. Notably, the first numerically observed Hopf bifurcation occurs at $r_4^H\approx10.13$, which is smaller than the theoretically predicted threshold $r_5^H\approx10.60$, indicating that the smallest Hopf bifurcation point may occur at the mode $m=4$.

For stationary spatially heterogeneous resource distributions, the emergence of periodic dynamics is jointly influenced by the global resource abundance, local resource level, and the degree of spatial uniformity.  When both the average and minimum resource levels exceed $10.13$, relatively uniform resource distributions promote temporal oscillations throughout the domain, whereas strong spatial heterogeneity suppresses oscillations and drives the system toward a non-constant steady state. Conversely, when the global average resource is insufficient, domain-wide oscillations disappear, but the sufficiently enriched local resource can act as oscillation sources and generate local even global periodic dynamics. Moreover, the numerical results reveal a clear compensatory effect: a lower global resource level requires a much higher local resource level to initiate and sustain temporal periodic dynamics.

\appendix
\renewcommand{\thesection}{ \Alph{section}}
\section{H\"{o}lder regularity}
Based on the boundedness of global classical solutions in \eqref{uvb*}, this appendix aims to improve the regularity, which is crucial in establishing the global stability of periodic solutions.
\begin{lemma}\label{lem:holder}
Let the hypothesis $(\operatorname{H})$ hold, and $(u, v, w)$ be the nonnegative global classical solution of \eqref{system} satisfying \eqref{uvb*}. Then there exists a constant $K_2>0$, independent of $t$ and $r_*$,
such that, for all $t\geq2$,
\begin{equation}\label{conclu_w}
\|w(\cdot, t)\|_{W^{1,\infty}} \le K_0 K_2 (r_*+1),
\end{equation}
and
\begin{equation}\label{conclu_uv}
\|u\|_{C^{\tilde{\alpha}, \frac{\tilde{\alpha}}{2}}(\bar{\Omega}\times[t, t+1])} + \|v\|_{C^{\tilde{\alpha}, \frac{\tilde{\alpha}}{2}}(\bar{\Omega}\times[t, t+1])} \le K_0^{\frac{n+11}{2}} K_2 (r_*+1)^{\frac{n+7}{2}},
\end{equation}
where $\tilde{\alpha}\in (0,1)$ and $K_0:=K_0(r_*)\geq 1$ is given in \eqref{uvb*}.
\end{lemma}

\begin{proof}
\textbf{Step 1: Uniform boundedness for $\|w(\cdot,t)\|_{W^{1,\infty}}$ and integral bounds for $\Delta w$.}
We rewrite the third equation of \eqref{system} as
$$ w_t - d\Delta w +\mu w = f(x,t),$$
where $f(x,t) := -\lambda(u+v)w + r(x,t)$. Applying the comparison principle to the $w$-equation yields 
$$ \|w(\cdot, t)\|_{L^\infty} \le \|w_0\|_{L^\infty} + \frac{r_*}{\mu}. $$
Then the boundedness of $u$ and $v$ in \eqref{uvb*} guarantees that 
$$ 
\begin{aligned}
\|f(\cdot, t)\|_{L^\infty} &\le \lambda(\|u(\cdot, t)\|_{L^\infty}+\|v(\cdot, t)\|_{L^\infty})\|w(\cdot, t)\|_{L^\infty} + \|r(\cdot, t)\|_{L^\infty} \\
&\le 2\lambda K_0\left(\|w_0\|_{L^\infty} + \frac{r_*}{\mu}\right) + r_* \\
&\leq \frac{2\lambda}{\mu} r_*K_0 + 2\lambda\|w_0\|_{L^\infty}K_0 + r_*K_0\\
&\leq  c_1K_0(r_*+1),
\end{aligned}
$$
where the constant $c_1 > 0$ is independent of $t$ and $r_*$. 
To simplify notation, we denote
$$N := (r_*+1)K_0.$$
Applying the standard local-in-time parabolic $L^p$ estimate on the enlarged cylinder $\Omega\times(t-1,t+1)$, there exists a constant $c_3>0$ independent of $t$ and $r_*$ such that
\begin{equation}\label{w21p-w}
\|w\|_{W_{n+3}^{2,1}(\Omega\times(t,t+1))}
\leq c_2\Big(
\|w\|_{L^\infty(\Omega\times(t-1,t+1))}
+\|f\|_{L^\infty(\Omega\times(t-1,t+1))}
\Big)
\leq c_3N,\qquad t\geq1.
\end{equation}
Since $n+3>n+2$, the parabolic Sobolev embedding theorem gives
\[
W_{n+3}^{2,1}(\Omega\times(t,t+1))
\hookrightarrow
C^{1+\theta,\frac{1+\theta}{2}}
(\overline{\Omega}\times[t,t+1])
\hookrightarrow
L^\infty([t,t+1];W^{1,\infty}(\Omega))
\]
for any $\theta\in\left(0,\frac{1}{n+3}\right)$. Hence, we obtain
\begin{equation}\label{conclu_w_new}
\|w(\cdot, s)\|_{W^{1,\infty}} \le c_4 N, \quad \forall s \in [t, t+1], \ t \ge 1,
\end{equation}
which implies \eqref{conclu_w}. Furthermore, \eqref{w21p-w} gives
\begin{equation}\label{d2w_n+3}
\int_t^{t+1} \|\Delta w(\cdot, s)\|_{L^{n+3}}^{n+3} ds \le c_5 N^{n+3}, \quad \forall t \ge 1.
\end{equation} 

\textbf{Step 2: Uniform boundedness of $\|\nabla u(\cdot, t)\|_{L^\infty}$.}
Multiplying the first equation of \eqref{system} by $-\Delta u$ and integrating the results by parts over $\Omega$, we obtain
\begin{equation}\label{grau-1}
 \frac{1}{2} \frac{d}{dt} \int_\Omega |\nabla u|^2 + \int_\Omega |\Delta u|^2 = \chi_1 \int_\Omega \Delta u \nabla \cdot (u\nabla w)=\chi_1 \int_\Omega \Delta u (\nabla u \cdot \nabla w + u\Delta w). 
\end{equation}
Applying the H\"{o}lder inequality, Young's inequality along with \eqref{uvb*} and \eqref{conclu_w}, we have
$$ 
\begin{aligned}
\chi_1 \int_\Omega \Delta u (\nabla u \cdot \nabla w + u\Delta w) 
&\le \chi_1 \|\nabla w\|_{L^\infty} \int_\Omega |\Delta u||\nabla u| + \chi_1 \|u\|_{L^\infty}\int_\Omega |\Delta w||\Delta u| \\
&\le \frac{1}{2} \int_\Omega |\Delta u|^2 + c_6 N^2 \int_\Omega |\nabla u|^2 + c_7 K^2_0 \|\Delta w\|_{L^2}^2,
\end{aligned}
$$
which, substituted into \eqref{grau-1}, gives
\begin{equation}\label{grad_u_energy_1}
\frac{d}{dt} \int_\Omega |\nabla u|^2 + \int_\Omega |\Delta u|^2 \le  2c_6 N^2 \int_\Omega |\nabla u|^2 + 2c_7 K_0^2 \|\Delta w\|_{L^2}^2.
\end{equation}
Applying the Gagliardo-Nirenberg inequality and using \eqref{uvb*}, we obtain
\begin{equation}\label{GN_ineq}
\begin{aligned}
(2c_6 N^2 + 1) \int_\Omega |\nabla u|^2 &\le c_{8}(N^2+1) \left( \|\Delta u\|_{L^2} \|u\|_{L^2} + \|u\|_{L^2}^2 \right) \\
&\le c_{8}(N^2+1) K_0\left( |\Omega|^{\frac{1}{2}}\|\Delta u\|_{L^2}  + |\Omega|K_0\right) \\
&\le \int_\Omega |\Delta u|^2 + c_{9}K_0^2(N^2+1)^2.
\end{aligned}
\end{equation}
Substituting \eqref{GN_ineq} into \eqref{grad_u_energy_1} implies
\begin{equation}\label{nul2-g}
\frac{d}{dt}\int_\Omega|\nabla u|^2+\int_\Omega|\nabla u|^2\leq c_{10}K_0^2(N^2+1)^2+2c_7K_0^2\|\Delta w\|_{L^2}^2,\quad \forall\,t>1.
\end{equation}
Moreover, by the H\"{o}lder inequality and \eqref{d2w_n+3}, we have
\begin{equation}\label{d2w_2}
\begin{aligned}
\int_t^{t+1}\|\Delta w(\cdot,s)\|_{L^2}^2\,ds \leq c_{11}\left(\int_t^{t+1}\|\Delta w(\cdot,s)\|_{L^{n+3}}^{n+3}\,ds\right)^{\frac{2}{n+3}} \leq c_{12}N^2,\qquad \forall\,t\geq 1. 
\end{aligned}
\end{equation}
Therefore, solving the differential inequality \eqref{nul2-g} and using \eqref{d2w_2}, we obtain
\begin{equation}\label{grad_u_L2}
\|\nabla u(\cdot,t)\|_{L^2}\leq c_{13}K_0(N^2+1),\qquad \forall\,t\geq1.
\end{equation}
On the other hand, applying Duhamel's formula to the $u$-equation in \eqref{system} over the interval $[t-1, t]$ for any $t\geq 2$ gives
$$ u(\cdot,t)=e^\Delta u(\cdot,t-1)-\chi_1\int_{t-1}^t e^{(t-s)\Delta}\nabla\cdot (u\nabla w)(\cdot, s)ds.$$
Using the well-known semigroup estimate (see e.g., \cite[Lemma 2.1]{CaoXinru-DCDSA-2015}) gives
\begin{equation}\label{Duhamel_expanded}
\begin{aligned}
\|\nabla u(\cdot,t)\|_{L^\infty}
&\leq \|\nabla e^\Delta u(\cdot,t-1)\|_{L^\infty}+\chi_1\int_{t-1}^t\big\|\nabla e^{(t-s)\Delta}\nabla\cdot(u\nabla w)(\cdot,s)\big\|_{L^\infty}\,ds\\
&\leq c_{14}K_0+c_{15}\int_{t-1}^t(t-s)^{-\frac12-\frac{n}{2(n+3)}}\|\nabla\cdot(u\nabla w)(\cdot,s)\|_{L^{n+3}}\,ds\\
&\leq c_{14}K_0+c_{16}N\int_{t-1}^t(t-s)^{-\frac12-\frac{n}{2(n+3)}}\|\nabla u(\cdot,s)\|_{L^{n+3}}\,ds\\
&\quad+c_{17}K_0\int_{t-1}^t(t-s)^{-\frac12-\frac{n}{2(n+3)}}\|\Delta w(\cdot,s)\|_{L^{n+3}}\,ds.
\end{aligned}
\end{equation}
Moreover, since $\left(\frac12+\frac{n}{2(n+3)}\right)\frac{n+3}{n+2}
=\frac{2n+3}{2n+4}<1,$ applying the H\"{o}lder inequality and \eqref{d2w_n+3} yields
\begin{equation}\label{Dwib}
\begin{aligned}
&\int_{t-1}^{t}(t-s)^{-\frac12-\frac{n}{2(n+3)}}\|\Delta w(\cdot,s)\|_{L^{n+3}}\,ds\\
&\leq\left(\int_{t-1}^{t}(t-s)^{-\left(\frac12+\frac{n}{2(n+3)}\right)\frac{n+3}{n+2}}\,ds\right)^{\frac{n+2}{n+3}}\left(\int_{t-1}^{t}\|\Delta w(\cdot,s)\|_{L^{n+3}}^{n+3}\,ds\right)^{\frac1{n+3}}\\
&\leq c_{18}N.
\end{aligned}
\end{equation}
Substituting \eqref{Dwib} into \eqref{Duhamel_expanded}, we obtain
\begin{equation}\label{Duhamel_simplified}
\|\nabla u(\cdot,t)\|_{L^\infty}\leq c_{19}K_0(N+1)+c_{20}N\int_{t-1}^t(t-s)^{-\frac12-\frac{n}{2(n+3)}}\|\nabla u(\cdot,s)\|_{L^{n+3}}\,ds.
\end{equation}
Let $\theta_1:=1-\frac{2}{n+3}\in(0,1)$, then by the interpolation inequality and \eqref{grad_u_L2}, it holds that
\begin{equation}\label{interp_Lp}
\begin{aligned}
\|\nabla u(\cdot,s)\|_{L^{n+3}}
\leq c_{21}\|\nabla u(\cdot,s)\|_{L^\infty}^{\theta_1}\|\nabla u(\cdot,s)\|_{L^2}^{1-\theta_1}\leq c_{22}\big[K_0(N^2+1)\big]^{1-\theta_1}\|\nabla u(\cdot,s)\|_{L^\infty}^{\theta_1}.
\end{aligned}
\end{equation}
We substitute \eqref{Duhamel_simplified} into \eqref{interp_Lp} to get
\begin{equation}\label{nuinfty_t-1}
\begin{aligned}
\|\nabla u(\cdot,t)\|_{L^\infty}
\leq c_{23}K_0(N+1)+c_{23}N\big[K_0(N^2+1)\big]^{1-\theta_1}\int_{t-1}^t(t-s)^{-\frac12-\frac{n}{2(n+3)}}\|\nabla u(\cdot,s)\|_{L^\infty}^{\theta_1}\,ds.
\end{aligned}    
\end{equation}
To estimate \eqref{nuinfty_t-1}, we set 
$$M_T:=\sup_{2<t<T}\|\nabla u(\cdot,t)\|_{L^\infty}$$
for any $T>3$.  We emphasize that all constants appearing below are independent of $T$. Now, we discuss two cases.\\
Case 1: $3\leq t<T$. In this case,  we have
\begin{equation*}\label{3-T}
\int_{t-1}^t(t-s)^{-\frac12-\frac{n}{2(n+3)}}\|\nabla u(\cdot,s)\|_{L^\infty}^{\theta_1}\,ds\leq M_T^{\theta_1}\int_{t-1}^t(t-s)^{-\frac12-\frac{n}{2(n+3)}}\,ds\leq c_{24}M_T^{\theta_1},   
\end{equation*}
which, substituted into \eqref{nuinfty_t-1}, gives
\begin{equation}\label{0709}
\|\nabla u(\cdot,t)\|_{L^\infty}
\leq c_{25}K_0(N+1)
+c_{25}N\big[K_0(N^2+1)\big]^{1-\theta_1}
M_T^{\theta_1},
\qquad 3\leq t<T.
\end{equation}
Case 2: $2\leq t<3$. In this case, noting $\frac{1}{2}+\frac{n}{2(n+3)}<1,$  we have
\begin{equation}\label{2-3*}
\begin{aligned}
&\int_{t-1}^t(t-s)^{-\frac12-\frac{n}{2(n+3)}}
\|\nabla u(\cdot,s)\|_{L^\infty}^{\theta_1}\,ds\\
&=\int_{t-1}^2(t-s)^{-\frac12-\frac{n}{2(n+3)}} (s-1)^{-\frac{\theta_1}{2}}[(s-1)^{\frac{1}{2}}\|\nabla u(\cdot,s)\|_{L^\infty}]^{\theta_1}\,ds\\
&\ \ \ \ +\int_2^t(t-s)^{-\frac12-\frac{n}{2(n+3)}}
\|\nabla u(\cdot,s)\|_{L^\infty}^{\theta_1}\,ds\\
&\leq H^{\theta_1} \int_{1}^2(2-s)^{-\frac12-\frac{n}{2(n+3)}}(s-1)^{-\frac{\theta_1}{2}}\,ds
+M_T^{\theta_1}\int_2^t(t-s)^{-\frac12-\frac{n}{2(n+3)}}\,ds\\
&\leq c_{26}\big(H^{\theta_1}+M_T^{\theta_1}\big),
\end{aligned}    
\end{equation}
where
\begin{equation}\label{79h}
H:=\sup_{1<\tau\leq 2}(\tau-1)^{\frac12}\|\nabla u(\cdot,\tau)\|_{L^\infty}\geq 0.
\end{equation}
We substitute  \eqref{2-3*} into \eqref{nuinfty_t-1} to obtain
\begin{equation}\label{0709-1}
\|\nabla u(\cdot,t)\|_{L^\infty}
\leq c_{27}K_0(N+1)
+c_{27}N\big[K_0(N^2+1)\big]^{1-\theta_1}
(H^{\theta_1}+M_T^{\theta_1}),
\qquad 2\leq t<3.
\end{equation}
Combining \eqref{0709} with \eqref{0709-1} implies
\begin{equation*}
\|\nabla u(\cdot,t)\|_{L^\infty}
\leq c_{28}K_0(N+1)
+c_{28}N\big[K_0(N^2+1)\big]^{1-\theta_1}
(H^{\theta_1}+M_T^{\theta_1}), \qquad 2\leq t<T.
\end{equation*}
Taking the supremum over $t\in (2,T)$, we have
\begin{equation}\label{0709-2}
M_T
\leq c_{28}K_0(N+1)
+c_{28}N\big[K_0(N^2+1)\big]^{1-\theta_1}
(H^{\theta_1}+M_T^{\theta_1}).
\end{equation}
We next estimate the term $H$ defined in \eqref{79h}. For any $1<\tau\leq 2$, applying Duhamel's formula to the $u$-equation in \eqref{system} on $(1,\tau)$ and arguing as in \eqref{Duhamel_expanded}--\eqref{nuinfty_t-1}, we have
\begin{equation}\label{nui_t=1}
\|\nabla u(\cdot,\tau)\|_{L^\infty}
\leq c_{29}K_0(N+1)(\tau-1)^{-\frac12}+c_{29}N\big[K_0(N^2+1)\big]^{1-\theta_1}\int_1^\tau(\tau-s)^{-\frac12-\frac{n}{2(n+3)}}\|\nabla u(\cdot,s)\|_{L^\infty}^{\theta_1}\,ds.    
\end{equation}
By the change of variables $\eta=\frac{s-1}{\tau-1}$, and noting that
$\frac12+\frac{n}{2(n+3)}<1$ and $\frac{\theta_1}{2}<1$, one derives
\begin{equation}\label{H1}
\begin{aligned}
\int_1^\tau(\tau-s)^{-\frac12-\frac{n}{2(n+3)}}\|\nabla u(\cdot,s)\|_{L^\infty}^{\theta_1}\,ds
&\leq H^{\theta_1}\int_1^\tau(\tau-s)^{-\frac12-\frac{n}{2(n+3)}}(s-1)^{-\frac{\theta_1}{2}}\,ds\\
&=H^{\theta_1}(\tau-1)^{\frac12-\frac{n}{2(n+3)}-\frac{\theta_1}{2}}\int_0^1(1-\eta)^{-\frac12-\frac{n}{2(n+3)}}\eta^{-\frac{\theta_1}{2}}\,d\eta\\
&\leq c_{30}H^{\theta_1}(\tau-1)^{\frac12-\frac{n}{2(n+3)}-\frac{\theta_1}{2}}.    
\end{aligned}    
\end{equation}
Substituting \eqref{H1} into \eqref{nui_t=1} and using $0<\tau-1\leq 1$ yield
\begin{equation*}
\begin{split}
(\tau-1)^{\frac12}\|\nabla u(\cdot,\tau)\|_{L^\infty} &\leq c_{31}K_0(N+1)+c_{31}N\big[K_0(N^2+1)\big]^{1-\theta_1}H^{\theta_1}(\tau-1)^{1-\frac{n}{2(n+3)}-\frac{\theta_1}{2}}\\
&\leq c_{31}K_0(N+1)+c_{31}N\big[K_0(N^2+1)\big]^{1-\theta_1}H^{\theta_1}.
\end{split}
\end{equation*}
Taking the supremum over $\tau\in (1,2)$ and using Young's inequality, one derives
\begin{equation}\label{H2}
H\leq c_{32}K_0\big(N^{\frac{n+7}{2}}+1\big).
\end{equation}
We substitute \eqref{H2} into \eqref{0709-2} to get
\begin{equation*}
\begin{split}
M_T
&\leq c_{28}K_0(N+1)
+c_{33}N\big[K_0(N^2+1)\big]^{1-\theta_1}
[K_0\big(N^{\frac{n+7}{2}}+1\big)]^{\theta_1}+c_{33}N\big[K_0(N^2+1)\big]^{1-\theta_1}M_T^{\theta_1}\\
&\leq c_{28}K_0(N+1) + c_{34}K_0\big(N^{\frac{n+7}{2}}+1\big)+c_{33}N\big[K_0(N^2+1)\big]^{1-\theta_1}M_T^{\theta_1},
\end{split}
\end{equation*}
which, along with Young's inequality, implies
\[
M_T\leq c_{35}K_0\big(N^{\frac{n+7}{2}}+1\big).
\]
Since the above estimate is independent of $T$, letting $T\to\infty$ yields
\begin{equation}\label{grad_u_Linfty}
\|\nabla u(\cdot,t)\|_{L^\infty}
\leq c_{35}K_0\big(N^{\frac{n+7}{2}}+1\big),
\qquad \forall\,t\geq2.
\end{equation}

\textbf{Step 3: H\"{o}lder estimates for $u$ and $v$.}
We rewrite the first equation of \eqref{system} as
$$ u_t = \nabla \cdot A_1(x, t, u, \nabla u), $$
where $A_1(x, t, u, \nabla u) = \nabla u - \chi_1 u \nabla w$. By \eqref{uvb*} and  \eqref{conclu_w_new}, we obtain 
$$\|u\|_{L^\infty} \le K_0,\ \  \|\nabla w\|_{L^\infty(\Omega \times [t, t+1])} \le c_4 N.$$
Then it holds that
\begin{equation}\label{A1_cond1}
\begin{aligned}
A_1(x, t, u, \nabla u) \cdot \nabla u &= (\nabla u - \chi_1 u \nabla w) \cdot \nabla u \\
&\ge \frac{1}{2}|\nabla u|^2 - \frac{\chi_1^2}{2} u^2 |\nabla w|^2 \\
&\ge \frac{1}{2}|\nabla u|^2 - c_{36}K_0^2 N^2,
\end{aligned}
\end{equation}
and
\begin{equation}\label{A1_cond2}
|A_1(x, t, u, \nabla u)| \le |\nabla u| + \chi_1 |u| |\nabla w| \le |\nabla u| + c_{37}K_0N.
\end{equation}
With \eqref{A1_cond1} and \eqref{A1_cond2}, we apply the H\"{o}lder regularity for quasilinear parabolic equations \cite[Theorem 1.3 and Remark 1.4]{PV-JDE-1993} to obtain 
\begin{equation}\label{holder_u}
\|u\|_{C^{\alpha_5, \frac{\alpha_5}{2}}(\bar{\Omega}\times[t, t+1])} \le c_{38}K_0(N+1) \quad \text{for all } t \ge 1.
\end{equation}
Now we rewrite the second equation of \eqref{system} as
$$ v_t = \nabla \cdot A_2(x, t, v, \nabla v),$$
where $A_2(x, t, v, \nabla v) = \nabla v - \chi_2 v \nabla u$. 
Noting \eqref{uvb*} and \eqref{grad_u_Linfty}, adopting the same arguments as proving \eqref{holder_u} gives 
\begin{equation}\label{holder_v}
\|v\|_{C^{\alpha_6, \frac{\alpha_6}{2}}(\bar{\Omega}\times[t, t+1])} \le c_{39}K_0\big(1 + \|\nabla u\|_{L^\infty}\big) \le c_{40}K_0^2(N^{\frac{n+7}{2}}+1) \quad \text{for all } t \ge 2.
\end{equation}
Since $N=(r_*+1)K_0$ and $K_0\geq1$, we have $K_0^2\left(N^{\frac{n+7}{2}}+1\right)\leq2K_0^{\frac{n+11}{2}}(r_*+1)^{\frac{n+7}{2}}.$ Hence, \eqref{conclu_uv} follows directly from \eqref{holder_u} and
\eqref{holder_v} by letting
$\tilde{\alpha}:=\min\{\alpha_5,\alpha_6,\alpha_0\}$.
\end{proof}

\bigbreak
\noindent \textbf{Acknowledgment}.
The research of H.Y. Jin was supported by the NSF of China (No. 12371203), Guangdong Major Project of Basic Research (2026B0303000003).

\small{
}

\end{document}